\documentclass[10pt,reqno]{amsart}
\usepackage{hyperref}
\usepackage{latexsym,amsmath}
\usepackage{enumerate}
\usepackage{amsfonts}
\usepackage{amssymb}
\usepackage{latexsym}
\usepackage{fixmath}
\usepackage{microtype}

\usepackage{mathtools} 

 \usepackage[usenames,dvipsnames]{color}

\usepackage{graphicx}

\usepackage[T1]{fontenc}
\usepackage{fourier}
\usepackage{bbm}
\usepackage{color}
\usepackage{fullpage}

\numberwithin{equation}{section}

\newcommand{\jtd}[1]{{#1}}
\newcommand{\fixxx}[1]{{\color{black}#1}}
\newcommand{\fixx}[1]{{\color{black}#1}}

\newcommand\nup{N}

\newcommand\vB{\vec B}
\newcommand\vQ{\vec Q}
\renewcommand{\vec}[1]{\boldsymbol{#1}}

\renewcommand{\subset}{\subseteq}

\newcommand\disteq{\,\sim\,}

\newcommand\vC{\vec C}
\newcommand\vD{\vec D}

\newcommand\vX{\vec X}
\newcommand\vw{\vec w}

\newcommand\row{\vec \zeta}
\newcommand\col{\vec \xi}

\newcommand\vk{\vec k}
\newcommand\vd{\vec d}
\newcommand\vy{\vec y}

\newcommand\vx{\vec x}

\newcommand\vi{\vec i}

\newcommand\THETA{\vec\theta}

\newcommand\CPC{Combinatorics, Probability and Computing}

\newcommand{\GG}{\mathbb G}

\newcommand\ALPHA{\vec\alpha}
\newcommand\MU{\vec\mu}

\newcommand\vm{{\vec m}}

\newcommand\CHI{{\vec\chi}}
\newcommand\DELTA{{\vec\Delta}}
\newcommand\GAMMA{{\vec\gamma}}

\newcommand\nix{\,\cdot\,}

\newcommand\vA{\vec A}

\newcommand\G{\vec G}

\newcommand\bemph[1]{{\bf\em #1}}

\newcommand\SIGMA{\vec\sigma}

\newcommand\aco[1]{\textcolor{black}{#1}}

\newcommand\fA{\mathfrak{A}}
\newcommand\fE{\mathfrak{E}}
\newcommand\fF{\mathfrak{F}}
\newcommand\cA{\mathcal{A}}
\newcommand\cB{\mathcal{B}}
\newcommand\cC{\mathcal{C}}
\newcommand\cD{\mathcal{D}}
\newcommand\cF{\mathcal{F}}

\newcommand\cE{\mathcal{E}}
\newcommand\cU{\mathcal{U}}
\newcommand\cN{\mathcal{N}}
\newcommand\cQ{\mathcal{Q}}
\newcommand\cH{\mathcal{H}}
\newcommand\cS{\mathcal{S}}
\newcommand\cT{\mathcal{T}}

\newcommand\cK{\mathcal{K}}

\newcommand\cL{\mathcal{L}}
\newcommand\cM{\mathcal{M}}

\newcommand\cP{\mathcal{P}}
\newcommand\cX{\mathcal{X}}
\newcommand\cY{\mathcal{Y}}
\newcommand\cV{\mathcal{V}}
\newcommand\cW{\mathcal{W}}

\def\cR{{\mathcal R}}
\def\cC{{\mathcal C}}
\def\cE{{\mathcal E}}

\newcommand\eps{\varepsilon}

\newcommand\ZZ{\mathbb{Z}}
\newcommand\FF{\mathbb{F}}
\newcommand\NN{\mathbb{N}}

\newcommand\Var{\mathrm{Var}}
\newcommand\Erw{\mathbb{E}}
\newcommand{\vecone}{\vec{1}}

\newcommand{\Po}{{\rm Po}}
\newcommand{\Bin}{{\rm Bin}}

\newcommand\dTV{d_{\mathrm{TV}}}
\newcommand{\bink}[2] {{\binom{#1}{#2}}}
\newcommand\bc[1]{\left({#1}\right)}
\newcommand\cbc[1]{\left\{{#1}\right\}}

\newcommand\brk[1]{\left\lbrack{#1}\right\rbrack}

\newcommand\abs[1]{\left|{#1}\right|}

\newcommand\RR{\mathbb{R}}

\newcommand{\Whp}{A.a.s.}
\newcommand{\whp}{a.a.s.}

\newcommand{\Erdos}{Erd\H os}

\newcommand{\Komlos}{Koml\'os}
\newcommand{\Luczak}{\L uczak}
\newcommand{\Rucinski}{Ruci\'nski}
\newcommand{\Mezard}{M\'ezard}

\newcommand\pr{\mathbb{P}} 
 
\newcommand\Lem{Lemma}
\newcommand\Prop{Proposition}
\newcommand\Thm{Theorem}

\newcommand\Cor{Corollary}
\newcommand\Sec{Section}
\newcommand\Chap{Chapter}

\newtheorem{definition}{Definition}[section]
\newtheorem{claim}[definition]{Claim}
\newtheorem{example}[definition]{Example}
\newtheorem{remark}[definition]{Remark}
\newtheorem{theorem}[definition]{Theorem}
\newtheorem{lemma}[definition]{Lemma}
\newtheorem{proposition}[definition]{Proposition}
\newtheorem{corollary}[definition]{Corollary}

\newtheorem{fact}[definition]{Fact}

\DeclareMathOperator{\nul}{nul}
\DeclareMathOperator{\rank}{rk}

\newcommand{\rk}{\rank}
\newcommand{\supp}{{\text{supp}}}

\newcommand\A{\vA}

\def\B{{\mathcal B}}

\def\E{{\mathcal E}}

\def\po{{\bf Po}}
\def\ex{{\mathbb E}}
\def\pr{{\mathbb P}}

\def\bfd{{\vec d}}

\def\bfk{\vk}

\def\bfm{{\vec m}}
\def\bfn{{\vec n}}

\def\bfG{{\G}}

\def\bftheta{{\pmb{\theta}}}

\def\bbN{{\mathbb N}}

\def\cH{{\mathcal H}}

\newcommand\mr[1]{\aco{#1}}

\newcommand{\remove}[1]{}
\newcommand\eqn[1]{(\ref{#1})}

\newcommand{\be}{\begin{equation}}
\newcommand{\bel}[1]{\begin{equation}\lab{#1}\ }
\newcommand{\ee}{\end{equation}}
\newcommand{\bea}{\begin{eqnarray}}
\newcommand{\eea}{\end{eqnarray}}
\newcommand{\bean}{\begin{eqnarray*}}
\newcommand{\eean}{\end{eqnarray*}}

\newcommand{\vn}{{\vec n}}
\newcommand{\vM}{\vec M}
\newcommand{\vN}{\vec N}

\begin{document}


\title{The rank of sparse random matrices}

\author{Amin Coja-Oghlan, Alperen A.~Erg\"ur, Pu Gao, Samuel Hetterich, Maurice Rolvien}
\thanks{Coja-Oghlan supported by DFG CO 646/3 and 646/4.
Erg\"ur's research was partially supported by Einstein Foundation, Berlin and NSF CCF 2110075.
Gao's research is supported by ARC DE170100716 and ARC DP160100835.
This submission combines the preprints {\tt arXiv:1810.07390} and {\tt arXiv:1906.05757}.
An extended abstract of this work appeared in the proceedings of the 31st ACM-SIAM Symposium on Discrete Algorithms (2020) 579--591.}

\address{Amin Coja-Oghlan, {\tt amin.coja-oghlan@tu-dortmund.de}, TU Dortmund, Faculty of Computer Science, 12 Otto Hahn St, Dortmund 44227, Germany.}

\address{Alperen A.~Erg\"ur, 
{\tt alperen.ergur@utsa.edu}, 
The University of Texas at San Antonio, TX, USA.}

\address{Pu Gao, {\tt p3gao@uwaterloo.ca}, 
  Department of Combinatorics and Optimization
    University of Waterloo, Canada.}

\address{Samuel Hetterich, {\tt hetterich@math.uni-frankfurt.de}, Goethe University, Mathematics Institute, 10 Robert Mayer St, Frankfurt 60325, Germany.}

\address{Maurice Rolvien, {\tt maurice.rolvien@tu-dortmund.de}, TU Dortmund, Faculty of Computer Science, 12 Otto Hahn St, Dortmund 44227, Germany.}

\begin{abstract}
	We determine the \aco{asymptotic normalized rank} of a random matrix $\vA$ over an arbitrary field with prescribed numbers of non-zero entries in each row and column.
As an application we obtain a formula for the rate of low-density parity check codes.
This formula vindicates a conjecture of Lelarge (2013).
The proofs are based on coupling arguments and a novel random perturbation, applicable to any matrix, that diminishes the number of short linear relations.
\hfill
\end{abstract}

\maketitle

\section{Introduction}\label{Sec_intro}

\subsection{Background and motivation}
The theory of random matrices, which commenced with the nuclear physics-inspired work of Wigner in the 1950s~\cite{Wigner}, has been one of the great success stories at the junction of probability, mathematical physics and combinatorics.
Nevertheless, quite a few basic questions remain open to this day.
For instance, while dense random matrices such as the Gaussian Orthogonal Ensemble are reasonably well understood (e.g., \cite{Mehta}), far less is known about sparse random matrices where the expected number of non-zero entries per row or column is bounded.
Yet over the last two or three decades such sparse random matrices, with entries from finite or infinite fields, have emerged to play a pivotal role in several exciting applications.
Modern error-correcting codes are a case in point.
For instance, the codebook of a low-density parity check code (`ldpc code'), a class of codes that has been at the centre of tremendous recent developments in coding theory~\cite{GMU,KYMP,RichardsonUrbanke}, comprises the kernel of a sparse random matrix over a finite field drawn from a carefully tailored distribution.
In addition, sparse random matrices 
occur in randomised constructions of Ramanujan graphs~\cite{AlonSpencer,Bordenave,Friedman},
 statistical inference~\cite{redemption}, the analysis of algorithms~\cite{Dietzfelbinger} and the theory of random constraint satisfaction problems \cite{AchlioptasMolloy,Ibrahimi}.

Among the fundamental questions about such random matrices that have remained open, perhaps the most conspicuous one concerns the rank.
Although this parameter was already studied in early contributions~\cite{Balakin1,Balakin2,Kovalenko}, there has been no comprehensive rank formula for sparse random matrices.
The present paper furnishes one.
To be precise, we will determine the asymptotic rank of a sparse random matrix with prescribed numbers of non-zero entries in the rows and columns.
Among other applications, important classes of ldpc codes are based on precisely such random matrices as a diligent choice of the degrees greatly boosts the code's performance~\cite{RichardsonUrbanke}.
Moreover, the rank is linearly related to the rate of the code, arguably the code's most basic parameter.

Lelarge~\cite{Lelarge} noticed that an upper bound on the rank of a sparse random matrix  can be derived from the matching number of random bipartite graphs, which was determined by Bordenave, Lelarge and Salez~\cite{BLS2}.
Lelarge went on to conjecture that this bound be tight for sparse random matrices over the binary field $\FF_2$.
We prove this conjecture.
In fact, we prove a much stronger result.
Namely, we show that Lelarge's conjectured formula holds for sparse random matrices over any field, finite or infinite, regardless the distribution of the non-zero matrix entries.
Thus, the rank is governed  by the {\em location} of the non-zero entries rather than the distribution of the matrix entries.

The proof of the rank formula evinces an interesting connection to statistical physics.
Indeed, Lelarge already observed that a sophisticated but mathematically non-rigorous physics approach  called the `cavity method' renders a wrong prediction as to the rank for certain degree distributions.%
	\footnote{The derivation of this erroneous prediction was posed as an exercise in~\cite[\Chap~19]{MM}.}
This discrepancy merits attention because the cavity method has been brought to bear on a panoply of theoretical as well as real-world problems, ranging from spin glasses to machine learning~\cite{LF}.
We manage to shed light on the issue.
Specifically, the `replica symmetric' version of the cavity method predicts that the rank of a random matrix over a finite field can be expressed analytically as the maximum of a variational problem.
A priori, this variational problem asks to optimise a functional called the Bethe free entropy over an infinite-dimensional space of probability measures.
Such optimisation problems have been tackled in the physics literature numerically by means of a heuristic called population dynamics.
For the rank problem this was carried out by Alamino and Saad~\cite{AlaminoSaad}.
\aco{But thanks to the algebraic nature of the problem we can show that the rank actually comes out as the solution to a variational problem on a restricted domain.
	We are thus left with a dramatically simplified variational problem, which ultimately boils down to a humble one-dimensional optimisation task.
We will see that the optimal solution to this one-dimensional problem does indeed yield the rank (over any field).
Furthermore, the solution can be lifted to a solution to the original infinite-dimensional problem.
As an aside, we do not know if the original infinite-dimensional variational problem may possess spurious maximisers that boost its value beyond the optimal value of the restricted version, thereby spoiling the accuracy of the original physics formula.
We will return to this question, and to the physics slant on the problem, in \Sec~\ref{Sec_cavity}.
In any case, for certain degree distributions the maximum values that we obtain by way of the restricted variational problem actually exceed those that surfaced in the experiments from~\cite{AlaminoSaad} or the heuristic derivations from~\cite{MM} for the unrestricted formula;
hence the discrepancy between the physics predictions and mathematical reality.
}

Apart from remedying the discrepancy, we prove the rank formula by  effectively turning the physicists' cavity calculations into a rigorous mathematical argument.
The crucial tool that makes this possible is a novel perturbation, applicable to any matrix, that diminishes the number of short linear relations (see \Prop~\ref{Prop_Alp} below).
We expect that this perturbation will find future applications.
Let us proceed to introduce the random matrix model and state the main results.
A discussion of related work and a detailed comparison with the physics work follow in \Sec~\ref{Sec_overview}, once we have the necessary notation in place.

\subsection{The rank formula}\label{Sec_rankFormula}
Let $\FF$ be a field equipped with a $\sigma$-algebra that turns $\FF $ into a standard Borel space and let  $\chi:[0,1]^2\to\FF^*=\FF\setminus\cbc 0$ be a measurable map.
Let $(\row_i,\col_i)_{i\geq1}$ be mutually independent uniformly distributed $[0,1]$-valued random variables.
Moreover, let $\vd,\vk\geq0$ be integer-valued random variables such that $0<\Erw[\vd^r]+\Erw[\vk^r]<\infty$ for a real $r>2$ and set $d=\Erw[\vd]$, $k=\Erw[\vk]$.
Let $n>0$ be an integer divisible by the greatest common divisor of the support of $\vk$ and let $\vm\disteq\Po(dn/k)$ be independent of the $\row_i,\col_i$.
Further, let $(\vd_i,\vk_i)_{i\geq1}$ be copies of $\vd,\vk$, mutually independent and independent of $\vm,\row_i,\col_i$.
Given
\begin{equation}\label{eqWellDef1}
\sum_{i=1}^n\vd_i=\sum_{i=1}^{\vm}\vk_i,
\end{equation}
draw a simple bipartite graph $\G$ comprising a set $\{a_1,\ldots,a_{\vm}\}$ of {\em check nodes} and a set $\{x_1,\ldots,x_n\}$ of 
{\em variable nodes} such that the degree of $a_i$ equals $\vk_i$ and the degree of $x_j$ equals $\vd_j$ for all $i,j$ uniformly at random.
Then let $\A$ be the $\vm\times n$-matrix with entries
\begin{align*}
\A_{ij}&=\vecone\{a_ix_j\in E(\G)\}\cdot\chi_{\row_i,\col_j}.
\end{align*}	
Thus, the $i$-th row of $\A$ features precisely $\vk_i$ non-zero entries and the $j$-th column contains precisely $\vd_j$ non-zero entries.
Moreover, the non-zero entries of $\vA$ are drawn in the vein of an exchangeable array by evaluating the function $\chi$ at a random poisition $(\row_i,\col_j)$.
Routine arguments show that $\A$ is well-defined for large enough $n$, i.e., \eqref{eqWellDef1} is satisfied and there exists a simple $\G$ with the desired degrees with positive probability; see \Prop~\ref{Lemma_welldef} below.
We call $\G$ the {\em Tanner graph} of $\A$.
\mr{Also recall that the rank $\rank A$ of the matrix $A$ is defined as the maximal number of linear independent rows (or columns).
In addition, $\nul A$ is the dimension of the kernel of $A$ and the sum $\rank A+\nul A$ equals the number of columns of $A$. }

The following theorem, the main result of the paper, provides an asymptotic formula for the  rank of $\A$.
Let $D(x)$ and $K(x)$ denote the probability generating functions of $\vd$ and $\vk$, respectively.
Since $\Erw[\vd^2]+\Erw[\vk^2]<\infty$, the functions $D(x),K(x)$ are continuously differentiable on the unit interval.
Therefore, the function
\begin{align}\label{eqBFE}
\Phi:[0,1]&\to\RR,&\alpha&\mapsto D\left(1-K'(\alpha)/k\right)-\frac{d}{k}\bc{1-K(\alpha)-(1-\alpha)K'(\alpha)}.
\end{align}
is continuous.

\begin{theorem}\label{thm:rank}
For any $\vd,\vk$ we have, uniformly for all $\chi$,
	\begin{align}\label{eqthm:rank}
	\lim_{n\to\infty}\frac{\rank(\A)}n&=1-\max_{\alpha\in[0,1]}\Phi(\alpha)&&\mbox{in probability.}
	\end{align}
\end{theorem}

\aco{Perhaps surprisingly, the r.h.s.\ of \eqref{eqthm:rank} depends only on the degree distributions $\vd,\vk$ but not in any way on the field $\FF$ or the choice of non-zero entries (within the aforementioned model).
Furthermore, let us emphasise that the function $\Phi$, being continuous on the unit interval, is guaranteed to attain a maximum.
However, this maximum need not be unique, and non-uniqueness of the maximiser may have interesting combinatorial repercussions~\cite{CCKLRR}.
}

\aco{A second point that may seem surprising at first glance is that the rank converges to any non-random value at all, as provided by \eqref{eqthm:rank}.
A heuristic explanation can be given on grounds of physics reasoning.
Indeed, the nullity of $\A$ (dimension of the kernel) corresponds to the logarithm of the partition function of a natural Boltzmann distribution, namely the uniform distribution on the kernel of $\A$.
Commonly the normalised logarithm of such a partition functions (known as the ``free entropy'' in physics jargon) converges to a constant for random systems that are ``self-averaging''.
Here ``self-averaging'' means that a small perturbation to the system, i.e., the matrix in our case, cannot cause disproportionate tremors in logarithm of the partition function.
In the random matrix model that we consider here the self-averaging condition is clearly satisfied because changing a single matrix entry can at most alter the nullity by one.
Therefore, the Azuma--Hoeffding inequality easily implies that $\nul\A$ concentrates about its mean.
That said, there is no general theorem that guarantees convergence to a deterministic value in self-averaging systems, so even this aspect of \Thm~\ref{thm:rank} is not in any way a triviality.}

\aco{\Thm~\ref{thm:rank} establishes a generalised version of Lelarge's rank conjecture~\cite{Lelarge} with a tighter conditions on the moments of $\vd,\vk$.
Specifically, Lelarge only considered matrices over the field $\FF_2$, while here we consider general fields and allow for a very general choice of non-zero entries. 
That said, while here we assume that $\Erw[\vd^r],\Erw[\vk^r]<\infty$ for a real $r>2$, Lelarge considered degree distributions with $\Erw[\vd^2],\Erw[\vk^2]<\infty$.
We did not undertake a serious attempt to weaken the moment condition to $r=2$, but this may conceivably introduce significant new techical difficulties.}

The theorem covers a very general class of sparse random matrices.
Indeed, since $\vd,\vk$ have finite means the matrix $\A$ is sparse, i.e., the expected number of non-zero entries is $O(n)$ as $n\to\infty$.
Yet because the degree distributions are subject  only to the condition $\Erw[\vd^r]+\Erw[\vk^r]<\infty$, the typical maximum number of non-zero entries per row or column may approach $\sqrt n$.
Furthermore, the choice of the non-zero entries of the matrix by way of the measurable map $\chi$, reminiscent of an exchangeable array, allows for rather general choices of non-zero matrix entries.
\aco{To elaborate, recall that an exchangeable array is an infinite matrix $(\vec\chi_{ij})_{i,j\geq1}$ of $\FF^*$-valued random variables such that the distribution of any finite top-left sub-matrix is invariant under row and column permutations~\cite{Kallenberg}.
The Aldous--Hoover representation theorem shows that any such array can be described by a function $\cX:[0,1]^4\to\FF^*$~\cite{Aldous,Hoover}.
Specifically, any finite sub-matrix of $\vec\chi_{ij}$ can be obtained by substituting suitable independent random variables that are uniformly distributed on the unit interval $[0,1]$ into $\cX$.
\Thm~\ref{thm:rank} therefore implies the rank formula for a Hadamard product of the biadjacency matrix of the random bipartite graph $\G$ and the commensurately dimensioned top-left bit of the exchangeable array $(\vec\chi_{ij})_{i,j}$.}
Of course, an immediate special case is the random matrix whose non-zero entries are drawn mutually independently from an arbitrary distribution on $\FF^*$.%
	\footnote{To see this, assume that $\CHI$ is an $\FF^*$-valued random variable.
			Then given $n$ pick a large integer $N\gg n^2$.
			Let $\chi:[0,1]^2\to\FF^*$ be a step function obtained by chopping $[0,1]$ into $N$ sub-intervals of size $1/N$ and assigning a value drawn from $\CHI$ independently to each of the $N^2$ resulting rectangles.
			Because \Thm~\ref{thm:rank} provides uniform convergence in $\chi$, we obtain the rank of a matrix with non-zero entries drawn from $\CHI$.}

The lower bound on the rank constitutes the principal contribution of \Thm~\ref{thm:rank}.
Indeed, the upper bound $\rank(\A)/n\leq1-\max_{\alpha\in[0,1]}\Phi(\alpha)+o(1)$ as $n\to\infty$ \whp\ was already derived in~\cite{Lelarge} from the Leibniz determinant formula and the formula for the matching number of a random bipartite graph from~\cite{BLS2}.%
\footnote{While~\cite{Lelarge} only dealt with matrices over $\FF_2$, the argument extends to other fields without further ado.}
Nonetheless, in the appendix we give an independent proof of the upper bound, which is shorter than the combination~\cite{BLS2,Lelarge}.

\aco{\Thm~\ref{thm:rank} implies a formula for the rate of a common class of ldpc codes.
Such codes are based on random matrices $\A$ over finite fields $\FF_q$ with suitable degree distributions $\vd,\vk$.
Specifically, a common construction of ldpc codes involves an optimisation over the degree distributions $\vd,\vk$ of the variables/checks so as to maximise the probability that the Belief Propagation message passing algorithm (or a variant thereof) recovers the original codeword from the received, noisy data~\cite{RichardsonUrbanke}.
The codebook consists of the kernel of the random matrix $\A$.
Hence, the rate of the code equals $\nul\A/n$.
Since \Thm~\ref{thm:rank} implies that
\begin{align*}
	\frac1n\nul\A\to\max_{\alpha\in[0,1]}\Phi(\alpha)&&\mbox{in probability},
\end{align*}
we thus obtain the rate.}

\subsection{The 2-core bound}\label{Sec_intro_2core}
There is a simple graph-theoretic upper bound on the rank, and \Thm~\ref{thm:rank} puts us in a position to investigate if and when this bound is tight.
To state this bound, we recall that the {\em 2-core} of $\G$ is the subgraph $\G_*$ obtained by repeating the following operation.%
\begin{quote}
	While there is a variable node $x_i$ of degree one or less, remove that variable node along with the adjacent check node (if any).	\footnote{Strictly speaking, what we describe here is the 2-core of the hypergraph whose vertices are the variable nodes and whose edges are the neighbourhoods of the check nodes.}
\end{quote}
\fixxx{Of course, the 2-core may be empty, i.e.\ with no variable or check nodes. In the case that $\pr(\bfk=0)>0$ it is possible to have a 2-core without any variable node but with a non-empty set of check nodes whose degrees are all zero.}
Extending prior results that dealt with the degrees of all check nodes coinciding~\cite{Cooper04,molloy2005cores},
we compute the likely number of variable and check nodes in the 2-core.
\fixx{Let
	\begin{equation}\label{def:ff}
		\phi(\alpha)=1-\alpha-D'\left(1-K'(\alpha)/{k}\right)/d.
	\end{equation}
	Note that $\Phi'(\alpha)=dK''(\alpha)\phi(\alpha)/k$.
	Since $\vd,\vk$ have finite second moments and $\phi(0) \ge 0$ while $\phi(1)  \le 0$, we can define
	\begin{equation}\label{def:rho}
		\rho=\max\{x\in[0,1]: \phi(x)=0\}.
	\end{equation}
}

\begin{theorem} \label{thm:2core}
	Assume that $\phi'(\rho)<0$ and let  $\bfn^*$ and $\bfm^*$ be the number of  variable and check nodes in the 2-core, respectively.
	Then
	\begin{align}\label{eqthm:2core}
		\lim_{n\to\infty}\frac{\bfn^*}{n}&=1-D\left(1-\frac{K'(\rho)}{k}\right)-\frac{K'(\rho)}{k} D'\left(1-\frac{K'(\rho)}{k}\right),&
		\lim_{n\to\infty}\frac{\bfm^*}{n}&=\frac dk K( \rho)\quad\mbox{in probability.}
	\end{align} 
\end{theorem}


\fixx{
	
	\begin{remark} \label{remark:specialCases}
		\begin{enumerate}
			\item[(a)] If \ $\pr(\bfk=1)=0$ then $1-D\left(1-\frac{K'(0)}{k}\right)-\frac{K'(0)}{k} D'\left(1-\frac{K'(0)}{k}\right)$  evaluates to zero, and \whp\ $dK(0)/k$ is the number of check nodes with degree zero in $\G$ divided by $n$, up to an $o(1)$ error.
			\item[(b)] If $\bfd\le 1$ then we observe that $\phi(\alpha)=-\alpha$ and thus $\rho=0$. In this case, $1-D\left(1-\frac{K'(0)}{k}\right)-\frac{K'(0)}{k} D'\left(1-\frac{K'(0)}{k}\right)$ evaluates to zero. This agrees with the trivial fact that $\bfn^*=0$, and $\frac{\bfm^*}{n}\to dK(0)/k$ \whp\ in this case.
			\item[(c)] If \ $\pr(\bfk=1)>0$ and $\pr(\bfd\ge 2)>0$ then $\phi(0)>0$, which implies that $\rho>0$. Thus the right hand sides of~\eqref{eqthm:2core} are both positive.
		\end{enumerate}
	\end{remark}
	
	\Thm~\ref{thm:2core} yields an elementary upper bound on the rank of $\A$, as follows, which we refer to as the {\em 2-core bound}:
	\begin{align}\label{eq2corebound}
		\rank(\A)/n\leq 1-\max\{\Phi(0),\Phi(\rho)\}+o(1)\qquad\mbox\whp
	\end{align}
	To see that $\rank(\A)/n\leq 1-\Phi(0)+o(1)$ \whp,
	let $\A'$ be the matrix comprising the rows of $\A$ that contain at most one non-zero entry and let $\vm'$ be the number of such rows.
	Then $\rank(\A)\le \vm-\vm'+\rank(\A')$.
	Moreover, routine arguments reveal that\ $(\vm-\vm')/n \sim 
	d(1-K(0)-K'(0))/k$  and $\rank(\A')/n\sim 1-D(1-K'(0)/k)$ \whp\ (see Appendix~\ref{apx_maurice}  for a proof), deducing the desired upper bound for $\rank(\A)$.
	
	The other upper bound in~\eqn{eq2corebound} can be deduced by considering the 2-core and lower bounding the nullity.
	Counting only solutions to $\A x=0$ where $x_i=0$ for all variables that belong to the 2-core $\G_*$, we obtain
	$\nul(\A)\geq n-\vn^*-(\vm-\vm^*)$. 
	Invoking \Thm~\ref{thm:2core}, we thus find that as $n\to\infty$,
	\[
	\frac{\rank(\A)}{n}\le 1- D\left(1-\frac{K'(\rho)}{k}\right) +\frac{d}{k}(1-K(\rho)) - \frac{K'(\rho)}{k} D'\left(1-\frac{K'(\rho)}{k}\right).
	\]
	Now $\phi(\rho)=0$ implies $D'\left(1-\frac{K'(\rho)}{k}\right)=d(1-\rho)$. Substituting this into the inequality above yields
	\begin{align}\label{eq2corebound_1}
		\rank(\A)/n&\leq 1-\Phi(\rho)+o(1)\qquad\mbox\whp
	\end{align}
 }
The following theorem shows that the 2-core bound is tight in several cases of interest.

\begin{theorem}\label{Thm_tight}
	Assume that
	\begin{enumerate}[(i)]
		\item either $\Var(\vd)=0$ or $\vd\disteq\Po_{\geq\ell}(\lambda)$ for an integer $\ell\geq0$ and $\lambda>0$, and
		\item either $\Var(\vk)=0$ or $\vk\disteq\Po_{\geq\ell'}(\lambda')$ for an integer $\ell'\geq0$ and $\lambda'>0$.
	\end{enumerate}
	Then 
	\begin{align*}
		\lim_{n\to\infty}\rank(\A)/n=1-\max\{\Phi(0),\Phi(\rho)\}\qquad\mbox{in probability}.
	\end{align*}
\end{theorem}

\fixx{
	\begin{remark}\label{remark:condition}
		Under either condition of Theorem~\ref{Thm_tight} (i) or (ii), the condition $\phi'(\rho)<0$ of Theorem~\ref{thm:2core} is satisfied, unless $\pr(\bfd=1)=0$ and $2(k-1)\pr(\bfd=2)\ge d$. We will prove this in the proof of Theorem~\ref{Thm_tight}.
	\end{remark}
}

On the basis of a canny but non-rigorous statistical physics approach called the cavity method several authors predicted  that (over finite fields) the 2-core bound \eqref{eq2corebound} is universally tight for all $\vd,\vk$.
Alamino and Saad reached this conclusion by way of numerical experiments~\cite{AlaminoSaad}, while \Mezard{} and Montanari~\cite{MM} posed a non-rigorous but analytical derivation as an exercise.
However, the prediction turns out to be erroneous.
Indeed, Lelarge~\cite{Lelarge} produced an example of $\vd,\vk$ whose function $\Phi(\alpha)$ attains its unique maximum at a value $0<\alpha<\rho$.
We will see another counterexample momentarily.
On the positive side, \Thm~\ref{Thm_tight} verifies that the 2-core bound actually is tight in all the cases for which Alamino and Saad~\cite{AlaminoSaad} conducted numerical experiments. 

\subsection{Examples}
Let us conclude this section by investigating a few examples of degree distributions $\vd,\vk$ and their resulting rank formulas.

\aco{\begin{example}[the identity matrix]\label{Ex_0}\upshape
	As was brought to our attention by an anonymous reviewer, in the case $\vd=\vk=1$ deterministically the matrix $\A$ is just a permutation matrix, which clearly has full rank.
	Accordingly, we find $D(x)=K(x)=x$ and $\Phi(x)=0$.
	Hence, \eqref{eqthm:rank} boils down to the trivial fact $\rk\A\sim n$.
\end{example}
}

\begin{example}[the adjacency matrix of random bipartite graphs]\label{Ex_1}\upshape
Let $\GG=\GG(n,n,p)$ be a random bipartite graph on vertices $v_1,\ldots,v_n,v_1',\ldots,v_n'$ such that for any $i,j\in[n]$ the edge $\{v_i,v_j'\}$ is present with probability $p$ independently.
With $p=\Delta/n$ for a fixed $\Delta>0$ for large $n$ the vertex degrees asymptotically have distribution $\Po(\Delta)$.
Indeed, with the choice $\vd\disteq\Po(\Delta)$ and $\vk\disteq\Po(\Delta)$ the adjacency matrix $A(\GG(n,n,p))$ and the random matrix $\vA$ can be coupled such that
$\rk{A(\GG(n,n,p))}=\rk\bc{\vA}+o(n)$ \whp\ 
Hence, \Thm~\ref{thm:rank} shows that over any field $\FF$,
\begin{align*}
\lim_{n\to\infty}	\frac{\rk(A(\GG(n,n,p)))}{n}&=2-
\max\cbc{
	\exp(-\Delta\exp(\Delta(\alpha-1)))+(1+(1-\alpha)\Delta)\exp(\Delta(\alpha-1)):{\alpha\in[0,1]}}
\end{align*}
in probability.
\Thm~\ref{Thm_tight} implies that the 2-core bound is tight in this example.
\end{example}

\begin{example}[fixed row sums]\label{Ex_XOR}\upshape
Motivated by the minimum spanning tree problem on weighted random graphs, Cooper, Frieze and Pegden~\cite{CFP} studied the rank of the random matrix with degree distributions $\vk=k\geq3$ fixed and $\vd\disteq\Po(d)$ over the field $\FF_2$.
The same rank formula was obtained independently in~\cite{Ayre} for arbitrary finite fields.
Extending both these results, \Thm~\ref{thm:rank} shows that the rank of the random matrix with these degrees over any field $\FF$ with any choice $\chi$ of non-zero entries is given by
\begin{align*}
\lim_{n\to\infty}	\frac{\rk\vA}{n}&=1-\max\cbc{\exp(-d\alpha^{k-1})-\frac{d}{k}\bc{1-k\alpha^{k-1}+(k-1)\alpha^k}:{\alpha\in[0,1]}}.
\end{align*}
Once more \Thm~\ref{Thm_tight} shows that the 2-core bound is tight.
\end{example}

\begin{example}[non-exact 2-core bound]\label{Ex_3}\upshape
There are plenty of choices of $\vd,\vk$ where the 2-core bound fails to be tight. Degree distributions that render graphs $\G$ with an unstable 2-core furnish particularly egregious offenders.
In such graphs the removal of a small number of randomly chosen checks $a_i$ likely causes the 2-core to collapse.
Analytically, the instability manifests itself in $\rho$ from \eqref{def:rho} being a local minimum of $\Phi(x)$.
\aco{For instance, letting $\vd,\vk$ be the distributions with $D(x)=(22x^2+3x^{11})/25$ and $K(x)=x^3$, we obtain
	\begin{align}\nonumber
		\Phi(x)=&-\frac{3}{25} \, x^{22} + \frac{33}{25} \, x^{20} - \frac{33}{5} \, x^{18} + \frac{99}{5} \, x^{16} - \frac{198}{5} \, x^{14} + \frac{1386}{25} \, x^{12} - \frac{1386}{25} \, x^{10} + \frac{198}{5} \, x^{8}\\& - \frac{99}{5} \, x^{6} + \frac{187}{25} \, x^{4} - \frac{154}{75} \, x^{3} - \frac{2}{75}.
			\label{eqEx16Phi}
\end{align}}
Hence, $\rho=1$ and $\Phi''(1)>0$, while the global maximum is attained at $\alpha\approx0.75$.
\end{example}

\begin{figure}
\includegraphics[height=40mm]{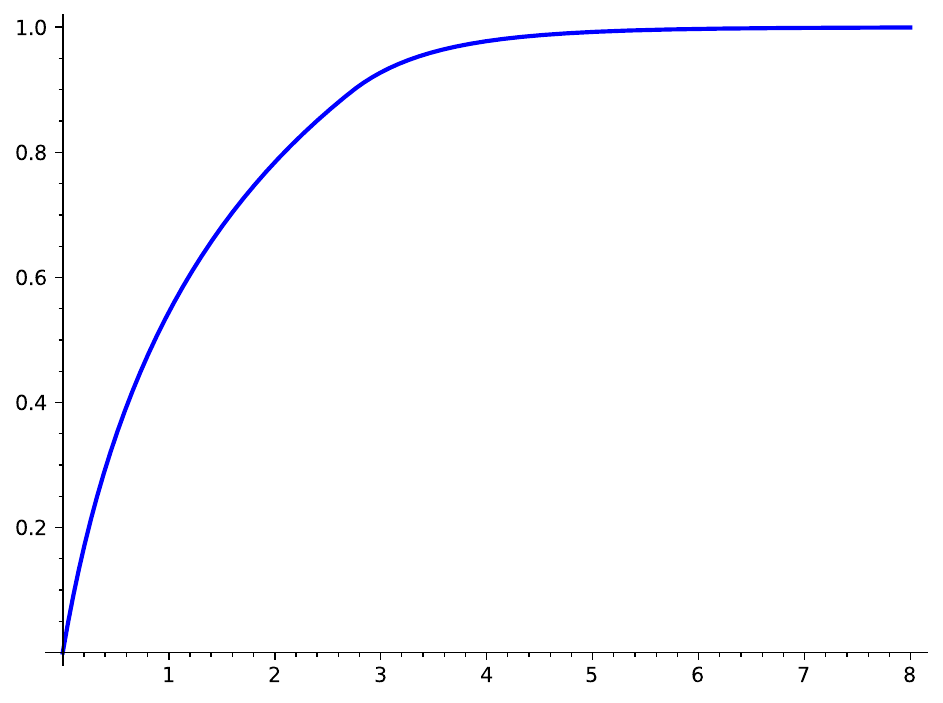}
\includegraphics[height=40mm]{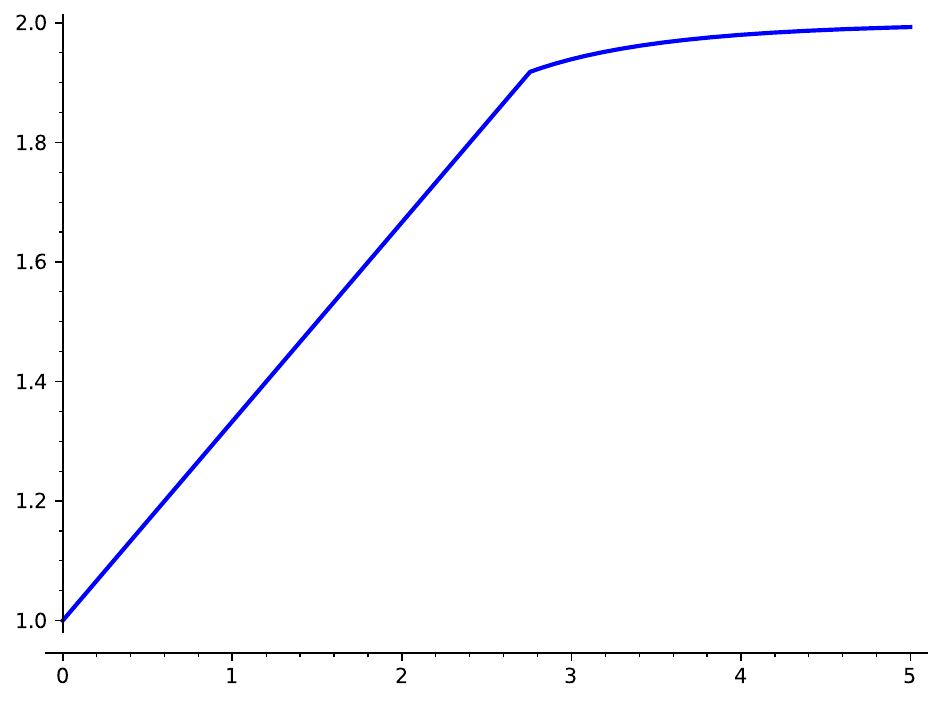}
\includegraphics[height=40mm]{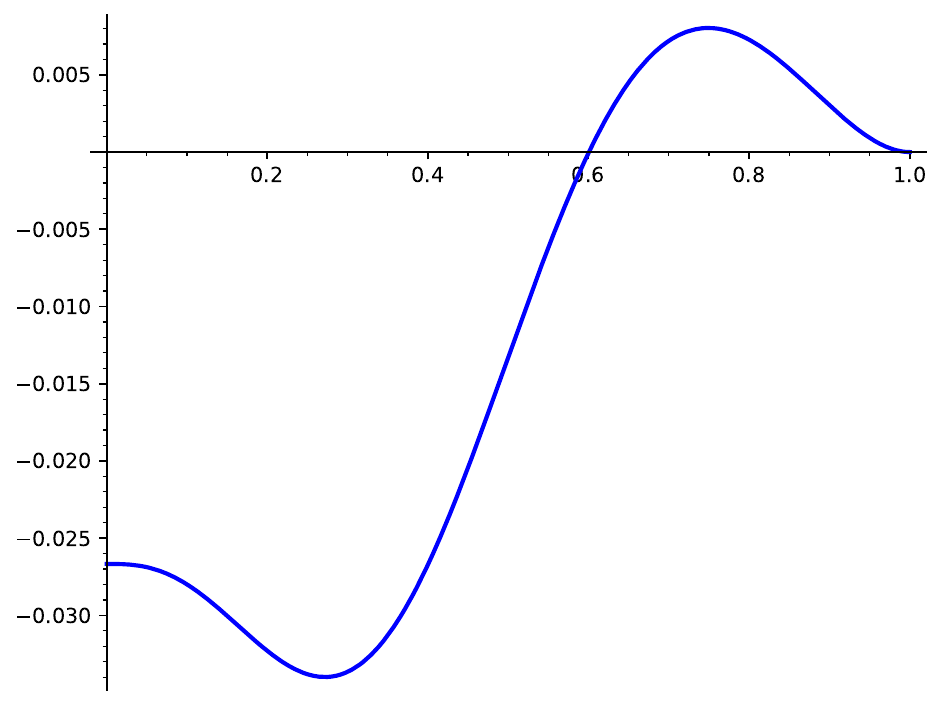}
\caption{Left: the function $\Delta\mapsto2-\max_{\alpha\in[0,1]}
	\exp(-\Delta\exp(\Delta(\alpha-1)))+(1+(1-\alpha)\Delta)\exp(\Delta(\alpha-1))$ for Example~\ref{Ex_1}.
	Middle: the function $d\mapsto1-\max_{\alpha\in[0,1]}\exp(-d\alpha^{k-1})-d(1-k\alpha^{k-1}+(k-1)\alpha^k)/k$ from Example~\ref{Ex_XOR} with $k=3$.
Right: the function $\Phi(x)$ from \eqref{eqEx16Phi} for Example~\ref{Ex_3}.}
\end{figure}

\subsection{Preliminaries}\label{Sec_pre}
Throughout the paper we consistently keep the assumptions on the distributions $\vd,\vk$ listed in \Sec~\ref{Sec_intro}.
In particular, $\Erw[\vd^r]+\Erw[\vk^r]<\infty$ for some real $r>2$.
Because all-zero rows and columns do not add to the rank, we may assume that $\vd\geq1,\vk\geq1$.
We write $\gcd(\vk)$ and $\gcd(\vd)$ for the greatest common divisor of the support of $\vd$ and $\vk$, respectively.
When working with $\A$ we tacitly assume that $\gcd(\vk)$ divides $n$.
In order to highlight the number of columns we write $\vA_n=\vA$ and $\G_n=\G$ for the corresponding Tanner graph.
The following proposition,  whose proof can be found in \Sec~\ref{Sec_welldef}, shows that $\A_n$ is well-defined.

\begin{proposition}\label{Lemma_welldef}
With probability $\Omega(n^{-1/2})$ over the choice of $\vm$, $(\vd_i)_{i\geq1}$, $(\vk_i)_{i\geq1}$ the condition \eqref{eqWellDef1} is satisfied and there exists a simple Tanner graph $\G$ with variable degrees $\vd_1,\ldots,\vd_n$ and check degrees $\vk_1,\ldots,\vk_{\vm}$.
\end{proposition}

We introduce the size-biased random variables
\begin{align}\label{eqSizeBiasd}
\pr\brk{\hat\vd=\ell}&=\ell\pr\brk{\vd=\ell}/{d},&
\pr\brk{\hat\vk=\ell}&=\ell\pr\brk{\vk=\ell}/{k}&(\ell\geq0).
\end{align}
Throughout the paper we let $(\vk_i,\vd_i,\hat\vk_i,\hat\vd_i)_{i\geq1}$ denote mutually independent copies of $\vk,\vd,\hat\vk,\hat\vd$.
Unless specified otherwise, all these random variables are assumed to be independent of any other sources of randomness.

We use common notation for graphs and multi-graphs.
For instance, for a vertex $v$ of a multi-graph $G$ we denote by $\partial_G v$ the set of neighbours of $v$.
More generally, for an integer $\ell\geq1$ we let $\partial_G^\ell v$ be the set of vertices at distance precisely $\ell$ from $v$.
We omit the reference to $G$ where possible.

The proofs of the main results rely on taking a double limit where we first take the number $n$ of columns to infinity and subsequently send an error parameter $\eps$ to zero.
We use the asymptotic symbols with an index $n$ such as $O_n(\nix)$, $o_n(\nix)$ to refer to the inner limit $n\to\infty$ only.
Thus, for functions $f(\eps,n),g(\eps,n)$ we write
\begin{align*}
f(\eps,n)&=O_n(g(n,\eps))&\mbox{if pointwise for every $\eps>0$, }&&
	\limsup_{n\to\infty}\abs{\frac{f(\eps,n)}{g(\eps,n)}}&<\infty,\\
f(\eps,n)&=o_n(g(n,\eps))&\mbox{if pointwise for every $\eps>0$, }&&
	\limsup_{n\to\infty}\abs{\frac{f(\eps,n)}{g(\eps,n)}}&=0.
\end{align*}
For example, $1/(\eps n)=o_n(1)$.
Additionally, we will use the symbols $O_{\eps,n}$, $o_{\eps,n}$, etc.\ to refer to the double limit $\eps\to 0$ after $n\to\infty$.
Thus, 
\begin{align*}
f(\eps,n)&=O_{\eps,n}(g(\eps,n))&\mbox{if }&&\limsup_{\eps\to0}\limsup_{n\to\infty}\abs{\frac{f(\eps,n)}{g(\eps,n)}}&<\infty,\\
f(\eps,n)&=o_{\eps,n}(g(\eps,n))&\mbox{if }&&\limsup_{\eps\to0}\limsup_{n\to\infty}\abs{\frac{f(\eps,n)}{g(\eps,n)}}&=0.
\end{align*}
For instance, $\eps+1/(\eps n)=o_{\eps,n}(1)$.

Finally, we need the following basic lemma on sums of independent random variables.

\begin{lemma}\label{Lemma_sums}
Let $r>2$, $\delta>0$ and suppose that $(\vec\lambda_i)_{i\geq1}$ are independent copies of a random variable $\vec\lambda\geq0$ with $\Erw[\vec\lambda^r]<\infty$.
Further, let $s=\Theta_n(n)$.
Then
$\pr\brk{\abs{\sum_{i=1}^{s}(\vec\lambda_i-\Erw[\vec\lambda])}>\delta n}=o_n(1/n).$
\end{lemma}

\noindent
For the sake of completeness the proof of \Lem~\ref{Lemma_sums} is included in the appendix.

\section{Overview}\label{Sec_overview}

\noindent
We survey the proof of \Thm~\ref{thm:rank} and  subsequently compare these techniques with those employed in prior work.
The main contribution of the paper is the `$\geq$'-part of \eqref{eqthm:rank}, i.e., the lower bound on the rank.
We prove this lower bound via a technique inspired by the physicists' cavity method.
The scaffolding of the proof is provided by a coupling argument reminiscent of a proof strategy known in  mathematical physics jargon under the name `Aizenman-Sims-Starr scheme'~\cite{Aizenman} or `cavity ansatz'~\cite{MM}:
\begin{quote}
To calculate the mean of a random variable $X_n$ on a random system of size $n$ in the limit $n\to\infty$, calculate the difference $\Erw[X_{n+1}]-\Erw[X_n]$ upon going to a system of size $n+1$.
Perform this calculation by coupling the systems of sizes $n$ and $n+1$ such that the latter results from the former by adding only a bounded number of elements.
\end{quote}
We will apply this approach to $X_n=\nul\vA_n$.
The coupling will be such that $X_{n+1}$ is the nullity of a random matrix obtained from $\vA_n$ obtained by adding a few rows and columns.
Thus, we need to calculate the ensuing change in nullity upon adding to a matrix several rows/columns whose number is random and bounded in expectation.

In general, such a calculation hardly seems possible.
To carry it out we would need to understand the linear dependencies among the coordinates where the new rows sport non-zero entries, an exceedingly complicated task.
Two facts deliver us from this complexity.
First, the positions of the non-zero entries of the new rows are (somewhat) random.
Second, we develop a random perturbation, applicable to any matrix, that diminishes the number of short linear relations (\Prop~\ref{Prop_Alp} below).
\aco{To be precise, we will conclude that by applying the perturbation, for any fixed $\ell$ the probability that a set of $\ell$ coordinates forms a proper relation in the sense of Definition~\ref{Def_Alp} below can be made negligibly small without substantially altering the nullity.}
In effect, the probability that there will be linear dependencies among the positions of the non-zero entries of the new rows will turn out to be negligible.
Since this perturbation argument is the linchpin of the entire proof, this is what we shall begin with.
Subsequently we will explain how this general perturbation renders the desired lower bound on the rank.

\subsection{Short linear relations}\label{Sec_stoch}
Define the {\em support} of a vector $\xi\in\FF^U$ as $\supp(\xi)=\cbc{i\in U:\xi_i\neq0}$.

\begin{definition}\label{Def_Alp}
Let $A$ be an $m\times n$-matrix over a field $\FF$.
\begin{itemize}
\item A set $\emptyset\neq I\subset[n]$ is a \bemph{relation} of $A$ if there exists a row vector 
$y\in\FF^{1\times m}$ such that $\emptyset\neq\supp(y A)\subset I$.
\item If $I=\cbc{i}$ is a relation of $A$, then we call $i$ \bemph{frozen} in $A$. Let $\fF(A)$ be the set of all frozen $i\in[n]$.
\item A set $I\subset[n]$ is a \bemph{proper relation} of $A$ if $I\setminus\fF(A)$ is a relation of $A$.
\item For $\delta>0$, $\ell\geq1$ we say that $A$ is \bemph{$(\delta,\ell)$-free} if there are no more than $\delta n^\ell$ proper relations $I\subset[n]$ of size $|I|=\ell$.
\end{itemize}
\end{definition}

Thus, if $I\subset[n]$ is a relation of $A$, then by adding up suitable multiples of the rows of the homogeneous linear system $Ax=0$ we can infer a non-trivial linear relation involving the variables $(x_i)_{i\in I}$ only.
In the simplest case the set $I=\{i\}$ may be a singleton.
Then the equation $x_i=0$ is implicit in $Ax=0$ and we call coordinate $i$ frozen.
\aco{In particular, $i$ is frozen if $A$ contains a row whose only non-zero entry appears in column $i$.
	However, this is not the only possibility.
	For instance, in the following $\FF_2$-matrix variable $x_1$ is frozen because the sum of all three rows equals $(1\ 0\ 0)$:
	\begin{align}\label{eqKernelEx}
	\begin{pmatrix}1&0&1&1\\1&1&0&1\\1&1&1&0\end{pmatrix}.
\end{align}
In effect, for any vector $\xi$ in the kernel of \eqref{eqKernelEx} we have 
\begin{align}\label{eqfact_frozen1}
0=(1\ 1\ 1)\begin{pmatrix}0\\0\\0\end{pmatrix}=(1\ 1\ 1)\brk{\begin{pmatrix}1&0&1&1\\1&1&0&1\\1&1&1&0\end{pmatrix}\xi}=\brk{(1\ 1\ 1)\begin{pmatrix}1&0&1&1\\1&1&0&1\\1&1&1&0\end{pmatrix}}\xi=(1\ 0\ 0)\xi=\xi_1.
\end{align}
Generally, a linear number $\Omega(n)$ of rows may have to collude to cause freezing.
Moreover, although the proof is just a bit of routine linear algebra, it is worthwhile including the following characterisation of frozen coordinates.}

\begin{fact}\label{fact_frozen}
	\aco{A coordinate $i$ is frozen in the matrix $A$ iff $\xi_i=0$ for all $\xi\in\ker A$.}
\end{fact}
\begin{proof}
\aco{	Let $A$ be an $m\times n$ matrix over an arbitrary field.
The calculation from \eqref{eqfact_frozen1} readily generalises to arbitrary matrices and implies that $\xi_i=0$ for any frozen coordinate $i\in[n]$ and any $\xi\in\ker A$.}

\aco{	Conversely, assume that for coordinate $i\in[n]$ we have $\xi_i=0$ for all $\xi\in\ker A$.
	Let $e^{(i)}\in\FF^{1\times n}$ be the vector whose $i$-th coordinate equals one and whose other coordinates are equal to zero.
	Moreover, obtain $A^+$ from $A$ by adding $e^{(i)}$ as an extra row.
	Because $\xi_i=0$ for all $\xi\in\ker A$ we have $\ker A^+=\ker A$.
	Therefore, $\rank A=\rank A^+$ and thus $e^{(i)}$ is a linear combination of the rows of $A$.
Hence, $i\in\fF(A)$.}
\end{proof}

Furthermore, excluding frozen coordinates, a proper relation $I$ of $A$ renders a non-trivial linear relation amongst at least two of the variables $(x_i)_{i\in I}$.
Finally, $A$ is $(\delta,\ell)$-free if only few $\ell$-subsets $I\subset[n]$ are proper relations.

We proceed to put forward a small random perturbation that will mostly rid a given matrix of short proper relations, an observation that we expect to be of independent interest.

\begin{definition}\label{Def_pin}
Let $A$ be an $m\times n$ matrix and let $\theta\geq0$ be an integer.
Let $\vi_1,\vi_2,\ldots,\vi_\theta\in[n]$ be uniformly random and mutually independent column indices.
Then the matrix $A[\theta]$ is obtained by adding $\theta$ new rows to $A$ such that for each $j\in[\theta]$ the $j$-th new row has precisely one non-zero entry, namely a one in the $\vi_j$-th column.
\end{definition}

\noindent
In other words, in $A[\theta]$ we expressly peg $\theta$ randomly chosen variables $x_{\vi_1},\ldots,x_{\vi_\theta}$ of the linear system $Ax=0$ to zero.
The proof of the following proposition is based on a blend of algebraic and probabilistic ideas. 

\begin{proposition}\label{Prop_Alp}
For any $\delta>0$, $\ell>0$ there exists $\cT=\cT(\delta,\ell)>0$ such that for any matrix $A$ over any field $\FF$ the following is true.
With $\THETA\in[\cT]$ chosen uniformly at random we have
\begin{align}\label{eqAlp}
\pr\brk{\mbox{$A[\THETA]$ is $(\delta,\ell)$-free}}>1-\delta.
\end{align}
\end{proposition}  

\noindent
The key feature of \Prop~\ref{Prop_Alp} is that the maximum number $\cT$ of variables that get pegged to zero does not depend on the matrix $A$ or its size but on $\delta$ and $\ell$ only.
Moreover, since adding a single row can change the nullity by at most one, we obtain
$|\nul(A)-\nul A[\THETA]|\leq\cT$.
Hence, while eliminating short proper relations, the perturbation does not shift the nullity significantly.
\Prop~\ref{Prop_Alp} is a sweeping generalisation of a probabilistic result from~\cite{Ayre}, where the perturbation from Definition~\ref{Def_pin} was applied to matrices over finite fields to diminish stochastic dependencies amongst entries of randomly chosen vectors in the kernel.
That argument, in turn, was inspired by ideas from information theory~\cite{CKPZ,Montanari,Raghavendra}.
We will come back to this in \Sec~\ref{Sec_discussion}.

We will incorporate the perturbation from \Prop~\ref{Prop_Alp} into the Aizenman-Sims-Starr coupling argument, which reduces the rank calculation to studying the impact of a few additional rows and columns on the rank.
The following lemma, whose proof consists of a few lines of linear algebra, shows how 
the impact of such operations can be tracked in the absence of proper relations.
Specifically, the lemma shows that all we need to know about the matrix $A$ to which we add rows/columns is the set $\fF(A)$ of frozen variables.

\begin{lemma}\label{Cor_free}
Let $A,B,C$ be matrices of size $m \times n$, $m'\times n$ and $m'\times n'$, respectively, and let $I\subset[n]$ be the set of all indices of non-zero columns of $B$.
Moreover, obtain $B_*$ from $B$ by replacing for each $i\in I\cap\fF(A)$ the $i$-th column of $B$ by zero.
Unless $I$ is a proper relation of $A$ we have
\begin{align}\label{eqLemma_free}
\nul\begin{pmatrix}A&0\\B&C\end{pmatrix}-\nul A=n'-\rk(B_*\ C).
\end{align}
\end{lemma}

\noindent
\aco{Observe that the quantity on the l.h.s.\ of \eqref{eqLemma_free} (and thus the one on the r.h.s.\ as well) may be either positive or negative, depending on $A,B,C$.}

To put \Prop~\ref{Prop_Alp} and \Lem~\ref{Cor_free} to work, we need to explain the construction of the telescoping series of random variables upon which the Aizenman-Sims-Starr argument is based.
That is our next step.

\subsection{The Aizenman-Sims-Starr scheme}\label{Sec_outline1}
In order to derive the desired lower bound on the rank we need to bound the nullity of $\vA_n$ from above.
\aco{In line with the Aizenman-Sims-Starr scheme~\cite{Aizenman,Panchenko}, a first stab at this problem might be to write a telescoping sum
\begin{align}\label{eqtelescope1}
	\limsup_{n\to\infty}\frac1n\Erw[\nul(\vA_n)]&=\limsup_{\nup\to\infty}\frac1{\nup}\sum_{n=1}^{\nup-1}\Erw[\nul(\vA_{n+1})]-\Erw[\nul(\vA_{n})].
\end{align}
Providing that $\Erw[\nul(\vA_{n+1})]-\Erw[\nul(\vA_{n})]$ is bounded, the lim sup of the sequence of summands exists.
In this case, due to the normalising factor $1/N$ on the r.h.s.\ of \eqref{eqtelescope1}, we obtain
\begin{align}\label{eqtelescope2}
	\limsup_{\nup\to\infty}\frac1{\nup}\sum_{n=1}^{\nup-1}\Erw[\nul(\vA_{n+1})]-\Erw[\nul(\vA_{n})]&\leq\limsup_{n\to\infty}\Erw[\nul(\vA_{n+1})]-\Erw[\nul(\vA_{n})].
\end{align}
Hence, combinig \eqref{eqtelescope1} and \eqref{eqtelescope2}, we obtain the bound
\begin{align*}
	\limsup_{n\to\infty}\frac1n\Erw[\nul(\vA_n)]\leq\limsup_{n\to\infty}\Erw[\nul(\vA_{n+1})]-\Erw[\nul(\vA_{n})].
\end{align*}}
To obtain an explicit estimate, we should thus attempt to couple $\vA_{n+1}$ and $\vA_{n}$ so that we can write a single expectation
\begin{align}\label{eqASS1}
\Erw[\nul(\vA_{n+1})]-\Erw[\nul(\vA_{n})]=\Erw\brk{\nul(\vA_{n+1})-\nul(\vA_{n})}.
\end{align}
Ideally, to bring the tools from \Sec~\ref{Sec_stoch} to bear, under this coupling $\vA_{n+1}$ should be obtained from $\vA_n$ by adding one column and a few rows.

Unfortunately, this direct approach flounders for obvious reasons.
For instance, depending on the distributions $\vd,\vk$, due to divisibility issues $\A_{n+1}$ may not even be defined for all $n$.%
	\footnote{For instance, suppose that $\vd=3$ and $\vk=4$ deterministically.
		Then \eqref{eqWellDef1} boils down to $4\vm=3n$, and thus $\vA_n$ is well-defined only if $n$ is divisible by four.}
To deal with this issue we introduce a more malleable version of the random matrix model, without significantly altering the rank.
Specifically, we introduce a parameter $\eps>0$, for which we choose a large enough $\cT=\cT(\eps)>0$.
Then for integers $n\geq\cT$ we construct a random matrix  $\vA_{\eps,n}$ as follows.
Like in \Sec~\ref{Sec_rankFormula} let $\chi:[0,1]^2\to\FF^*$ be a measurable map and let $(\row_i,\col_i)_{i\geq1}$ be uniformly distributed $[0,1]$-valued random variables.
Further, let
 \begin{align*}
  \vm_{\eps,n}&\disteq\Po((1-\eps)dn/k)
 \end{align*}
Additionally, choose $\THETA\in[\cT]$ uniformly at random and, as before, let $(\vd_i)_{i\geq1}$, $(\vk_i)_{i\geq1}$ be copies of $\vd$, $\vk$.
All of these random variables are mutually independent. 
Further, let $\vec\Gamma_{\eps,n}$ be a uniformly random maximal matching of the complete bipartite graph with vertex classes
\begin{align*}
\bigcup_{i=1}^{\vm_{\eps,n}}\cbc{a_i}\times[\vk_i]\qquad\mbox{and}\qquad\bigcup_{j=1}^{n}\cbc{x_j}\times[\vd_j].
\end{align*}
As in the well known configuration model of random graphs, 
we think of $\cbc{a_i}\times[\vk_i]$ as a set of clones of $a_i$ and of $\{x_j\}\times[\vd_j]$ as a set of clones of $x_j$.
We obtain a random Tanner graph $\G_{\eps,n}$ with variable nodes $x_1,\ldots,x_n$ and check nodes $a_1,\ldots,a_{\vm_{\eps,n}},p_1,\ldots,p_{\THETA}$ 
by inserting an edge between $a_i$ and $x_j$ for each matching edge that joins the sets $\cbc{a_i}\times[\vk_i]$ and $\{x_j\}\times[\vd_j]$.
Additionally, check node $p_i$ is adjacent to $x_i$ for each $i\in[\THETA]$.
\aco{To be clear, we do not need to set aside any unmatched variable clones as partners of the $p_i$. We simply add the $x_i$-$p_i$-edges on top of the configuration model.
Since ultimately $\cT$ will be chosen to be of order $o(n)$, the number of these additional edges is relatively small.}

Since there may be several edges joining clones of the same variable and check node, $\G_{\eps,n}$ may be a multigraph.
Finally, we construct a random matrix $\vA_{\eps,n}$ whose rows are indexed by the check nodes $a_1,\ldots,a_{\vm_{\eps,n}}$ and whose columns are indexed by $x_1,\ldots,x_n$ such that the non-zero entries of $\vA_{\eps,n}$ represent the edges of the matching $\vec\Gamma_{\eps,n}$.
Specifically, the matrix entries read
\begin{align*}
(\vA_{\eps,n})_{p_i,x_j}&=\vecone\cbc{i=j}&&(i\in[\THETA],j\in[n]),\\
(\vA_{\eps,n})_{a_i,x_j}&=\chi_{\row_i,\col_j}\sum_{s=1}^{k_i}\sum_{t=1}^{\vd_j}\vecone\cbc{\{(a_i,s),(x_j,t)\}\in\vec\Gamma_{\eps,n}}&&(i\in[\vm_{\eps,n}],j\in[n]).
\end{align*}

Morally, $\vA_{\eps,n}$ mimics the matrix obtained from the original model $\vA_n$ by deleting every row with probability $\eps$ independently (which, of course, would be unworkable because still the model is not generally defined for all $n$).
Furthermore, the purpose of the check nodes $p_1,\ldots,p_{\THETA}$ is to ensure that $\vA_{\eps,n}$ is $(\delta,\ell)$-free for a small enough $\delta=\delta(\eps)$ and a large enough $\ell=\ell(\eps)$.
Indeed, while \Prop~\ref{Prop_Alp} requires that a {\em random} set of $\THETA$ variables be pegged, the checks $p_1,\ldots,p_{\THETA}$ just freeze the first $\THETA$ variables.
But since the distribution of the Tanner graph $\G_{\eps,n}-\{p_1,\ldots,p_{\THETA}\}$ is invariant under permutations of the variable nodes,  both constructions are equivalent.
The following proposition shows that going to $\vA_{\eps,n}$ does not shift the rank significantly.

\begin{proposition}\label{Cor_lower}
For any any $0<C<C'$  and any function $\cT=\cT(\eps)\geq0$ the following is true.
If 
\begin{align}\label{eqCor_lower_ub}
\limsup_{\eps\to0}\limsup_{n\to\infty}\frac1n\Erw[\nul(\vA_{\eps,n})]&\leq C
\qquad\mbox{then}\qquad\lim_{n\to\infty}\pr\brk{\nul(\vA_{n})\leq C'n}=1.
\end{align}
Analogously, if
\begin{align*}
\liminf_{\eps\to0}\liminf_{n\to\infty}\frac1n\Erw[\nul(\vA_{\eps,n})]\geq C'
\qquad\mbox{then}\qquad
\lim_{n\to\infty}\pr\brk{\nul(\vA_{n})\geq Cn}=1.
\end{align*}
\end{proposition}

By construction, the degrees of the checks $a_i$ and the variables $x_j$ in $\G_{\eps,n}-\{p_1,\ldots,p_{\THETA}\}$ are upper-bounded by $\vk_i$ and $\vd_j$, respectively.
We thus refer to $\vk_i$ and $\vd_j$ as the {\em target degrees} of $a_i$ and $x_j$.
Indeed, since $\G_{\eps,n}$ will turn out to feature few if any multi-edges and $\vm_{\eps,n}$ is significantly smaller than $dn/k$ and thus
$$\pr\brk{\sum_{i=1}^{\vm_{\eps,n}}\vk_i\leq\sum_{i=1}^{n}\vd_i}=1-o_n(1),$$
most check nodes $a_i$ have degree precisely $\vk_i$ \whp\
But we expect that about $\eps dn$ variable nodes $x_i$ will have degree less than $\vd_i$.
In fact, \whp\ $\vec\Gamma_{\eps,n}$ fails to cover about $\eps dn$ `clones' from the set $\bigcup_{i=1}^n\{x_i\}\times[\vd_i]$.
Let us call such unmatched clones {\em cavities}.

The cavities provide the wiggle room that we need to couple $\vA_{\eps,n}$ and $\vA_{\eps,n+1}$.
An instant idea might be to couple $\G_{\eps,n+1}$ and $\G_{\eps,n}$ such that the former is obtained  by adding one variable node $x_{n+1}$ along with $\vd_{n+1}$ new adjacent check nodes.
Additionally, the new checks get connected with some random cavities of $\G_{\eps,n}$.
In effect, the coupling takes the form
\begin{align}\label{eqroughcoupling}
\nul\vA_{\eps,n+1}&=\nul\begin{pmatrix}
\vA_{\eps,n}&0\\\vB&\vC
\end{pmatrix},
\end{align}
where $\vB$ has $n$ columns and $\vd_{n+1}$ rows and $\vC$ is a column vector of size $\vd_{n+1}$ \whp{}
But this direct attempt has a subtle flaw.
Indeed, going from $\vA_{\eps,n}$ to $\vA_{\eps,n+1}$, \eqref{eqroughcoupling} adds $\Erw[\vd_{n+1}]=d$ rows on the average.
Yet actually we should be adding merely $\Erw[\vm_{\eps,n+1}-\vm_{\eps,n}]=(1-\eps)d/k$ rows.
To remedy this problem we borrow a trick from prior applications of the Aizenman-Sims-Starr scheme in combinatorics~\cite{Ayre,CKPZ,BetheLattices}.
Namely, we set up a coupling under which both $\vA_{\eps,n},\vA_{\eps,n+1}$ are obtained by adding a few rows/columns to a common `base matrix' $\vA'$.
Thus, instead of \eqref{eqroughcoupling} we obtain
\begin{align}\label{eqfinecoupling}
\nul\vA_{\eps,n}&=\nul\begin{pmatrix}\vA'\\\vB\end{pmatrix},&
\nul\vA_{\eps,n+1}&=\nul\begin{pmatrix}\vA'&0\\\vB'&\vC'\end{pmatrix}.
\end{align}
To be precise, $\vC'$ above is a column vector with an expected $(1-\eps)d$ non-zero entries and $\vB,\vB'$ are matrices whose numbers of non-zero entries are bounded in expectation.
Furthermore, the base matrix $\vA'$ itself is quite similar to $\vA_{\eps,n}$, except that $\vA'$ has a slightly smaller number of rows.
In \Sec~\ref{Sec_lower} we will present the construction in full detail and apply \Prop~\ref{Prop_Alp} and \Lem~\ref{Cor_free}  to prove the following upper bound on the change in nullity.
Recall the function $\Phi$ from \eqref{eqBFE} and recall that $\THETA\in[\cT]$ with $\cT=\cT(\eps)$ dependent on $\eps$ only is the number of pinned variables in the construction of $\vA_{\eps,n}$.

\begin{proposition}\label{Prop_coupling}
There exists a function $\cT=\cT(\eps)>0$ such that
\begin{align*}
\limsup_{\eps\to0}\limsup_{n\to\infty}\Erw[\nul(\vA_{\eps,n+1})]-\Erw[\nul(\vA_{\eps,n})]\leq\max_{\alpha\in[0,1]}\Phi(\alpha).
\end{align*}
\end{proposition}

\noindent
As an immediate consequence of \Prop~\ref{Prop_coupling} we obtain the desired upper bound on the nullity.

\begin{corollary}\label{Prop_lower}
We have
	$$\limsup_{\eps\to0}\limsup_{n\to\infty}\frac1n\Erw[\nul(\vA_{\eps,n})]\leq\max_{\alpha\in[0,1]}\Phi(\alpha).$$
\end{corollary}
\begin{proof}
 \Prop~\ref{Prop_coupling} yields
\begin{align*}
\frac1n\Erw[\nul(\vA_{\eps,n})]&=\frac1n\brk{\Erw[\nul(\vA_{\eps,1})]+\sum_{N=1}^{n-1}\bc{\Erw[\nul(\vA_{\eps,N+1})]-\Erw[\nul(\vA_{\eps,N})]}}\leq\max_{\alpha\in[0,1]}\Phi(\alpha)+o_{\eps,n}(1),
\end{align*}
as claimed.
\end{proof}

\begin{proof}[Proof of \Thm~\ref{thm:rank}]
The desired lower bound on the rank of $\vA_n$ is an immediate consequence of \Prop~\ref{Cor_lower} and \Cor~\ref{Prop_lower}.
\end{proof}

\subsection{Discussion}\label{Sec_discussion}

\noindent
Before delving into the technical details of the proofs of the various propositions, we compare the proof strategy and the results with previous work.
We begin with a discussion of related work on the rank problem.
Roughly speaking, prior work on the rank of random matrices relies on separate strands of techniques, depending on whether the average number of non-zero entries per row/column is bounded or unbounded.
Subsequently we discuss the physicists' (non-rigorous) cavity method and explain how it led to an erroneous prediction.

\subsubsection{Dense matrices}
The difficulty of the rank problem for dense random matrices strongly depends on the distribution of the matrix entries.
For instance, a square matrix with independent Gaussian entries in each row has full rank with probability one simply because the submanifold of singular matrices has Lebesgue measure zero.
By contrast, the case of matrices with independent uniform $\pm1$ entries is more subtle.
\Komlos~\cite{Komlos} proved that such matrices are regular \whp\ 
Vu~\cite{Vu} subsequently presented a simpler proof, based on collision probabilities and \Erdos' Littlewood-Offord inequality.
An intriguing conjecture, which has inspired a distinguished line of research~\mr{\cite{Jai,Lit,KKS,TaoVu,Tikhomirov}},  asserts that the dominant reason for a random  $\pm1$-matrix being singular is the existence of a pair of identical rows or columns.

Interesting enough, the limiting probability that a dense square matrix with entries drawn uniformly from a finite field is singular lies strictly between zero and one.
Kovalenko and Levitskaya~\cite{Kovalenko,Kovalenko2,Lev1,Lev2} obtained a precise formula for the distribution of the rank of dense random matrices with independent entries over finite fields via the method of moments.
For more recent improvements see~\cite{Fulman,Luh} and the references therein.

A further line of work deals with random $m\times n$ matrices in which the number of non-zero entries per row diverges in the limit of large $n$ but is of order $o_n(n)$.
Relating the permanent and the determinant, Balakin~\cite{Balakin2} and, using delicate moment calculations,  Bl\"omer, Karp and Welzl~\cite{BKW} dealt with the rank of such matrices over finite fields.
Moreover, using expansion arguments, Costello and Vu~\cite{costello2008rank,costello2010rank} studied the real rank of random symmetric matrices of a similar density.
They find that such matrices essentially have full rank \whp, apart from a small defect based on local phenomena.
In the words of~\cite{costello2010rank}, ``dependency [comes] from small configurations''.

\subsubsection{Sparse matrices}
Matters are quite different in the sparse case where the average number of non-zero entries per row is bounded.
In fact, as we will discover in due course
the formula from \Thm~\ref{thm:rank} is driven by ``dependency coming from large configurations'', i.e., by minimally linearly dependent sets of unbounded size.

The first major contribution dedicated to sparse matrices was a paper by Dubois and Mandler~\cite{DuboisMandler} on the random $3$-XORSAT problem.
Translated into random matrices, this problem asks for what ratios $m/n$  a random $m\times n$-matrix over $\FF_2$ with precisely three one-entries in each row has full rank (i.e., equal to $m\wedge n$) \whp\
Thus, the random matrix model is just the one from Example~\ref{Ex_XOR} with $k=3$.
Dubois and Mandler pinpointed the precise full row  rank threshold $m/n\approx2.75$.
The proof relies on the first moment method applied to $|\ker\vA|$, which boils down to a one-dimensional calculus problem.
Matters get more complicated when one considers a greater number $k>3$ of non-zero entries per row.
This more general problem, known as random $k$-XORSAT, was solved independently by Dietzfelbinger et al.~\cite{Dietzfelbinger} and by Pittel and Sorkin~\cite{PittelSorkin} via technically demanding moment calculations.
Unfortunately, considering fields $\FF_q$ with $q>2$ complicates the moment calculation even further.
Yet undertaking a computer-assisted tour-de-force Falke and Goerdt~\cite{GoerdtFalke} managed to extend the method to $\FF_3$.
However, extending this strategy to infinite fields is a non-starter as $|\ker\vA|$ may be infinite.

In a previous paper Ayre, Coja-Oghlan, Gao and M\"uller~\cite{Ayre} applied the Aizenman-Sims-Starr scheme to the study of sparse random matrices with precisely $k$ non-zero entries per row as in Example~\ref{Ex_XOR}, over finite fields.
The present paper goes beyond that earlier contribution in two crucial ways.
First, we develop a far more delicate coupling scheme that accommodates general degree distributions $\vd,\vk$ rather than just the Poisson-constant degrees from Example~\ref{Ex_XOR},
including degree sequences for which the 2-core bound fails to be tight (in contrast to Example~\ref{Ex_XOR}).
Apart from rendering a proof of Lelarge's conjecture, we expect that this more general coupling scheme will find further uses in the theory of random factor graphs; for instance, it seems applicable to generalisations of the models from~\cite{CKPZ}.

Second, the rank calculation in~\cite{Ayre} is based on a probabilistic view that does not extend to infinite fields.
Indeed, the proof there is based on a close study of a uniformly random element $\SIGMA$ of the kernel of the random matrix $\vA$.
Specifically, \cite[\Lem~3.1]{Ayre} analyses the impact of the perturbation from Definition~\ref{Def_pin} on a matrix $A\in\FF^{m\times n}$ for a finite field $\FF$.
With $\SIGMA=(\SIGMA_1,\ldots,\SIGMA_n)\in\ker(A)$ a uniformly random element of the kernel, the lemma shows that for a large enough $\cT=\cT(\delta,\FF)>0$ and a uniformly random $0\leq\THETA\leq\cT$,
\begin{align}\label{eqstoch}
\sum_{\substack{1\leq i<j\leq n\\\omega,\omega'\in\FF}}
	\Erw\abs{\pr\brk{\SIGMA_i=\omega,\SIGMA_j=\omega'\mid A[\THETA]}-
		\pr\brk{\SIGMA_i=\omega\mid A[\THETA]}\cdot\pr\brk{\SIGMA_j=\omega'\mid A[\THETA]}}&<\delta n^2,
\end{align}
As in \Prop~\ref{Prop_Alp}, the necessary value of $\cT$ is independent of $n,m$ and~$A$.
Thus, the random perturbation renders the vector entries $(\SIGMA_i,\SIGMA_j)$ nearly stochastically independent, for most $i,j$.
Thanks to general results from~\cite{Victor}, \eqref{eqstoch} extends from pairwise independence to $\ell$-wise independence, albeit with a weaker error bound $\delta$.
The result \cite[\Lem~3.1]{Ayre} was inspired by general statements about probability measures on discrete cubes from~\cite{CKPZ,Montanari,Raghavendra}.

Inherently, this stochastic approach does not generalise to infinite fields, where, for starters, it does not even make sense to speak of a uniformly random element of the kernel.
That is why here we replace the stochastic approach from the earlier paper by the more versatile algebraic approach summarised in \Prop~\ref{Prop_Alp}, which are applicable to any field---say, the reals, the field $\mathbb Q_p$ of $p$-adic numbers, the algebraic closure of a finite field or a structure as complex as a function field.
Instead of showing stochastic independence, \Prop~\ref{Prop_Alp} renders linear independence amongst most bounded-size subsets of coordinates.
Apart from being more general, this algebraic viewpoint allows for a cleaner, more direct proof of the rank formula.
Additionally, on finite fields the stochastic independence \eqref{eqstoch} follows from the linear independence provided by \Prop~\ref{Prop_Alp}, with a significantly improved bound on $\cT(\delta)$.
We work this out in detail in Appendix~\ref{Apx_0pinning}.

The single prior contribution on the rational rank of sparse random matrices is due to Bordenave, Lelarge and Salez~\cite{BLS}, who computed the rational rank of the (symmetric) adjacency matrix of a random graph with a given vertex degree distribution.
The proof is based on local weak convergence and the `objective method'~\cite{AldousSteele}.
An intriguing question for future research is to extend the techniques from the present paper to symmetric random matrices.

\subsubsection{The cavity method (and its caveats)}\label{Sec_cavity}
On the basis of the cavity method, an analytic but non-rigorous technique inspired by the statistical mechanics of disordered systems, 
it had been predicted erroneously that over finite fields the 2-core bound \eqref{eq2corebound} on the rank of $\vA$ is universally tight for general degree distributions $\vd,\vk$~\cite{AlaminoSaad,MM}.
Where did the cavity method go astray?

The method comes in two instalments, the simpler {\em replica symmetric ansatz} and the more elaborate {\em one-step replica symmetry breaking ansatz} (`1RSB').
The former predicts that the rank of $\A$  over a finite field $\FF_q$ converges in probability to the solution of an optimisation problem on an  infinite-dimensional space of probability measures.
To be precise, let $\cP(\FF_q)$ be the space of probability measures on $\FF_q$.
Identify this space with the standard simplex in $\RR^q$.
Further, let $\cP^2(\FF_q)$ be the space of all probability measures on $\cP(\FF_q)$.
Given $\pi\in\cP^2(\FF_q)$ let $(\MU_{i,j}^\pi)_{i,j\geq1}$ be a sequence of independent samples from $\pi$. 
Recalling \eqref{eqSizeBiasd},  the {\em Bethe free entropy} is defined by
\begin{align*}
\cB(\pi)=\Erw&\brk{\log_q\sum_{\sigma_1\in\FF_q}\prod_{i=1}^{\vd}\sum_{\sigma_2,\ldots,\sigma_{\hat\vk_i}\in\FF_q}
		\vecone\cbc{\sum_{j=1}^{\hat\vk_i}\sigma_j\CHI_{i,j}=0}\prod_{j=2}^{\hat\vk_i}\MU_{i,j}^\pi(\sigma_j)}\\
	&-\frac{d}k\Erw\brk{(\vk-1)\log_q\sum_{\sigma_1,\ldots,\sigma_{\vk}\in\FF_q}
			\vecone\cbc{\sum_{i=1}^{\vk}\sigma_i\CHI_{1,i}=0}
		\prod_{i=1}^{\vk}\MU_{1,i}^\pi(\sigma_i)}.&&\mbox{(cf.~\cite[\Chap~14]{MM}).}
\end{align*}
The replica symmetric ansatz predicts that
\begin{align}\label{eqRS}
\lim_{n\to\infty}\frac1n\nul\A&=\sup_{\pi\in\cP^2(\FF_q)}\cB(\pi)&\mbox{in probability.}
\end{align}

For a detailed (heuristic) derivation of the Bethe free entropy and the prediction~\eqref{eqRS} we refer to~\cite{AlaminoSaad}.
But let us briefly comment on the intended semantics of $\pi$.
Consider the Tanner graph $\G$ representing $\A$.
Suppose that variable node $x_i$ and check node $a_j$ are adjacent.
Then for $\sigma\in\FF_q$ we define the {\em Belief Propagation message} $\mu_{\A,x_j\to a_i}(\sigma)$ from $x_j$ to $a_i$ as follows.
Obtain $\A_{x_j\to a_i}$ from $\A$ by changing the $ij$-th matrix entry to zero; this corresponds to deleting the $x_j$-$a_i$-edge from the Tanner graph.
Then $\mu_{\A,x_j\to a_i}(\sigma)$ is the probability that in a uniformly random vector $\SIGMA\in\ker\A_{x_j\to a_i}$ we have $\SIGMA_j=\sigma$.
Further, define $\pi_{\A}$ as the empirical distribution of the $\mu_{\A,x_j\to a_i}$ over the edges of the Tanner graph:
\begin{align*}
\pi_{\A}&=\frac1{\sum_{i=1}^n\vd_i}\sum_{j=1}^n\sum_{i=1}^{\vm}\vecone\{\A_{ij}\neq0\}\delta_{\mu_{\A,x_j\to a_i}}\in\cP^2(\FF_q).
\end{align*}
Then the replica symmetric ansatz predicts that $\pi_{\A}$ is asymptotically a maximiser of the Bethe free energy, i.e.,
$\sup_{\pi\in\cP^2(\FF_q)}\cB(\pi)=\cB(\pi_{\A})+o_n(1)$ \whp\
Thus, the maximiser $\pi$ in \eqref{eqRS} is deemed to encode the Belief Propagation messages on the edges of the Tanner graph of $\vA$.

A bit of linear algebra that seems to have gone unnoticed in the physics literature reveals that the messages actually have a very special form~\cite[\Lem~2.3]{Ayre}.
Namely, any message $\mu_{\A,x_j\to a_i}$ is either the uniform distribution $q^{-1}\vecone$ on $\FF_q$ or the atom $\delta_0$ on $0$.
In effect, the rank should come out as the Bethe free entropy $\cB(\pi_\alpha)$ of a convex combination
\begin{align}\label{eqpialpha}
\pi_\alpha&=\alpha\delta_{\delta_0}+(1-\alpha)\delta_{q^{-1}\vecone}&&(\alpha\in[0,1]).
\end{align}
In fact, a simple calculation yields $\Phi(\alpha)=\cB(\pi_\alpha)$ for all $\alpha\in[0,1]$.
Thus, \Thm~\ref{thm:rank} shows that
\begin{align*}
\lim_{n\to\infty}\frac{\rk\vA}{n}&=1-\sup_{\alpha\in[0,1]}\cB(\pi_\alpha)&&\mbox{in probability},
\end{align*}
vindicating the cavity method to an extent.
However, we do not know whether the Bethe free entropy possesses other spurious maximisers
$\pi\in\cP^2(\FF_q)$ with $\cB(\pi)>\sup_{\alpha\in[0,1]}\cB(\pi_\alpha)$.

Alamino and Saad~\cite{AlaminoSaad} tackled the optimisation problem~\eqref{eqRS} by means of a numerical heuristic called population dynamics, without noticing the restriction to $(\pi_\alpha)_{\alpha\in[0,1]}$.
In all the examples that they studied they found that $\pi\in\{\pi_0,\pi_\rho\}$, with $\rho$ from \eqref{def:rho};
	in fact, all their examples fall within the purview of \Thm~\ref{Thm_tight}.%
	\footnote{Strictly speaking, Alamino and Saad, who worked numerically with $n$ in the hundreds, reported $\pi\in\{\pi_0,\pi_1\}$. Indeed, $\rho\in\{0,1\}$ in the first class of examples that they studied, but not in the other two. For instance, in their example (3) the actual value of $\rho$ is either $0$ or a number strictly smaller than one, although $\rho>0.97$ whenever $\Phi(\rho)>\Phi(0)$.}
This led Alamino and Saad to conjecture that the maximiser $\pi$ is generally of this form, although they cautioned that further evidence seems necessary.
Example~\ref{Ex_3} and \cite{Lelarge} provide counterexamples.
The more sophisticated 1RSB cavity method is presented in \cite[\Chap~19]{MM}, where
an exercise asks the reader to verify that the 2-core bound is tight (over finite fields).
While \Thm~\ref{Thm_tight} gives sufficient conditions for this to be correct, the aforementioned counterexamples apply.

\subsection{Organisation}
We proceed to prove \Prop~\ref{Prop_Alp}, the `key lemma' upon which the proof of \Thm~\ref{thm:rank} rests, in \Sec~\ref{Sec_Alp}.
Subsequently in \Sec~\ref{Sec_conc} we use concentration inequalities and the local limit theorem for sums of independent random variables to prove \Prop~\ref{Cor_lower}.
Additionally, \Sec~\ref{Sec_conc} contains \Prop~\ref{Lemma_welldef}, which shows that the random matrix model \eqref{eqWellDef1} is well defined, a standard argument that we include for the sake of completeness.
Dealing with the full details of the coupling scheme, \Sec~\ref{Sec_lower} contains the proof of  \Prop~\ref{Prop_coupling}.
Further, \Sec~\ref{sec:2core} deals with the proof of \Thm~\ref{thm:2core} and in \Sec~\ref{Sec_tight} we prove \Thm~\ref{Thm_tight}.
For the sake of completeness a proof of \Lem~\ref{Lemma_sums} is included in Appendix~\ref{Sec_sums}.
Moreover, in Appendix~\ref{Apx_0pinning} we elaborate on the relation between the algebraic perturbation from \Prop~\ref{Prop_Alp} and the stochastic version from~\cite{Ayre}.
Finally, Appendix~\ref{Sec_interp} contains a self-contained proof of the upper bound on the rank for \Thm~\ref{thm:rank} via the interpolation method from mathematical physics.

\section{Linear relations: proof of \Prop~\ref{Prop_Alp}}\label{Sec_Alp}

\noindent
In this section we prove \Prop~\ref{Prop_Alp} and \Lem~\ref{Cor_free}.
The somewhat delicate proof of the former is based on a blend of probabilistic and algebraic arguments.
The proof of the latter is purely algebraic and fairly elementary.

\subsection{Proof of \Prop~\ref{Prop_Alp}}
\aco{Observe that \Prop~\ref{Prop_Alp} is not an asymptotic statement to the extent that we need to exhibit a function $\cT=\cT(\delta,\ell)$ such that \eqref{eqAlp} holds for all matrices $A$ (ultimately in \eqref{eqTdeltaell} we will see that $\cT(\delta,\ell)$ scaling as $\ell^3/\delta^4$ does the trick).
Nevertheless, letting $n$ denote the number of columns of $A$, we may safely assume that $n>n_0=n_0(\delta,\ell)$ for any specific $n_0$ that depends on $\delta,\ell$ only.
Indeed, to deal with $n\leq n_0$ for any fixed value $n_0$ we could just pick $\cT\geq\cT_0(\delta,\ell)$ for a large enough $\cT_0(\delta,\ell)$ so that with probability at least $1-\eps$ we have $\{\vec i_1,\ldots,\vec i_{\vec\theta}\}=[n]$.
Note that we do not need to worry about the possibility that $\cT>n$ because the $\vec i_j$ are drawn with replacement.
Further, if $\{\vec i_1,\ldots,\vec i_{\vec\theta}\}=[n]$, then a glimpse at Definition~\ref{Def_Alp} shows that {\em all} coordinates are frozen.
Therefore, $A[\vec\theta]$ is $(\delta,\ell)$-free.
Hence, from now on we assume that $n\geq n_0=n_0(\delta,\ell)$ for a sufficiently large $n_0$.
}

Given any matrix $M$ we define a {\em minimal $h$-relation of $M$} as a relation $I$  of $M$ of size $|I|=h$ that does not contain a proper subset that is a relation of $M$.
Let $\cR_h(M)$ be the set of all minimal $h$-relations of $M$ and set $R_h(M)=|\cR_h(M)|$.
Thus, $R_1(M)=|\fF(M)|$ is just the number of frozen variables of $M$.
Additionally, let $\cR_{\leq h}(M)=\bigcup_{1\leq i\leq h}\cR_i(M)$ and $R_{\leq h}(M)=|\cR_{\leq h}(M)|$.
Let $\vi_1,\vi_2,\vi_3,\ldots\in[n]$ be uniformly distributed independent random variables.

The proof of \Prop~\ref{Prop_Alp} is based on a potential function argument.
To get started we observe that
\begin{align}\label{eqProp_Alp_1}
\cR_1(A[t])\subset\cR_1(A[t+1])&&\mbox{ for all $t\geq0$}.
\end{align}
This inequality implies that the random variable
\begin{align*}
\Delta_t&=\frac{\Erw[R_1(A[t+\ell])\mid A[t]]-R_1(A[t])}{n}
\end{align*}
is non-negative.
The random variable $\Delta_t$ gauges the increase in frozen variables upon addition of $\ell$ more rows that expressly freeze specific variables.
Thus, `big' values of $\Delta_t$, say $\Delta_t=\Omega_n(1)$, witness a kind of instability as pegging a few variables to zero entails that another $\Omega_n(n)$ variables get frozen to zero due to implicit linear relations.
We will exploit the observation that, since $\Delta_t\in[0,1]$ and $\Erw[\Delta_t]$ is monotonically increasing in $t$, such instabilities cannot occur for many $t$.
Thus, the expectation $\Erw[\Delta_{\THETA}]$ will serve as our potential.
A similar potential was used in~\cite{Ayre} to study stochastic dependencies in the case of finite fields $\FF$.
But in the present more general context the analysis of the potential is significantly more subtle.
The following lemma puts a lid on the potential.

\begin{lemma}\label{Claim_Alp1}
	We have $\Erw[\Delta_{\THETA}]\leq\ell/\cT$.
\end{lemma}
\begin{proof}
	For any $r\in\{0,1,\ldots,\ell-1\}$ we have
	\begin{align}\label{eqClaim_Alp1_0}
		\sum_{j\geq0}\Erw[\Delta_{r+j\ell}]&=\frac{1}{n}\sum_{j\geq0}\Erw[R_1(A[r+(j+1)\ell])]-\Erw[R_1(A[r+j\ell])]\leq\frac{1}{n}\lim_{j\to\infty}\Erw[R_1(A[r+j\ell])].
	\end{align}
	\aco{Observe that there is no problem here taking the limit $j\to\infty$ as the coordinates $\vec i_j$ from Definition~\ref{Def_pin} are chosen independently with replacement.
	In the case that $j\gg n$ the likely outcome is thus that {\em all} coordinates of $A[r+j\ell]$ are frozen, which is why $\lim_{j\to\infty}\Erw[R_1(A[r+j\ell])]=n$.}
	Hence, \eqref{eqClaim_Alp1_0} yields
	\begin{align}\label{eqClaim_Alp1_00}
		\sum_{j\geq0}\Erw[\Delta_{r+j\ell}]&\leq1.
	\end{align}
	Summing \eqref{eqClaim_Alp1_00} on $r$, we obtain
	\begin{align}\label{eqProp_Alp_2}
		\sum_{\theta\in[\cT]}\Erw[\Delta_\theta]&\leq\sum_{r=0}^{\ell-1}\sum_{j\geq0}\Erw[\Delta_{r+j\ell}]\leq\ell.
	\end{align}
\aco{Since $\THETA\in[\cT]$ is chosen uniformly and independently of everything else, dividing \eqref{eqProp_Alp_2} by $\cT$ yields
\begin{align}\label{eqProp_Alp_3}
	\Erw[\Delta_{\THETA}]&=\frac1{\cT}\sum_{\theta\in[\cT]}\Erw[\Delta_\theta]\leq\frac1{\cT}\sum_{r=0}^{\ell-1}\sum_{j\geq0}\Erw[\Delta_{r+j\ell}]\leq\frac\ell\cT,
\end{align}
as desired.}
\end{proof}

\aco{\begin{remark}
	\Lem~\ref{Claim_Alp1} provides a bound on the mean of $\Erw[\Delta_{\THETA}]$ for a random $\vec\theta$.
	The requirement that $\vec\theta$ be random stems from the fact that the proof is based on an averaging argument.
	It is an open question whether this random value could be replaced by a deterministic value, and whether the choice of such a deterministic value would have to depend on $A$.
\end{remark}}

\noindent
The following lemma shows that unless $A[t]$ is $(\delta,\ell)$-free, there exist many minimal $h$-relations for some $2\leq h\leq\ell$.

\begin{lemma}\label{Claim_Alp2a}
If $A[t]$ fails to be $(\delta,\ell)$-free then there exists $2\leq h\leq\ell$ such that $R_h(A[t])\geq\delta n^h/\ell$.
\end{lemma}
\begin{proof}
Assume that
\begin{align}\label{eqClaim_Alp2_1}
R_h(A[t])&<\delta n^h/\ell&\mbox{for all }2\leq h\leq\ell.
\end{align}
Since every proper relation $I$ of size $|I|=\ell$ contains a minimal $h$-relation $J\subset I$ for some $2\leq h\leq\ell$, \eqref{eqClaim_Alp2_1} implies that $A[t]$ possesses fewer than $\delta n^\ell$ proper relations of size $\ell$ in total.
Hence, if \eqref{eqClaim_Alp2_1} holds, then $A[t]$ is $(\delta,\ell)$-free.
\end{proof}

\noindent
As a next step we show that $\Delta_t$ is large if $A[t]$ possesses many minimal $h$-relations for some $2\leq h\leq\ell$.

\begin{lemma}\label{Claim_Alp2}
If $R_h(A[t])\geq\delta n^h/\ell$ for some $2\leq h\leq\ell$, then $\Delta_t\geq\delta^2/\ell^2$.
\end{lemma}
\begin{proof}
Let $\cR_{v,h}(A[t])$ be the set of all relations $I\in\cR_h(A[t])$ that contain $v\in[n]$ and set $r_{v,t,h}=|\cR_{v,h}(A[t])|$.
Moreover, let $\cV_{t,h}$ be the set of all $v\in[n]$ with $r_{v,t,h}\geq \delta hn^{h-1}/(2\ell)$.
We assumed $|R_h(A[t])|\geq\delta n^h/\ell$, and every $h$-relation is affiliated with an $h$-element subset of $[n]$.
Consequently,
\begin{align*}
\delta hn^h/\ell\leq hR_h(A[t]) \leq |\cV_{t,h}|n^{h-1}+\bc{n-|\cV_{t,h}|}\cdot\delta hn^{h-1}/(2\ell),
\end{align*}
whence
\begin{align}\label{eqClaim_Alp2_3}
|\cV_{t,h}|\geq\frac{\delta h n}{2\ell}.
\end{align}

Consider $v\in \cV_{t,h}$ along with a minimal $h$-relation $I\in\cR_{v,h}(A[t])$.
If $I=\{v,\vi_{t+1},\ldots,\vi_{t+h-1}\}$, i.e., $I$ comprises $v$ and the next $h-1$ indices that get pegged, then $v\in\fF(A[t+h-1])$.
Indeed, since $I$ is a minimal $h$-relation of $A[t]$ there is a row vector $y$ such that $\supp(yA[t])=I$.
Hence, if $I\setminus\cbc v=\cbc{\vi_{t+1},\ldots,\vi_{t+h-1}}$, then we can extend $y$ to a row vector $y'$ such that $\supp(y'A[t+\ell])=\cbc v$, and thus $v\in\fF(A[t+h-1])$.
Furthermore, since $(\vi_{t+1},\ldots,\vi_{t+h-1})\in[n]^{h-1}$ is uniformly random, we conclude that
\begin{align}\label{eqClaim_Alp2_4}
\pr\brk{I=\cbc{v, \vi_{t+1},\ldots,\vi_{t+h-1}}\mid A[t]}&=(h-1)!/n^{h-1}\geq n^{1-h}.
\end{align}
Now, because every $v\in\cV_{t,h}$ satisfies $r_{v,t,h}\geq \delta hn^{h-1}/(2\ell)$, \eqref{eqClaim_Alp2_4} implies that
\begin{align}\label{eqClaim_Alp2_5}
\pr\brk{v\in\fF(A[t+h-1])\mid A[t]}&\geq r_{v,t,h}/n^{h-1}\geq  \delta h/(2\ell).
\end{align}
We also notice that $\cV_{t,h}\cap\fF(A[t])=\emptyset$ because no minimal $h$-relation contains a frozen variable.
Therefore, combining \eqref{eqProp_Alp_1}, \eqref{eqClaim_Alp2_3} and \eqref{eqClaim_Alp2_5} and using linearity of expectation, we obtain
\begin{align*}
\Delta_t&\geq\frac{1}{n}\sum_{v\in \cV_{t,h}}\pr\brk{v\in\fF(A[t+h-1])\mid A[t]}\geq
\frac{\delta h|\cV_{t,h}|}{2\ell n}\geq\frac{\delta^2 h^2}{4\ell^2}\geq\frac{\delta^2}{\ell^2},
\end{align*}
as desired.
\end{proof}

\noindent
Combining \Lem s~\ref{Claim_Alp2a} and~\ref{Claim_Alp2}, we immediately obtain the following.

\begin{corollary}\label{Claim_Alp3}
If $A[t]$ fails to be $(\delta,\ell)$-free then $\Delta_t\geq\delta^2/\ell^2$.
\end{corollary}

\noindent
We have all the ingredients in place to complete the proof of \Prop~\ref{Prop_Alp}.

\begin{proof}[Proof of \Prop~\ref{Prop_Alp}.]
We define $T=\cbc{t\in[\cT]:\pr\brk{A[t]\mbox{ fails to be $(\delta,\ell)$-free}}\geq\delta/2}$
so that
\begin{align}\label{eqpfProp_Alp1}
\pr\brk{A[\THETA]\mbox{ is $(\delta,\ell)$-free}}&> 1-\delta/2-\pr\brk{\THETA\in T}.
\end{align}
Hence, we are left to estimate $\pr\brk{\THETA\in T}$.
Applying \Cor~\ref{Claim_Alp3}, we obtain for every $t\in T$,
\begin{align}\label{eqpfProp_Alp2}
\Erw[\Delta_t]&\geq\frac{\delta^2}{\ell^2}\cdot\pr\brk{A[t]\mbox{ fails to be $(\delta,\ell)$-free}}
	\geq\frac{\delta^3}{2\ell^2}.
\end{align}
Moreover, averaging \eqref{eqpfProp_Alp2} on $t\in[\cT]$ and applying \Lem~\ref{Claim_Alp1}, we obtain
\begin{align*}
\frac{\delta^3}{2\ell^2}\cdot\pr\brk{\THETA\in T}
	&=\frac{\delta^3}{2\ell^2}\cdot\frac{|T|}{\cT}
	\leq\frac{1}{\cT}\sum_{t\in T}\Erw[\Delta_t]\leq\Erw[\Delta_{\THETA}]\leq\frac\ell\cT.
\end{align*}
Consequently, choosing 
	\begin{align}\label{eqTdeltaell}\cT>4\ell^3/\delta^4\end{align}
ensures  $\pr\brk{\THETA\in T}\leq\delta/2$.
Thus, the assertion follows from \eqref{eqpfProp_Alp1}.
\end{proof}

\begin{remark}\label{Prop_Alp++}
The proof presented in this section actually renders a slightly stronger statement than \Prop~\ref{Prop_Alp}.
Specifically, let $A$ be an $m\times N$-matrix and let $n\leq N$.
Obtain $A[\theta,n]$ by pegging $\theta$ random variables from among the first $n$ variables $x_1,\ldots,x_n$ of the linear system $Ax=0$ to zero.
Then with $\THETA=\THETA(\delta,\ell)$ chosen as in \Prop~\ref{Prop_Alp} we find that with probability at least $1-\delta$, there are no more than $\delta n^\ell$ proper relations $I\subset[n]$.
Thus, in order to rid a subset of the columns of short linear relations, it suffices to peg $\THETA$ random variables from that subset to zero.
The proof of this stronger statement proceeds as above, except that we confine ourselves to minimal relations among the first $n$ columns.
\end{remark}

\subsection{Proof of \Lem~\ref{Cor_free}}
We are going to derive \Lem~\ref{Cor_free} from the following simpler, \aco{deterministic and non-asymptotic} statement.

\begin{lemma}\label{Lemma_free}
	\aco{Let $m,n,m',n'\geq1$ be integers.}
	Let $A$ be an $m \times n$ matrix, let $B$ be an $m'\times n$ matrix and let $C$ be an $m'\times n'$ matrix.
	Let $I\subset[n]$ be the set of all indices of non-zero columns of $B$.
	Unless $I$ is a relation of $A$ we have
	\begin{align*}
		\nul A-\nul\begin{pmatrix}A&0\\B&C\end{pmatrix}=\rk(B\ C)-n'.
	\end{align*}
\end{lemma}
\begin{proof}
	Suppose that $I$ is not a relation of $A$.
	We begin by showing that
	\begin{align}\label{eqLemma_free_1}
		\nul A-\nul\begin{pmatrix}A\\B\end{pmatrix}=\rk(B).
	\end{align}
	Writing $B_1,\ldots,B_{m'}$ for the rows of $B$ and $r=\rk(B)$ for the rank and applying a row permutation if necessary, we may assume that $B_1,\ldots,B_r$ are linearly independent.
	Hence, to establish \eqref{eqLemma_free_1} it suffices to prove that for all $0\leq \ell<r$,
	\begin{align}\label{eqLemma_free_2}
		\rk\begin{pmatrix}A\\B_1\\\vdots\\B_\ell\end{pmatrix}<\rk\begin{pmatrix}A\\B_1\\\vdots\\B_{\ell+1}\end{pmatrix}.
	\end{align}
	In other words, we need to show that $B_{\ell+1}$ does not belong to the space spanned by  $B_1,\ldots,B_\ell$ and the rows $A_1,\ldots,A_m$ of $A$.
	Indeed, assume that $B_{\ell+1}=\sum_{i=1}^\ell x_iB_i+\sum_{i=1}^m y_iA_i$.
	Then $0\neq B_{\ell+1}-\sum_{i=1}^\ell x_iB_i=\sum_{i=1}^m y_iA_i$ and thus $\emptyset\neq\supp\sum_{i=m}^\ell y_iA_i\subset I$, in contradiction to the assumption that $I$ is no relation of $A$.
	Hence, we obtain \eqref{eqLemma_free_2} and thus \eqref{eqLemma_free_1}.
	Finally, to complete the proof of \eqref{eqLemma_free} we apply \eqref{eqLemma_free_1} to the matrices $(A\ 0)$ and $(B\ C)$, obtaining
	\begin{align*}
		\nul(A)+n'-\nul\begin{pmatrix}A&0\\B&C\end{pmatrix}&=\nul(A\ 0)-\nul\begin{pmatrix}A&0\\B&C\end{pmatrix}=\rk(B\ C),
	\end{align*}
	as desired.
\end{proof}

\begin{proof}[Proof of \Lem~\ref{Cor_free}]
\aco{Recall tha $A$ has size $m\times n$.
By Definition \ref{Def_Alp} a coordinate $i$ is frozen iff the vector $e^{(i)}\in\FF^{1\times n}$ whose $i$-th entry equals one and whose other entries equal zero can be written as a linear combination of the rows of $A$.
	For every $i\in\fF(A)$ we can therefore apply elementary row operations (like in Gaussian elimination) to zero out the entire $i$-column of $B$.
	Since elementary row operations do not alter the nullity of a matrix, we therefore obtain the idenitity
	\begin{align*}
		\nul\begin{pmatrix}A&0\\B&C\end{pmatrix}&=\nul\begin{pmatrix}A&0\\B_*&C\end{pmatrix}.
	\end{align*}
The assertion thus follows from \Lem~\ref{Lemma_free}.}
\end{proof}

\section{Concentration}\label{Sec_conc}

\noindent
The principal aim of this section is to prove \Prop~\ref{Cor_lower}, i.e., to argue that the rank of the actual matrix $\vA_n$ that does not have any cavities and whose Tanner graph is simple is close to the expected rank of $\vA_{\eps,n}$ \whp\
In other words, we need to show that the rank of a random matrix is sufficiently concentrated that conditioning on
	\begin{align*}
	\cD=\cbc{\sum_{i=1}^n\vd_i=\sum_{i=1}^{\vm}\vk_i}
	\end{align*}
and on the event $\cS$ that the Tanner graph is simple is inconsequential.
The main tool will be the following local limit theorem for sums of independent random variables, which we use in \Sec~\ref{Sec_llt} to calculate the probability of $\cD$.

\begin{theorem}[{\cite[p.~130]{Durrett}}]\label{Thm_llt}
Suppose that $(\vX_i)_{i\geq1}$ is a sequence of i.i.d.\ variables that take values in $\ZZ$ such that the greatest common divisor of the support of $\vX_1$ is one.
Also assume that $\Var[\vX_1]=\sigma^2\in(0,\infty)$.
Then
\begin{align*}
\lim_{n\to\infty}\sup_{z\in\ZZ}\abs{\sqrt n\pr\brk{\sum_{i=1}^n\vX_i=z}-\frac{\exp(-(z-n\Erw[\vX_1])^2/(2n\sigma^2))}{\sqrt{2\pi}\sigma}}&=0.
\end{align*}
\end{theorem}

\noindent
 Subsequently, 
in \Sec~\ref{Sec_welldef} we calculate the probability of the event $\cS$, proving \Prop~\ref{Lemma_welldef} along the way.
Finally, in \Sec~\ref{Sec_Cor_lower} we complete the proof of \Prop~\ref{Cor_lower}.

\subsection{The event $\cD$}\label{Sec_llt}

\noindent
Because $\Erw[\vd^r]+\Erw[\vk^r]<\infty$ for an $r>2$, the event
\begin{align}\label{eqM}
\cM=\cbc{\max_{i\in[n]}\vd_i+\max_{i\in[\vm]}\vk_i\leq \sqrt n/\log^9 n}\qquad\mbox{satisfies}\qquad
\pr\brk{\cM}=1-o_n(1).
\end{align}
As an application of \Thm~\ref{Thm_llt} we obtain the following estimate.

\begin{lemma}\label{Cor_llt}
If $\gcd(\vk)$ divides $n$, then \jtd{$\pr\brk{\cD}=\Theta_n(n^{-1/2})$ and} $\pr\brk{\cD\mid\cM}=\Theta_n(n^{-1/2})$.
\end{lemma}
\begin{proof} \jtd{For $\pr\brk{\cD\mid\cM}$}
there are several cases to consider.  
First, that $\Var(\vd)=\Var(\vk)=0$, i.e., $\vd,\vk$ are both atoms.
Since $\vm$ is a Poisson variable with mean $dn/k$ we find $\pr\brk{\cD\mid\cM}=\pr\brk{\vm=d n/k}=\Theta_n(n^{-1/2})$.

Second, suppose that $\Var(\vd)>0$ but $\Var(\vk)=0$.
Then \Thm~\ref{Thm_llt} 
and~\eqref{eqM} show that
\begin{align}\label{eqCor_llt1}
\pr\brk{\abs{dn-\sum_{i=1}^n\vd_i}\leq\sqrt n\wedge k\mbox{ divides }\sum_{i=1}^n\vd_i\mid\cM}&=\Omega_n(1).
\end{align}
Further, given $\abs{dn-\sum_{i=1}^n\vd_i}\leq\sqrt n$ and given $k$ divides $\sum_{i=1}^n\vd_i$, the event $k\vm=\sum_{i=1}^n\vd_i$
has probability $\Theta_n(n^{-1/2}) $ by the local limit theorem for the Poisson distribution.

The case that $\Var(\vd)=0$ but $\Var(\vk)>0$ can be dealt with similarly.
Indeed, pick a large enough number $L>0$ and let $I=\cbc{i\in[\vm]:\vk_i>L}$, $\vm'=|I|$, $\vm''=\vm-|I|$, $S'=\sum_{i\in I}\vk_i$
and $S''=\sum_{i\in[\vm]\setminus I}\vk_i$.
Then $\vm',\vm''$ are stochastically independent, as are $S',S''$.
Moreover, since $S'$ satisfies the central limit theorem we have
\begin{align}\label{eqCor_llt1a}
\pr\brk{|S'-\Erw[S'\mid\cM]|\leq\sqrt n\mid\cM}&=\Omega_n(1).
\end{align}
Further, \Thm~\ref{Thm_llt} applies to $S''$, which is distributed as $\sum_{i=1}^{\vm}\vk_i\vecone\{\vk_i\leq L\}$.
Hence, as $n$ is divisible by $\gcd(\vk)$, for large enough $L$ we have
\begin{align}\label{eqCor_llt1b}
\pr\brk{S'+S''=dn\mid|S'-\Erw[S'\mid\cM]|\leq\sqrt n,\cM }&=\Omega_n(n^{-1/2}).
\end{align}
Thus, \eqref{eqCor_llt1a} and \eqref{eqCor_llt1b} show that $\pr\brk{\cD\mid\cM}=\Omega_n(n^{-1/2})$.
The upper bound $\pr\brk{\cD\mid\cM}=O_n(n^{-1/2})$ follows from the uniform upper bound from \Thm~\ref{Thm_llt}.

A similar argument applies in the final case $\Var(\vd)>0$, $\Var(\vk)>0$.
Indeed, \Thm~\ref{Thm_llt} and~\eqref{eqM} yield
\begin{align}\label{eqCor_llt3}
\pr\brk{\gcd(\vk)\mbox{ divides }\sum_{i=1}^{n}\vd_i\mbox{ and }\abs{dn-\sum_{i=1}^{n}\vd_i}\leq\sqrt n\mid\cM}=\Omega_n(1).
\end{align}
Moreover, \eqref{eqCor_llt1a} remains valid regardless the variance of $\vd$.
Hence, applying \Thm~\ref{Thm_llt} to $S''$, we obtain
\begin{align}\label{eqCor_llt1c}
\pr\brk{S'+S''\negmedspace=\negmedspace\sum_{i=1}^n\vd_i\,\bigg|\,
\gcd(\vk)\mbox{ divides }\negmedspace\sum_{i=1}^{n}\vd_i,\,\abs{dn-\sum_{i=1}^{n}\vd_i}\leq\negmedspace\sqrt n,\,
|S'-\Erw[S'\mid\cM]|\leq\sqrt{n},\cM }\negmedspace=\Omega_n(n^{-1/2}).
\end{align}
Combining~\eqref{eqCor_llt3} and~\eqref{eqCor_llt1c}, we see that $\pr\brk{\cD\mid\cM}=\Omega_n(n^{-1/2})$.
The matching upper bound $\pr\brk{\cD\mid\cM}=O_n(n^{-1/2})$ follows from the universal upper bound from \Thm~\ref{Thm_llt} once more. 
The treatment of the unconditional $\pr\brk{\cD}$ is similar but slightly simpler.
\end{proof}

\subsection{The event $\cS$}\label{Sec_welldef}
The random matrix $\vA_n$ for \Thm~\ref{thm:rank} is identical in distribution to the random matrix $\vA_{0,n}$ with $\eps=0$ conditioned on the event $\cD$ and on the event $\cS$ that the Tanner graph $\G_{0,n}$ does not contain any multi-edges.
Therefore, \Prop~\ref{Lemma_welldef} is going to be a consequence of \Lem~\ref{Cor_llt} and the following statement.

\begin{lemma}\label{Lemma_simple}
We have $\pr\brk{\vA_{0,n}\in\cS|\cD}=\Omega_n(1)$.
\end{lemma}

\noindent
We proceed to prove \Lem~\ref{Lemma_simple}.
Recall the event $\cM$ from \eqref{eqM}.
The proof of \Lem~\ref{Lemma_simple} is essentially based on the routine approach of showing by way of a moment calculation that the number of multi-edges of $\G_{0,n}$ is asymptotically Poisson with a finite mean.
This argument has been carried out illustratively for the case of random regular graphs in~\cite[\Chap~9]{JLR}.
But since here we work with very general degree distributions, technical complications arise.
For instance, as a first step we need to estimate the empirical variance of the degree sequences.

\begin{claim}\label{Lemma_sm}
On the event $\cD\cap\cM$ we have
$\frac1n\sum_{i=1}^n\vd_i^2\to \Erw[\vd^2]$, $\frac1n\sum_{i=1}^{\vm}\vk_i^2\to d\Erw[\vk^2]/k$ in probability.
\end{claim}
\begin{proof}
We will only prove the statement about the $\vk_i$; the same (actually slightly simplified) argument applies to the $\vd_i$.
Thanks to Bennett's tail bound for the Poisson distribution we may condition on $\{\vm=m\}$ for some integer $m$ with $|m-dn/k|\leq\sqrt n\ln n$.
Fix a small $\delta>0$ and a large enough $L=L(\delta)>0$.
Given $\vm=m$ the variables $Q_j=\sum_{i\in[\vm]}\vecone\{\vk_i=j\}$ have a binomial distribution.
Therefore, the Chernoff bound yields
\begin{align*}
\pr\brk{\abs{Q_j-dn\pr\brk{\vk=j}/k}\leq\sqrt n\ln n\mid\vm=m}&=1-o_n(1/n)&&\mbox{for any $j\leq L$.} 
\end{align*}
Hence, \eqref{eqM} and \Lem~\ref{Cor_llt} yield
\begin{align}\label{eqLemma_sm_1}
\pr\brk{\forall j\leq L:|Q_j-dn\pr[\vk=j]/k|\leq \sqrt n\ln n\mid\cD\cap\cM,\,\vm=m}=1-o_n(1).
\end{align}
Further, let 
\begin{align*}
R_h&=\sum_{j\geq1}\vecone\{(1+\delta)^{h-1}L< j\leq (1+\delta)^{h}L\wedge \sqrt n/\ln^9n\}Q_j,\\
\bar R_h&=m\sum_{j\geq1}\vecone\{(1+\delta)^{h-1}L< j\leq (1+\delta)^{h}L\wedge \sqrt n/\ln^9n\}\pr\brk{\vk=j}
\end{align*}
and let $\cH$ be the set of all integers $h\geq1$ with $(1+\delta)^{h-1}L\leq \sqrt n/\ln^9n$.
Then the Chernoff bound implies that
\begin{align}\label{eqLemma_sm_2}
\pr\brk{\forall h\in\cH:\abs{R_h-\bar R_h}>\delta\bar R_h+\ln^2n\mid\cD\cap\cM,\vm=m}&=o_n(n^{-1}).
\end{align}
Finally, if $|Q_j-dn\pr\brk{\vk=j}/k|\leq\sqrt n\ln n$ for all $j\leq L$ and 
$\abs{R_h-\bar R_h}\leq\delta\bar R_h+\ln^2n$ for all $h\in\cH$, then
\begin{align*}
&\frac1n\sum_{i=1}^m\vk_i^2\leq o_n(1)+\frac dk\Erw\brk{\vk^2\vecone\{\vk\leq L\}}
	+\frac{d}{kn}\sum_{h\in\cH}(1+\delta)^{2h}L^2R_h\\
&=o_n(1)+\frac dk\Erw\brk{\vk^2\vecone\{\vk\leq L\}}
	+\frac{d}{kn}\sum_{h\in\cH}(1+\delta)^{2h+1}L^2\bar R_h\leq\frac{(1+\delta)d}k\Erw[\vk^2]+o_n(1),& \\
	&\mbox{and analogously}\hspace{4pt}\frac1n\sum_{i=1}^m\vk_i^2\geq \frac{(1-\delta)d}k\Erw[\vk^2]+o_n(1).
\end{align*}
Since this holds true for any fixed $\delta>0$, the assertion follows from \eqref{eqLemma_sm_1} and \eqref{eqLemma_sm_2}.
\end{proof}

\begin{claim}\label{Lemma_doubleEdges}
Let $Y$ be the number of multi-edges of the Tanner graph $\G_{0,n}$ and let $\ell\geq1$.
There is $\lambda>0$ such that on
\begin{align*}
\cM\cap\cD\cap
\cbc{\sum_{i=1}^n\vd_i=dn+o_n(n),\,\sum_{i=1}^n\vd_i^2=n\Erw[\vd^2]+o_n(n)} \\
\cap\cbc{
\sum_{i=1}^{\vm}\vk_i^2=dn\Erw[\vk^2]/k+o_n(n)}\cap\{\vm=dn/k+o_n(n)\}
\end{align*}
we have
$$
\Erw\brk{\prod_{i=1}^\ell Y-i+1\,\bigg|\,(\vd_i)_{i\in[n]},(\vk_i)_{i\in\vm}}=\lambda^\ell+o_n(1).
$$
\end{claim}
\begin{proof}
To estimate the $\ell$-th factorial moments of $Y$ for $\ell\geq1$, we split the random variable into a sum of indicator variables.
Specifically, let $U_\ell$ be the set of all families $(u_i,v_i,w_i)_{i\in\ell}$ with $u_i\in[\vm]$, $v_i\in[n]$ and $2\leq w_i\leq \vk_{u_i}\wedge\vd_{v_i}\leq\sqrt n/\log^9n$ such that the pairs $(u_1,v_1),\ldots,(u_\ell,v_\ell)$ are pairwise distinct.
Moreover, let $Y[(u_i,v_i,w_i)_{i\in[\ell]}]$ be the number of ordered $\ell$-tuples of multi-edges comprising precisely $w_i$ edges between check $a_{u_i}$ and variable $x_{v_i}$ for each $i$.
Then
\begin{align*}
\prod_{h=1}^\ell Y-h+1&=\sum_{(u_i,v_i,w_i)_{i\in[\ell]}\in U_\ell}Y[(u_i,v_i,w_i)_{i\in[\ell]}].
\end{align*}
Moreover, letting $w=\sum_iw_i$, we claim that
\begin{align}\label{eqLemma_doubleEdges5}
\Erw[Y[(u_i,v_i,w_i)_{i\in[h]}]\mid(\vd_i)_{i\in[n]},(\vk_i)_{i\in\vm}]&
		\sim\frac{1}{(dn)^w}\prod_{i=1}^\ell\bink{\vk_{u_i}}{w_i}\bink{\vd_{v_i}}{w_i}w_i!\enspace.
\end{align}
Indeed, the factors $\bink{\vd_{v_i}}{w_i}\bink{\vk_{u_i}}{w_i}w_i!$ count the number of possible matchings between $w_i$ clones of the variable node $x_{v_i}$, whose degree equals $\vd_{v_i}$, and of the check node $a_{u_i}$ of degree $\vk_{u_i}$.
Further, since $\ell$ is bounded, the probability that all these matchings are realised in $\G_{0,n}$ is asymptotically equal to $(dn)^{-w}$.

Now, for a sequence $\vw=(w_1,\ldots,w_\ell)$ let $Y_{\vw}=\sum_{(u_i,v_i,w_i)_{i\in[\ell]}\in U_\ell}Y[(u_i,v_i,w_i)_{i\in[\ell]}]$.
Then \eqref{eqLemma_doubleEdges5} yields
\begin{align*}\nonumber
\Erw[Y_{\vw}\mid(\vd_i)_{i\in[n]},(\vk_i)_{i\in\vm}]\leq O_n(n^{-w})\prod_{i=1}^\ell\bc{\sum_{j=1}^n\vd_{j}^{w_i}}\bc{\sum_{j=1}^{\vm}\vk_{j}^{w_i}} \\
	\leq O_n(n^{-w})\max_{j\in[n]}\vd_j^{w-2\ell}\max_{j\in[\vm]}\vk_j^{w-2\ell}\bc{\sum_{j=1}^n\vd_{j}^{2}}^\ell\bc{\sum_{j=1}^{\vm}\vk_{j}^{2}}^\ell\\
\leq O_n(n^{2\ell-w})\max_{j\in[n]}\vd_j^{w-2\ell}\max_{j\in[\vm]}\vk_j^{w-2\ell}=O_n(\ln^{2\ell-w} n);
\end{align*}
the last bound follows from our conditioning on $\cM$.
As a consequence,
\begin{align}\label{eqLemma_doubleEdges7}
\sum_{\vw:w>2\ell}\Erw[Y_{\vw}\mid(\vd_i)_{i\in[n]},(\vk_i)_{i\in\vm}]&=o_n(1).
\end{align}
Further, invoking \eqref{eqLemma_doubleEdges5}, we obtain
\begin{align}\label{eqLemma_doubleEdges8}
\Erw[Y_{(2,\ldots,2)}\mid(\vd_i)_{i\in[n]},(\vk_i)_{i\in\vm}]&\sim\vec\lambda^\ell,\qquad\mbox{where}&
\vec\lambda&\sim
\frac{\bc{\sum_{i=1}^n\vd_i(\vd_i-1)}\bc{\sum_{i=1}^n\vk_i(\vk_i-1)}}{2(dn)^2}.
\end{align}
Combining \eqref{eqLemma_doubleEdges7} and \eqref{eqLemma_doubleEdges8}, we conclude that on $\cD\cap\cM\cap\{\vm=dn/k+o_n(n)\}$,
\begin{align}\label{eqLemma_doubleEdges77}
\Erw[Y^\ell]\mid(\vd_i)_{i\in[n]},(\vk_i)_{i\in\vm}]&\sim\vec\lambda^\ell.
\end{align}
Finally, on $\cbc{\sum_{i=1}^n\vd_i=dn+o_n(n),\,\sum_{i=1}^n\vd_i^2=n\Erw[\vd^2]+o_n(n)}\cap\cbc{
\sum_{i=1}^{\vm}\vk_i^2=dn\Erw[\vk^2]/k+o_n(n)}$ we have
\begin{align}\label{eqLemma_doubleEdges78}
\vec\lambda&\sim\lambda=\frac{(\Erw[\vd^2]-d)(\Erw[\vk^2]-k)}{2d^2},
\end{align}
and the assertion follows from \eqref{eqLemma_doubleEdges77}--\eqref{eqLemma_doubleEdges78}.
\end{proof}

\begin{claim}\label{Cor_doubleEdges}
We have $\pr\brk{\cS\mid\cD\cap\cM}=\Omega_n(1)$.
\end{claim}
\begin{proof}
Claims~\ref{Lemma_sm} and~\ref{Lemma_doubleEdges} show together with inclusion/exclusion  (e.g., \cite[\Thm~1.22]{BB}) that 
\whp\ on $\cM\cap\cD$,
$$\pr\brk{Y=0\mid(\vd_i)_{i\in[n]},(\vk_i)_{i\in\vm}}=\exp(-\lambda)=\Omega_n(1).$$
Since $\cS=\{Y=0\}$, the assertion follows.
\end{proof}

\begin{proof}[Proof of \Lem~\ref{Lemma_simple}]
The assertion follows immediately from \eqref{eqM}, \Cor~\ref{Cor_llt} and \Cor~\ref{Cor_doubleEdges}.
\end{proof}

\begin{proof}[Proof of \Prop~\ref{Lemma_welldef}]
The proposition is immediate from \Lem s~\ref{Cor_llt} and~\ref{Lemma_simple}.
\end{proof}

\subsection{Proof of \Prop~\ref{Cor_lower}}\label{Sec_Cor_lower}

\noindent
The random matrix $\vA_n$ has $n$ columns and $\vm\disteq\Po(dn/k)$ rows, with the column and row degrees drawn from the distributions $\vd$ and $\vk$.
By comparison, $\vA_{\eps,n}$ has slightly fewer, namely $\vm_{\eps,n}\disteq\Po((1-\eps)dn/k)$ rows.
One might therefore think that the proof of \Prop~\ref{Cor_lower} is straightforward, as it appears that $\vA_n$ is obtained from $\vA_{\eps,n}$ by simply adding another random $\Po(\eps dn/k)$ rows.
Since adding $O_{\eps,n}(\eps n)$ rows cannot reduce the nullity by more than $O_{\eps,n}(\eps n)$, the bound on
 $\Erw[\nul(\vA_n)] -\Erw[\nul(\vA_{\eps,n})]$ appears to be immediate.
But there is a catch.
Namely, in constructing $\vA_n$  we condition on the event $\cD=\{\sum_{i=1}^{n}\vd_i=\sum_{i=1}^{\vm}\vk_i\}$.
Thus,  $\vA_{\eps,n}$ does not have the same distribution as the top $\Bin(\vm,1-\eps)$ rows of $\vA_n$ since the conditioning might distort the degree distribution.
We need to show that this distortion is insignificant.
To this end, recall that $\vm_{\eps,n}\disteq\Po((1-\eps)dn/k)$.


\begin{lemma}\label{Lemma_conc}
\Whp\ we have
\begin{align*}
\pr\brk{\abs{\nul(\vA_{\eps,n})-\Erw[\nul(\vA_{\eps,n})\mid\vm_{\eps,n},(\vd_i)_{i\geq1},\,(\vk_i)_{i\geq1}]}>\sqrt n\ln n\mid\vm_{\eps,n},(\vd_i)_{i\geq1},\,(\vk_i)_{i\geq1}}&=o_n(1),
\end{align*}
\end{lemma}

\begin{proof}
	\Lem~\ref{Lemma_sums} shows that $\sum_{i=1}^n\vd_i,\sum_{i=1}^{\vm_{\eps,n}}\vk_i=O_{\eps,n}(n)$
	and $\sum_{i=1}^{\vm_{\eps,n}}\vk_i\leq\sum_{i=1}^n\vd_i$  with probability $1-o_n(n^{-1})$.
	Assuming that this is so, consider a filtration $(\fA_t)_{t\leq\sum_{i=1}^{\vm_{\eps,n}}\vk_i}$ that reveals the random matching $\vec\Gamma_{\eps,n}$ one edge and the corresponding matrix entry at a time.  Consider
	\begin{align*}
	 \vX(\row,\col )  =  \Erw[\nul(\vA_{\eps,n})\mid\vm_{\eps,n},(\vd_i)_{i\geq1},\,(\vk_i)_{i\geq1}, (\row_i)_{i\geq1}, (\col_i) _{i\geq1} ]
	\end{align*}
	and with averaging over $\col_i$ obtain
	\begin{align*}
	\vX'(\row )  =  \Erw[\nul(\vA_{\eps,n})\mid \vm_{\eps,n},(\vd_i)_{i\geq1},\,(\vk_i)_{i\geq1}, (\row_i)_{i\geq1} ] &&
	\abs { \Erw[\vX(\row,\col ) \mid\fA_{t+1}] - \Erw[\vX(\row,\col ) \mid\fA_{t}] } \leq 1.
	\end{align*}
	Applying Azuma's inequality we conclude
	\begin{align*}
		\pr\brk{\abs{ \vX(\row,\col ) - \vX'(\row ) } > \sqrt n\ln n  \mid \vm_{\eps,n}, (\vd_i)_{i\geq1},\,(\vk_i)_{i\geq1}, (\row_i)_{i\geq1} }&=o_n(1).
	\end{align*}
	Then repeating this step again by averaging over $\row_i$ and applying Azuma's inequality we have
	\begin{align*}
	\vX'' =  \Erw[\nul(\vA_{\eps,n})\mid\vm_{\eps,n},(\vd_i)_{i\geq1},\,(\vk_i)_{i\geq1} ]  && 
	\abs { \Erw[\vX'(\row) \mid\fA_{t+1}] - \Erw[\vX'(\row) \mid\fA_{t}] } \leq 1 
	\end{align*}
	and
	\begin{align*}
		\pr\brk{\abs{ \vX'(\row) - \vX'' } > \sqrt n\ln n  \mid  \vm_{\eps,n}, (\vd_i)_{i\geq1},\,(\vk_i)_{i\geq1} }&=o_n(1).
	\end{align*}
	The same step again for $\nul(\vA_{\eps,n})$ and $X''$ results in the assertion.
\end{proof}


Let $\vA_{n,\cD}$ be the conditional version of the random matrix $\vA_{0,n}$ given $\cD$.
Thus, given $\sum_{i=1}^n\vd_i=\sum_{i=1}^{\vm}\vk_i$, we construct a random Tanner multi-graph with variable degrees $\vd_1,\ldots,\vd_n$ and check degrees $\vk_1,\ldots,\vk_{\vm}$.
Hence, the difference between $\vA_n$ and $\vA_{n,\cD}$ is merely that in the case of $\vA_n$ we also condition on the event $\cS$ that the Tanner graph is simple.

\begin{lemma}\label{Lemma_JanesCoupling}
There exists a coupling of $\vA_{n,\cD}$ and $\vA_{\eps,n}$ such that with probability at least $1-\eps$ the two matrices agree in all but $O_{\eps,n}(\eps n)$ rows.
\end{lemma}
\begin{proof}
Let $\G_{n,\cD}$ and $\G_{\eps,n}$ denote the Tanner graphs corresponding to $\vA_{n,\cD}$ and $\vA_{\eps,n}$, respectively.
It suffices to construct a coupling of $\G_{n,\cD}$ and $\G_{\eps,n}$ such that these graphs differ in edges incident with at most $O_{\eps,n}(\eps n)$ check nodes.
To construct the coupling we first generate the following parameters for $\vA_{\eps,n}$.
Parameter $\cT=\cT(\eps)$ is given.
Generate $\bftheta\in [\cT]$ uniformly at random.
Then generate $\vm\sim\Po(dn/k)$ and $\bfm_{\eps,n}=\Bin(\vm,1-\eps)$ and then check nodes $a_1,\ldots,a_{\vm_{\eps,n}}$.
Each check node $a_i$ is associated with an integer $\bfk_i$ which is an independent copy of $\bfk$. To distinguish $\G_{\eps,n}$ from $\G_{n,\cD}$, we colour these check nodes red.
Add $\bftheta$ check nodes $p_1,\ldots, p_{\bftheta}$ to both  $\G_{\eps,n}$ and $\G_{n,\cD}$.

Next generate $n$ variable nodes where variable node $x_i$ is associated with $\bfd_i$, which is an independent copy of $\bfd$. 
Further, let $\vec r_j=\sum_{h=1}^{\vm}\vecone\cbc{\vk_h=j}$ denote the prospective number of checks of $\G_{n,\cD}$ of degree $j$.
Applying Azuma's inequality and \eqref{eqM}, we see that for any $\eps>0$ there exists $L=L(\eps)>0$ such that
\[
\pr\brk{\sum_{j\ge L} \vec r_j >\eps n\mid\cM} +\pr\brk{
	\exists j\leq L:\vec r_j\leq\sum_{i=1}^{\bfm_{\eps,n}} \vecone\{\bfk_i=j\}\mid\cM} 
	+\pr\brk{\vm>\vm_{\eps,n}+2\eps dn/k}\leq 1/n.
\]
Hence,  \Lem~\ref{Cor_llt} implies that
\begin{align}\nonumber
\pr&\brk{\sum_{j\ge L} \vec r_j >\eps n\mid\cD\cap\cM}\\&+\negthickspace\pr\brk{\exists j\leq\negthickspace L:\vec r_j\leq\negthickspace\sum_{i=1}^{\bfm_{\eps,n}} \vecone\{\bfk_i\negthickspace=\negthickspace j\} \ \mbox{for all $j\le L$}\negthickspace\mid\cD\cap\cM\negthickspace}\negthickspace 
+\negmedspace\pr\brk{\vm\negthickspace >\vm_{\eps,n}\negthickspace+2\eps dn/k\mid\cD\cap\cM}\leq\negthickspace 1/n=o_n(1).
\label{c:dominance}
\end{align}

Now condition on the event
\begin{align*}
\cR&=\cD\cap\cM\cap\cbc{\sum_{j\ge L} \bfn_j\leq\eps n}\cap\cbc{\forall j\leq L:\bfn_j>\sum_{i=1}^{\bfm_{\eps,n}} \vecone\{\bfk_i=j\} }\cap\cbc{\vm\leq\vm_{\eps,n}+2\eps dn/k}.
\end{align*}
Uncolour all (red) check nodes $a_i$ with $\bfk_i\le L$.
Moreover, for each $j\le L$, generate $\vec r_j-\sum_{i=1}^{\bfm_{\eps,n}} \vecone\{\bfk_i=j\}$ additional check nodes of degree $j$ and  colour them blue.
Finally, for each $j>L$, generate $\vec r_j$ blue check nodes of degree $j$. 

Now $\G_{\eps,n}$ is generated by taking a random maximal matching from the clones of all {\em uncoloured} and {\em red} check nodes $\{a_i\}\times [\bfk_i]$ (excluding check nodes $p_1,\ldots,p_{\bftheta}$) to the set of variable clones $$\bigcup_{j=1}^n\{x_j\}\times [\bfd_j],$$ and then adding an edge between $p_i$ and $x_i$ for $1\le i\le \bftheta$.
The Tanner graph $\G_{n,\cD}$ is generated by removing all matching edges from the clones of the {\em red} check nodes, and removing edges between $p_i$ and $a_i$ for $1\le i\le \bftheta$, and then matching all clones of the {\em blue} check nodes to the remaining clones of the variable nodes. 
%
Finally, \eqref{c:dominance} ensures that with probability at least $1-\eps$, the two Tanner graphs differ in no more than  $O_{\eps,n}(\eps n)$ check nodes.
\end{proof}

\begin{proof}[Proof of \Prop~\ref{Cor_lower}]
Assume that \eqref{eqCor_lower_ub} is satisfied for $C>0$ and fix $C'>C$ and a small enough $\delta>0$.
Then we find a small $0<\eps=\eps(\delta)<\delta$ such that
\begin{align*}
\limsup_{n\to\infty}\Erw[\nul(\vA_{\eps,n})]/n&\leq C+\delta.
\end{align*}
Hence, combining \Lem s~\ref{Lemma_conc} and~\ref{Lemma_JanesCoupling} and taking into account that changing a single row can alter the nullity by at most one, we conclude that
\begin{align}\label{eqCor_lower_ub_1}
\pr\brk{\nul(\vA_{n,\cD})/n\leq C+O_{\eps,n}(\eps)}&>1-\eps+o_n(1).
\end{align}
Finally, combining \eqref{eqCor_lower_ub_1} and \Lem~\ref{Lemma_simple}, we conclude that
\begin{align}\label{eqCor_lower_ub_2}
\pr\brk{\nul(\vA_{n,\cD})/n\leq C+O_{\eps,n}(\eps)\mid\cS}&>1-\delta+o_n(1),
\end{align}
provided that $\eps=\eps(\delta)$ is small enough.
Since $\vA_{n,\cD}$ given $\cS$ is identical to $\vA_n$, the desired upper bound on the nullity of $\vA_n$ follows from \eqref{eqCor_lower_ub_2}.
The same argument renders the lower bound.
\end{proof}

\section{The Aizenman-Sims-Starr scheme: proof of \Prop~\ref{Prop_coupling}}\label{Sec_lower}

\noindent
In this section we prove \Prop~\ref{Prop_coupling}.
As set out in \Sec~\ref{Sec_outline1}, we are going to bound the difference of the nullities of $\vA_{\eps,n+1}$ and $\vA_{\eps,n}$ via \Prop~\ref{Prop_Alp} and \Lem~\ref{Cor_free}.
This requires a coupling of the random variables $\nul(\vA_{\eps,n+1})$ and $\nul(\vA_{\eps,n})$.

\subsection{The coupling}
We begin by introducing a more fine-grained description of the random matrices $\vA_{\eps,n}$ and $\vA_{\eps,n+1}$ to facilitate the construction of the coupling.
To this end, let $\vM=(\vM_j)_{j\geq1}$ and $\DELTA=(\DELTA_j)_{j\geq1}$ be sequences of Poisson variables with means
\begin{align}\label{eqPoissons}
\Erw[\vM_j]&=(1-\eps)\pr\brk{\vk=j}dn/k,&\Erw[\DELTA_j]&=(1-\eps)\pr\brk{\vk=j}d/k.
\end{align}
All of these random variables are mutually independent and independent of $\THETA$ and the $(\vd_i)_{i\geq1}$.
Further, let
\begin{align}\label{eqm}
\vM_j^+&=\vM_j+\DELTA_j,
&\vm_{\eps,n}&=\sum_{j\geq1}\vM_j,&\vm_{\eps,n}^+&=\sum_{j\geq1}\vM_j^+.
\end{align}
Since  $\sum_{j\geq1}\vM_j\disteq\Po((1-\eps)dn/k)$, (\ref{eqm}) is consistent with the earlier convention that $\vm_{\eps,n}\disteq\Po((1-\eps)dn/k)$.

The random vectors $(\vd_1,\ldots\vd_n),\vM$ naturally define a random Tanner (multi-)graph $\G_{n,\vM}$ with variable nodes $x_1,\ldots,x_n$ 
and check nodes $p_1,\ldots,p_{\THETA}$ and $a_{i,j}$, $i\geq1$, $j\in[\vM_i]$.
Its edges are induced by a random maximal matching  $\vec\Gamma_{n,\vM}$ of the complete bipartite graph with vertex classes
\begin{align*}
\bigcup_{h=1}^n\{x_h\}\times[\vd_h]\quad\mbox{and}\quad\bigcup_{i\geq1}\bigcup_{j=1}^{\vM_i}\{a_{i,j}\}\times[i].
\end{align*}
Each matching edge $(x_h,s,a_{i,j},t)\in \vec\Gamma_{n,\vM}$ induces an edge between $x_h$ and $a_{i,j}$ in the Tanner graph.
In addition, there is an edge between $p_i$ and $x_i$ for every $i\in[\THETA]$.

To define the random matrix $\vA_{n,\vM}$ to go with $\G_{n,\vM}$, let $\chi:[0,1]^2\to\FF^*$ be a measurable map and let $(\row_{i,j},\col_i)_{i,j\geq1}$ be uniformly distributed on $[0,1]$, mutually independent and independent of all other randomness.%
\footnote{\aco{Unfortunately at this point there does not seem to be an ideal notation for the matrix and its entries. Because the random vector $\vM$ depends on $n$ and to preserve the analogy with common random graphs notation, we denote the random Tanner graph by $\G_{n,\vM}$ and its associated random matrix by $\vA_{n,\vM}$. At the same time, in line with linear algebra conventions, when indexing matrix entries we let the first index refer to the row of the matrix and the second index to the column. Since the variable $n$ nodes correspond to the columns, a degree of incoherence seems unavoidable.}}
With the rows of $\vA_{n,\vM}$ indexed by the check nodes of $\G_{n,\vM}$ and the columns indexed by the variable nodes, we define the matrix entries by letting
\begin{align*}
(\vA_{n,\vM})_{p_i,\vx_h}&=\vecone\cbc{i=j}&&(i\in[\THETA],h\in[n]),\\
(\vA_{n,\vM})_{a_{i,j},\vx_h}&=\chi_{\row_{i,j},\col_h}\sum_{s=1}^i\sum_{t=1}^{\vd_h}
	\vecone\cbc{\{(x_h,t),(a_{i,j},s)\}\in\vec\Gamma_{n,\vM}}&&(i\geq1,j\in[\vM_i],h\in[n]).
\end{align*}
The Tanner graph $\G_{n+1,\vM^+}$  and its associated random matrix $\vA_{n+1,\vM^+}$ are defined analogously.

\begin{lemma}\label{Lemma_ComplicatedModel}
For any $\theta>0$ we have
$\Erw[\nul(\vA_{\eps,n})]=\Erw[\nul(\vA_{n,\vM})]$, 
$\Erw[\nul(\vA_{\eps,n+1})]=\Erw[\nul(\vA_{n+1,\vM^+})]$.
\end{lemma}
\begin{proof}
We defined $\vA_{\eps,n}$ as the $\vm_{\eps,n}\times n$-matrix with target column and row degrees drawn from $\vd$ and $\vk$ independently with a $\THETA\times\THETA$ identity matrix affixed at top.
In effect, because $\vm_{\eps,n}$ is a Poisson variable, the number of rows of with target degree $i$ is distributed as $\vM_i$, and these numbers are mutually independent.
Hence, $\nul\vA_{\eps,n}$ and $\nul\vA_{n,\vM}$ are identically distributed.
The same argument applies to $\vA_{\eps,n+1}$.
\end{proof}

Up to this point we merely introduced a new description of $\vA_{\eps,n}$ and $\vA_{\eps,n+1}$.
To actually couple them we introduce a third random matrix whose nullity we can easily compare to 
$\nul(\vA_{n,\vM})$ and $\nul(\vA_{n+1,\vM^+})$.
Specifically, let $\GAMMA_i\ge0$ be the number of checks $a_{i,j}$, $j\in[\vM_i^+]$, adjacent to the last variable node $x_{n+1}$ in $\G_{n+1,\vM^+}$.
Also let $\GAMMA=(\GAMMA_i)_{i\geq1}$ and set
\begin{equation}\label{eqminus}
	\vM_i^-=\max\{\vM_i-\GAMMA_i,0\}.
\end{equation}
 \aco{In \eqref{eqminus} the max is necessary because potentially $\vec\GAMMA_i$ might exceed $\vM_i$ as $\vec\GAMMA_i$ might include some of the ``extra'' $\vec\Delta_i$ checks included in $\G_{n+1,\vM^+}$.}
Consider the random Tanner graph $\G'=\G_{n,\vM^-}$ induced by a random maximal matching $\vec\Gamma'$ of the complete bipartite graph with vertex classes
\begin{align*}
\bigcup_{h=1}^n\{x_h\}\times[\vd_h]\quad\mbox{and}\quad\bigcup_{i\geq1}\bigcup_{j=1}^{\vM_i^-}\{a_{i,j}\}\times[i].
\end{align*}
For each variable $x_i$, $i=1,\ldots,n$, let $\cC$ be the set of clones from $\bigcup_{i\in[n]}\{x_i\}\times[\vd_i]$ that $\vec\Gamma_{n,\vM^-}$ leaves unmatched.
We call the elements of $\cC$ {\em cavities}.

Now, obtain the Tanner graph $\G''$ from $\G'$ by adding new check nodes 
\begin{align}\label{eqcplng1}
	a''_{i,j}\mbox{ with target degree $i$ for each }i\geq1,\ j\in[\vM_i-\vM_i^-].
\end{align}
The new checks are joined by a random maximal matching $\vec\Gamma''$ of the complete bipartite graph with vertex classes 
 $$\mbox{$\cC\qquad$ and}\qquad\bigcup_{i\geq1}\bigcup_{j\in[\vM_i-\vM_i^-]}\{a_{i,j}''\}\times[i],$$
i.e., for each matching edge we insert a corresponding variable-check edge.

Analogously, obtain $\G'''$ by adding one variable node $x_{n+1}$ as well as check nodes $a_{i,j}'''$, $i\geq1$, $j\in[\GAMMA_i]$
and $b_{i,j}'''$, $i\geq1$, $j\in[\vM_i^+-\vM_i^--\GAMMA_i]$ to $\G'$.
The new checks are connected to $\G'$ via a random maximal matching $\vec\Gamma'''$ of the complete bipartite graph with vertex classes
 $$\mbox{ $\cC\qquad$ and }\qquad\bigcup_{i\geq1}\bc{\bigcup_{j\in[\GAMMA_i]}\{a_{i,j}'''\}\times[i-1]
			\cup\bigcup_{j\in[\vM_i^+-\vM_i^--\GAMMA_i]}\{b_{i,j}'''\}\times[i]}.$$
For each matching edge we insert the corresponding variable-check edge and
in addition each of the check nodes $a_{i,j}'''$ gets connected to $x_{n+1}$ by exactly one edge.

Finally, we introduce the random matrices $\vA',\vA'',\vA'''$ whose non-zero entries represent the edges of $\G',\G'',\G'''$.
Recalling that $(\row_{i,j},\col_i)_{i,j\geq1}$ are uniform on $[0,1]$ and independent of everything else, we additionally introduce independent random variables $(\row_{i,j}',\row_{i,j}'')_{i,j\geq1}$, also uniform on $[0,1]$.
With the rows and columns indexed by check and variable nodes, respectively, we define
\begin{align*}
\vA'_{p_i,j}&=\vA''_{p_i,j}=\vA'''_{p_i,j}=\vecone\cbc{i=j}\hspace{12pt}(i\in[\THETA],j\in[n]),\\
\vA'_{a_{i,j},x_h}&=\vA''_{a_{i,j},x_h}=\vA'''_{a_{i,j},x_h}
	=\chi_{\row_{i,j},\col_h}\sum_{s=1}^i\sum_{t=1}^{\vd_h}
	\vecone\cbc{\{(x_h,t),(a_{i,j},s)\}\in\vec\Gamma'}\hspace{12pt}(i\geq1,j\in[\vM_i^-],h\in[n]),\\
\vA''_{a_{i,j}'',x_h}&=\chi_{\row_{i,j}',\col_h}\sum_{s=1}^i\sum_{t=1}^{\vd_h}
	\vecone\cbc{\{(x_h,t),(a_{i,j}'',s)\in\vec\Gamma''\}}\hspace{12pt}(i\geq1,j\in[\vM_i-\vM_i^-,h\in[n]),\\
\vA'''_{a_{i,j}''',x_h}&=\chi_{\row_{i,j}',\col_h}\sum_{s=1}^{i-1}\sum_{t=1}^{\vd_h}
	\vecone\cbc{\{(x_h,t),(a_{i,j}''',s)\in\vec\Gamma'''\}}\hspace{12pt}(i\geq1,j\in[\GAMMA_i,h\in[n]),\\
\vA'''_{b_{i,j}''',x_h}&=\chi_{\row_{i,j}'',\col_h}\sum_{s=1}^{i}\sum_{t=1}^{\vd_h}
	\vecone\cbc{\{(x_h,t),(b_{i,j}''',s)\in\vec\Gamma'''\}}\hspace{12pt}(i\geq1,j\in[\vM_i^+-\vM_i^--\GAMMA_i],h\in[n]).
\end{align*}
In line with the strategy outlined in \Sec~\ref{Sec_overview}, this construction ensures that $\vA''$ and $\vA'''$ are obtained from $\vA'$ by adding a bounded expected number of rows and, in the case of $\vA'''$, one column.
The following lemma links $\vA'',\vA'''$ to the random matrices $\vA_{n,\vM}$, $\vA_{n+1,\vM^+}$ from the beginning of the section.

\begin{lemma}\label{Lemma_valid}
We have  $\Erw[\nul(\vA'')]=\Erw[\nul(\vA_{n,\vM})]+o_n(1)$ and  $\Erw[\nul(\vA''')]=\Erw[\nul(\vA_{n+1,\vM^+})]+o_n(1).$
\end{lemma}

\noindent
The proof of \Lem~\ref{Lemma_valid}, deferred to \Sec~\ref{Sec_Lemma_valid}, is tedious but relatively straightforward.

As a next step we are going to calculate the  differences $\nul(\vA''')-\nul(\vA')$ and $\nul(\vA'')-\nul(\vA')$.
We obtain expressions of one parameter of $\vA'$, namely the fraction of cavities `frozen' to zero. 
To be precise, a cavity $(x_i,h)\in\cC$ is {\em frozen} if $x_i\in\fF(\vA')$.
Let $\cF\subset\cC$ be the set of all frozen cavities and define $\vec\alpha=|\cF|/|\cC|$;
in the unlikely event that $\cC=\emptyset$, we agree that $\vec\alpha=0$.
In \Sec s~\ref{Sec_A'''} and~\ref{Sec_A''} we are going to establish the following two estimates.

\begin{lemma}\label{Lemma_A'''}
We have  $\Erw[\nul(\vA''')-\nul(\vA')]=\Erw[D(1- K'(\vec\alpha)/k)+d(K'(\vec\alpha)+K(\vec\alpha)-1)/k]-d+o_\eps(1).$
\end{lemma}

\begin{lemma}\label{Lemma_A''}
We have $\Erw[\nul(\vA'')-\nul(\vA')]=d\Erw[\vec\alpha K'(\vec\alpha)]/k-d+o_\eps(1).$
\end{lemma}

\aco{We emphasise that the r.h.s.\ expressions in \Lem s~\ref{Lemma_A'''} and~\ref{Lemma_A''} involve expectations on the random variable $\vec\alpha$.
A key feature of the present argument is that we manage to avoid an analysis of $\vec\alpha$ altogether.
This is because, as the following proof of \Prop~\ref{Prop_coupling} shows, we can just replace the difference of the expectations by the largest conceivable value.}

\begin{proof}[Proof of \Prop~\ref{Prop_coupling}]
\aco{Combining \Lem s~\ref{Lemma_ComplicatedModel} and~\ref{Lemma_valid}, we see that
	\begin{align}\nonumber
		\Erw[\nul(\vA_{\eps,n+1})]-\Erw[\nul(\vA_{\eps,n})]&=\Erw[\nul(\vA_{n+1,\vM^+})]-\Erw[\nul(\vA_{n,\vM})]\\
														   &=\Erw[\nul(\vA''')]-\Erw[\nul(\vA'')]+o_n(1)\nonumber\\
														   &=(\Erw[\nul(\vA''')]-\Erw[\nul(\vA')])-(\Erw[\nul(\vA'')]-\Erw[\nul(\vA')])+o_n(1).\label{eqProp_coupling_1}
	\end{align}
	Further, combining \eqref{eqProp_coupling_1} with \Lem s~\ref{Lemma_A'''} and \ref{Lemma_A''}, we obtain
\begin{align}\nonumber
	\Erw[\nul(\vA_{\eps,n+1})]-\Erw[\nul(\vA_{\eps,n})]
		&\leq\Erw[D(1- K'(\vec\alpha)/k)+d(K'(\vec\alpha)+K(\vec\alpha)-1)/k-d\vec\alpha K'(\vec\alpha)]/k+o_\eps(1)\\
		&=\Erw[\Phi(\vec\alpha)]+o_\eps(1)\leq\max_{\alpha\in[0,1]}\Phi(\alpha)+o_\eps(1).\label{eqProp_coupling_2}
	\end{align}
The proposition is an immediate consequence of \eqref{eqProp_coupling_2}.}
\end{proof}

While proving \Lem s~\ref{Lemma_A'''} and~\ref{Lemma_A''} in full detail requires a fair bit of work because we are dealing with very general degree distributions $\vd,\vk$, it is not at all difficult to fathom where the right hand side expressions come from.
They actually arise naturally from \Lem~\ref{Cor_free} and the scarcity of short proper relations supplied by \Prop~\ref{Prop_Alp}.
Indeed, we can write the matrices $\vA''$, $\vA'''$ in the form
\begin{align}\label{eqA''A'''}
\vA''&=\begin{pmatrix}\vA'\\\vB\end{pmatrix},&
\vA'''&=\begin{pmatrix}\vA'&0\\\vB'&\vC'\end{pmatrix}
\end{align}
with $\vB,(\vB'\ \vC')$ representing the new rows and, in the case of $\vA'''$, the additional column.
To calculate $\Erw[\nul(\vA''[\THETA])-\nul(\vA'[\THETA])]$ we basically need to assess the impact of adding a few more checks $a_{i,j}''$ to the Tanner graph $\G'$ of $\vA'$.
The new checks connect to randomly chosen cavities of $\A'$.
Let $k_1,\ldots,k_L$ denote the degrees of the new checks.
Since the distribution $\vk$ of the check degrees has a finite second moment, the total degree $k_1+\cdots+k_L$ is bounded \whp\
The random matrix $\vB$ therefore encodes the non-zero entries corresponding to the edges that connect the $a_{i,j}'$ with the cavities of $\vA'$ where the new checks attach.
Furthermore, the construction of $\vA'$ ensures that \whp\ the number of cavities is as large as $(1+o_n(1))\eps dn$, and the $a_{i,j}'$ hatch on to randomly chosen cavities.
Therefore, \Prop~\ref{Prop_Alp}, applied with $\cT=\cT(\eps)$ large enough, implies that the probability that the set $I$ of non-zero columns of $\vB$ forms a proper relation \aco{of $\vA'$} is $o_\eps(1)$.
Consequently, \Lem~\ref{Cor_free} yields
\begin{align}\label{eqxLemma_A''_1}
\Erw[\nul(\vA'')-\nul(\vA')]=-\Erw[\rk(\vB_*)]+o_n(1),
\end{align}
where $\vB_*$ is obtained from $\vB$ by zeroing out all columns indexed by $\fF(\vA')$.
Further, since the number of cavities of $\vA'$ is as large as $\Omega_n(n)$ while $k_1+\cdots+k_L=o_n(\sqrt n)$ \whp, the matrix $\vB$ has the following form \whp: there are $L$ rows containing $k_1,\ldots,k_L$ non-zero entries, respectively, and every column of $\vB$ contains at most one non-zero entry.
Consequently, once more because there are as many as $\Omega_n(n)$ cavities out of which an $\ALPHA$ fraction are frozen to zero, $\vB_*$ is close in total variation to the matrix obtained from $\vB$ by zeroing out every column with probability $\ALPHA$ independently.
In effect, the probability that the $i$-th row of $\vB_*$ gets zeroed out entirely is $\ALPHA^{k_i}+o_n(1)$.
Thus, \whp\ we have
\begin{align}\label{eqxLemma_A''_2}
\Erw[\rk(\vB_*)\mid \ALPHA,k_1,\ldots,k_L]&=\sum_{i=1}^L\bc{1-\ALPHA^{k_i}}+o_{\eps,n}(1).
\end{align}
Substituting \eqref{eqxLemma_A''_2} into \eqref{eqxLemma_A''_1} and the correct distribution of $k_1,\ldots,k_L$ supplied by the coupling into \eqref{eqxLemma_A''_2}, we obtain the expression displayed in \Lem~\ref{Lemma_A''}.
\aco{To be explicit, the correct degrees $k_1,\ldots,k_L$ are provided by \eqref{eqcplng1}, i.e., there are $\vM_i-\vM_i^-$ checks of degree $i$ for every $i$.
	Hence, to obtain the expression in \Lem~\ref{Lemma_A'''} we need to analyse the random variables $\vec\gamma_i$ from \eqref{eqminus}.
	This analysis will be conducted in \Lem~\ref{Cor_gamma} below, which shows that the $\vec\gamma_i$ are well approximated by the $\hat{\vec\gamma}_i$ from \eqref{eqhatGamma}, which in turn come in terms of the reweighted check degree distribution from \eqref{eqSizeBiasd}.
}
A similar but slightly more complicated calculation explains the expression in \Lem~\ref{Lemma_A'''}.
We proceed to prove \Lem s~\ref{Lemma_valid}--\ref{Lemma_A''} formally.
This requires a bit of groundwork.

\subsection{Groundwork}
Let $P=P_{\G'}$ be the distribution on the set $V_n=\cbc{x_1,\ldots,x_n}$ of variables induced by choosing a cavity uniformly at random, i.e.,
\begin{align*}
P(x_i)=|\cC\cap(\{x_i\}\times[\vd_i])|/|\cC|;
\end{align*}
in the (unlikely) event that $\cC=\emptyset$, we use the convention $P(x_1)=1$.
Let $\vx_1,\vx_2,\ldots\in V_n$ be independent samples drawn from $P$.
The following lemma shows that $|\cC|$ is linear in $n$ \whp{}

\begin{lemma}\label{Lemma_cavityCount}
\Whp\ we have $|\cC|\geq\eps dn/2$.
\end{lemma}
\begin{proof}
The choice \eqref{eqPoissons} of $\vM$ ensures that $\Erw\sum_{j\geq1}j\vM_j=(1-\eps)dn$.
Moreover, because the $\vM_j$ are mutually independent Poissons,
\begin{align*}
\Var\sum_{j\geq1}j\vM_j=\sum_{j\geq1}j^2\Var(\vM_j)=
\sum_{j\geq1}j^2\Erw[\vM_j]
=(1-\eps)dn\Erw[\vk^2]/k=O_{\eps,n}(n).
\end{align*}
Consequently, Chebyshev's inequality shows that 
\begin{align}\label{eqLemma_cavityCount1}
\pr\brk{\abs{\sum_{j\geq1}j\vM_j-(1-\eps)dn}\leq\sqrt n\log n}&=1-o_n(1).
\end{align}
Similarly, we have $\Erw\sum_{i=1}^n\vd_i=dn$ and $\Var\sum_{i=1}^n\vd_i=\sum_{i=1}^n\Var(\vd)=O_{\eps,n}(n),$
whence
\begin{align}\label{eqLemma_cavityCount2}
\pr\brk{\abs{\sum_{i=1}^n\vd_i-dn}\leq\sqrt n\log n}&=1-o_n(1).
\end{align}
Since $|\cC|\geq\sum_{i=1}^n\vd_i-\sum_{j\geq1}j\vM_j$ by construction, the assertion follows from
\eqref{eqLemma_cavityCount1} and~\eqref{eqLemma_cavityCount2}.
\end{proof}

Further, letting $\ell_*=\lceil\exp(1/\eps^4)\rceil$ and $\delta_*=\exp(-1/\eps^4)$, consider the event
\begin{align}\label{eqE}
\cE&=\cbc{\pr\brk{\vx_1,\ldots,\vx_{\ell_*}\mbox{ form a proper relation of }\vA'\mid\vA'}<\delta_*}.
\end{align}
The following simple lemma is an application of \Prop~\ref{Prop_Alp}.

\begin{lemma}\label{Lemma_theta}
For sufficiently large $\cT=\cT(\eps)>0$ we have
$\pr\brk{\vA'\in\cE}>\exp(-1/\eps^4)$.
\end{lemma}
\begin{proof}
\Lem~\ref{Lemma_cavityCount} provides that $|\cC|\geq\eps n/2$ \whp\
Moreover, since $\Erw[\vd]=O_{\eps,n}(1)$ we find $L=L(\eps)>0$ such that
the event  $\cL=\cbc{\sum_{i=1}^n\vd_i\vecone\{\vd_i>L\}<\eps\delta_*^2 n/(16\ell_*)}$ has probability at least $1-\delta_*/8$.
Thus, we may condition on the event $\cA=\cL\cap\{|\cC|\geq\eps n/2\}$.

Let $\hat\vx_1,\ldots,\hat\vx_{\ell_*}$ be a sequence of independently and uniformly chosen variables from $x_1,\ldots,x_n$.
Consider a set $\cW\subset\{x_1,\ldots,x_n\}^{\ell_*}$.
How can we estimate the probability that $(\vx_1,\ldots,\vx_{\ell_*})\in\cW?$
Either one of the variables $\vx_1,\ldots,\vx_{\ell_*}$ has degree greater than $L$;
	on the event $\cA$ this occurs with probability at most $\delta_*^2/16$.
Or  all of $\vx_1,\ldots,\vx_{\ell_*}$ have degree at most $L$.
Then the probability that $(\vx_1,\ldots,\vx_{\ell_*})\in\cW$ is not much greater than the probability that $(\hat\vx_1,\ldots,\hat\vx_{\ell_*})\in\cW$.
To be precise, since $\hat\vx_1,\ldots,\hat\vx_{\ell_*}$ are chosen uniformly and there are at least $\eps n/2$ cavities, the probabilities differ by no more than a factor of $(2L/\eps)^{\ell_*}$.
Hence, {on the event }$\cA$ we have
\begin{align}
\pr\brk{(\vx_1,\ldots,\vx_{\ell_*})\in\cW\mid\vA'}
	&\leq (2L/\eps)^{\ell_*}\pr\brk{(\hat\vx_1,\ldots,\hat\vx_{\ell_*})\in\cW\mid\vA'}+\delta_*^2/8.\label{eqLemma_theta3}
\end{align}
Applying \eqref{eqLemma_theta3} to the set $\cW$ of proper relations and invoking \Prop~\ref{Prop_Alp} completes the proof.
\end{proof}

Further, consider the event
\begin{align}\label{eqE'}
\cE'=\cbc{|\cC|\geq\eps dn/2\wedge \max_{i\leq n}\vd_i\leq n^{1/2}}.
\end{align}

\begin{lemma}\label{Claim_A'''2a}
We have $\pr\brk{\cE'}=1-o_n(1).$
\end{lemma}
\begin{proof}
This follows from the choice of the parameters in~\eqref{eqPoissons},  \Lem~\ref{Lemma_sums} and \Lem~\ref{Lemma_cavityCount}.
\end{proof}

To prove \Lem s~\ref{Lemma_A'''} and~\ref{Lemma_A''} we need an explicit description of the vector $\vec\gamma$ that captures the degrees of the checks adjacent to the new variable node $x_{n+1}$.
Since $\vec\gamma$ is defined in terms of the the `big' Tanner graph $\G_{n+1,\vM^+}$, $\vec\gamma$ and the random variables are stochastically dependent.
However, the next lemma shows that this dependence is very weak.
Additionally, the lemma shows that the law of $\vec\gamma$ can be expressed easily in terms of the sequence $(\hat\vk_i)_{i\geq1}$ of independent copies of $\hat\vk$ from \eqref{eqSizeBiasd}.
Indeed, let
\begin{align}\label{eqhatGamma}
\hat\GAMMA_j&=\sum_{i=1}^{\vd_{n+2}}\vecone\{\hat\vk_i=j\}&\mbox{and}&&
\hat\GAMMA=(\hat\GAMMA_j)_{j\geq1}.
\end{align}
Also let $\hat\DELTA=(\hat\DELTA_j)_{j\geq1}$ be a family random variables, mutually independent and independent of everything else, with distributions
\begin{align}\label{eqhatDelta}
\hat\DELTA_j\sim\Po\bc{(1-\eps)\pr\brk{\vk=j}d/k}.
\end{align}
Further, let $\Sigma'$ be the $\sigma$-algebra generated by $\G'$, $\A'$, $\vM^-$ and $(\vd_i)_{i\in[n]}$.
We write $\GAMMA\mid\Sigma',\DELTA\mid\Sigma'$ for the conditional versions of $\GAMMA,\DELTA$ given $\Sigma'$.

\begin{lemma}\label{Cor_gamma}
With probability $1-\exp(-\Omega_{\eps,n}(1/\eps))$ over the choice of $\G'$, $\A'$, $\vM^-$ and $(\vd_i)_{i\in[n]}$ we have
$$\dTV(\GAMMA\mid\Sigma',\hat\GAMMA)+\dTV(\DELTA\mid\Sigma',\hat\DELTA)=O_{\eps,n}(\eps^{1/2}).$$
\end{lemma}
\begin{proof}
We begin by studying the unconditional distributions of $\GAMMA$ and $\DELTA$.
\newline
Let $\zeta=(\sum_{i\geq1}i\vM_i^+)/(\sum_{i=1}^{n+1}\vd_i)$.
Proceeding as in the proof of \Lem~\ref{Lemma_cavityCount}, we conclude that
$\pr\brk{1-2\eps\leq\zeta\leq1-\eps/2}=1-o_n(1)$.
Further, given $1-2\eps\leq\zeta\leq1-\eps/2$ we can think of $\G_{n+1,\vM^+}$ as being generated by the following experiment.
\begin{enumerate}[(i)]
\item Choose a set $\vec C\subset\bigcup_{h=1}^{n+1}\cbc{x_h}\times[\vd_h]$ of size $(1-\zeta) \sum_{i=1}^{n+1}\vd_i$ uniformly at random.
\item Create a random perfect matching $\vec\Gamma^\star$ of the complete bipartite graph with vertex classes
\begin{align*}
\bc{\bigcup_{h=1}^{n+1}\cbc{x_h}\times[\vd_h]}\setminus\vec C\quad\mbox{and}\quad
\bigcup_{i\geq1}\bigcup_{j=1}^{\vM_i^+}\cbc{a_{i,j}}\times[i].
\end{align*}
\item Obtain $\G^\star$ with variable nodes $x_1,\ldots,x_{n+1}$ and check nodes $a_{i,j}$, $i\geq1$, $j\in[\vM_i^+]$ by inserting an edge between $x_h$ and $a_{i,j}$ for any edge of $\vec\Gamma^\star$ that links $\{x_h\}\times[\vd_h]$ to $\{a_{i,j}\}\times[i]$.
\end{enumerate}
In other words, in the first step we designate the set of $\cC=\vec C$ of cavities and in the next two steps we connect the non-cavities randomly.

By way of this alternative description we can easily get a grip on the degree of $x_{n+1}$.
Indeed, given that $\vd_{n+1}\leq\eps^{-1/2}$, the probability that one of the clones $\{n+1\}\times[\vd_{n+1}]$ ends up in $\vC$ is $O_\eps(\eps^{1/2})$.
Hence, the actual degree $\vd^\star_{n+1}$ of $x_{n+1}$ in $\G^\star$ satisfies
\begin{align}\label{eqnasty-1}
\dTV\bc{\vd^\star_{n+1}\mid\{\vd_{n+1}\leq\eps^{-1/2}\},\vd}&=O_{\eps,n}(\eps^{1/2}).
\end{align}
Regarding the degrees of the checks adjacent to $x_{n+1}$,
by the principle of deferred decisions we can construct $\vec\Gamma^\star$ by matching one variable clone at a time, starting with the clones $\{x_{n+1}\}\times[\vd_{n+1}]$.
Clearly, in this process the probability that a specific clone of $x_{n+1}$ links to a specific check is proportional to the degree of that check.
Therefore, since $\Erw\sum_{i\geq1}i\vM_i^+=O_{\eps,n}(1)$, we find a fixed number $L$ such that with probability $1-O_{\eps,n}(\eps^{-1})$ all checks adjacent to $x_{n+1}$ have degree at most $L$.
Further, Chebyshev's inequality shows that $\vM_i^+=(1-\eps)\pr\brk{\vk=i}dn/k+o_n(n)$ for all $i\leq L$ and $\sum_{i\geq1}i\vM_i^+=(1-\eps)dn+o_n(n)$ \whp\
In effect, if $\vd_{n+1}\leq\eps^{-1/2}$, the $\vd_{n+1}$ choices of the checks are asymptotically independent, and the distribution of the individual check degrees that $x_{n+1}$ joins is at total variation distance $o_n(1)$ of the distribution $\hat\vk$.
In summary, given $\vM_i^+=(1-\eps)\pr\brk{\vk=i}dn/k+o_n(n)$ for all $i\leq L$ and $\sum_{i\geq1}i\vM_i^+=(1-\eps)dn+o_n(n)$ we have
\begin{align}\label{eqnasty0}
\dTV(\GAMMA,\hat\GAMMA)=O_{\eps,n}(\eps^{1/2}).
\end{align}
Moreover, it is immediate from \eqref{eqPoissons} that the unconditional $\DELTA$ is distributed as $\hat\DELTA$.

To complete the proof we are going to argue that $\vM^-,\vec d_1,\ldots,\vec d_n$ and $\GAMMA,\DELTA$ are asymptotically independent.
Arguing along the lines of the previous paragraph, we find that for large  $L=L(\eps)>0$ the event
	$$\cK=\cbc{\sum_{i\geq1}i(\DELTA_i+\GAMMA_i)\leq L}$$
occurs with probability $\pr[\cK]\geq1-\exp(-1/\eps^2)$.
Consequently, the event $$\cL=\{\pr[\cK\mid\vM^-,\vd_1,\ldots,\vd_n]\geq1-\exp(-1/\eps)\}$$satisfies $\pr[\cL]\geq1-\exp(-1/\eps)$.
Moreover, since $\vM$ comprises independent Poisson variables, the event
$$\cM\negmedspace=\negmedspace\cbc{\forall i \negmedspace\leq\negmedspace L:|\vM_i^--\Erw[\vM_i]|\negmedspace\leq\negmedspace\sqrt n\ln n}\cap\cbc{\sum_{i=1}^n\vd_i\negmedspace=\negmedspace(1-\eps)dn+o_n(n)}\cap\cbc{\sum_{i\geq1}i\vM_i^-\negmedspace=\negmedspace(1-\eps)dn+o_n(n)}$$
satisfies $\pr[\cM]=1-o_n(1)$.
In summary,
\begin{align}\label{eqnasty1}
\pr\brk{\cK}&\geq\exp(-1/\eps^2),&\pr\brk{\cL}&\geq1-\exp(-1/\eps),&\pr\brk{\cM\mid\cK}&=1-o_n(1).
\end{align}

Further, we claim that for any outcomes $(M^-,d_1,\ldots,d_n)\in\cL\cap\cM$ and $(\gamma,\Delta)\in\cK$,
\begin{align}\label{eqnasty2}
\pr\brk{\GAMMA=\gamma,\DELTA=\Delta\mid\vM^-=M^-,\forall i\in[n]:\vd_i=d_i}
	&\sim\pr\brk{\GAMMA=\gamma}\pr\brk{\DELTA=\Delta}.
\end{align}
Indeed, on the event $\cM$ we have $\vM_i^-=\Erw[\vM_i]+O_n(\sqrt n\ln n)=\Omega_n(n)$ for any $i\leq L$ in the support of $\vk$, the local limit theorem for the Poisson distribution yields
\begin{align}
\pr&\brk{\vec M^-=M^-,\forall i\leq n:\vec d_i=d_i\mid \GAMMA=\gamma,\DELTA=\Delta}
	=\pr\brk{\vec M=M^-+\gamma,\forall i\leq n:\vec d_i=d_i\mid \GAMMA=\gamma,\DELTA=\Delta}\nonumber\\
	&=\frac{\pr\brk{\GAMMA=\gamma,\DELTA=\Delta\mid \vec M=M^-+\gamma,\forall i\leq n:\vec d_i=d_i			}}{\pr\brk{\GAMMA=\gamma,\DELTA=\Delta}}\cdot\pr\brk{\vec M=M^-+\gamma}\cdot\prod_{i=1}^n\pr\brk{\vec d_i=d_i}\nonumber\\
	&=(1+o_n(1)) \frac{\pr\brk{\GAMMA=\gamma\mid \vec M=M^-+\gamma,\forall i\leq n:\vec d_i=d_i,\DELTA=\Delta			}}{\pr\brk{\GAMMA=\gamma}}\cdot\pr\brk{\vec M=M^-}\cdot\prod_{i=1}^n\pr\brk{\vec d_i=d_i}.\label{eqnasty3}
\end{align}
Finally, given $\vM=M^-+\gamma$ and $\vec\Delta=\Delta$ we have
$\vM_i^+=(1-\eps)\pr\brk{\vk=i}dn/k+o_n(n)$ for all $i\leq L$ and $\sum_{i\geq1}i\vM_i^+=(1-\eps)dn+o_n(n)$.
Therefore, by the principle of deferred decisions, once we condition on a likely outcomes $M^-$ of $\vM^-$, $\vd_1,\ldots,\vd_n$ and of $\vec\Delta$, the conditional probability of obtaining $\vec\gamma=\gamma$ is close to the unconditional probability:
\begin{align*}
\pr\brk{\GAMMA=\gamma\mid \vec M=M^-+\gamma,\forall i\leq n:\vec d_i=d_i,\DELTA=\Delta			}=(1+o_n(1)) \pr\brk{\GAMMA=\gamma}.
\end{align*}
Hence, \eqref{eqnasty2} follows from \eqref{eqnasty1} and \eqref{eqnasty3}.

Finally, to complete the proof we combine \eqref{eqnasty1} and \eqref{eqnasty2} to conclude that
with probability $1-\exp(-\Omega_{\eps,n}(1/\eps))$,
\begin{align}
\pr\brk{\GAMMA=\gamma,\DELTA=\Delta\mid\Sigma'}&=
\pr\brk{\GAMMA=\gamma,\DELTA=\Delta\mid\vM^-,\vd_1,\ldots,\vd_n}=(1+o_n(1)) \pr\brk{\GAMMA=\gamma}\pr\brk{\DELTA=\Delta}.\label{eqnasty5}
\end{align}
The assertion follows from \eqref{eqnasty0} and \eqref{eqnasty5}.
\end{proof}

\subsection{Proof of \Lem~\ref{Lemma_A'''}}\label{Sec_A'''}
The proof comprises several steps, each relatively simple individually.
Let
	$$X=\sum_{i\geq1}\DELTA_i,\qquad Y=\sum_{i\geq1}i\DELTA_i,\qquad Y'=\sum_{i\geq1}i\GAMMA_i.$$
Then the total number of new non-zero entries upon going from $\vA'$ to $\vA'''$ is bounded by $Y+Y'$.
Let $$\cE''=\cbc{X\vee Y\vee Y'\leq1/\eps}.$$

\begin{claim}\label{Claim_A'''1}
We have $\pr\brk{\cE''}=1-O_{\eps,n}(\eps)$.
\end{claim}
\begin{proof}
Since \eqref{eqPoissons} yields $\Erw[X],\Erw[Y]=O_{\eps,n}(1)$, Markov's inequality yields $\pr\brk{X>1/\eps}=O_{\eps,n}(\eps)$ and  $\pr\brk{Y>1/\eps}=O_{\eps,n}(\eps)$.
Further, 
we can bound the probability that a check of degree $i$ is adjacent to $x_{n+1}$ by $i\vd_{n+1}/n$, because one of the $i$ clones of the check has to be matched to one of the $\vd_{n+1}$ clones of $x_{n+1}$ and $\sum_{i=1}^n\vd_i\geq n$.
Hence,
\begin{align*}
\Erw\brk{Y'}=\Erw\sum_{i\geq1}i\GAMMA_i\leq\Erw\sum_{i\in[\vm_{\eps,n}^+]}\vk_i^2\vd_{n+1}/n=O_{\eps,n}(1).
\end{align*}
Thus, the assertion follows from Markov's inequality.
\end{proof}

Going from $\G'$ to $\G'''$ we add checks $a_{i,j}'''$, $i\geq1$, $j\in[\GAMMA_i]$ and $b_{i,j}'''$, $i\geq1$, $j\in[\vM_i^+-\vM_i^--\GAMMA_i]$.
Let 
$$\cX=\bc{\bigcup_{i\geq1}\bigcup_{j=1}^{\GAMMA_i}\partial a_{i,j}'''\setminus\{x_{n+1}\}}\cup 
		\bc{\bigcup_{i\geq1}\bigcup_{j\in[\vM_i^+-\vM_i^--\GAMMA_i]}\partial b_{i,j}'''}$$
comprise all the variable nodes adjacent to the new checks, except for $x_{n+1}$.
Further, let
$$\cE'''=\cbc{|\cX|=Y+\sum_{i\geq1}(i-1)\GAMMA_i}$$
be the event that the variables of $\G'$ where the new checks attach are all distinct.

\begin{claim}\label{Claim_A'''2b}
We have $\pr\brk{\cE'''\mid\cE'\cap\cE''}=1-o_n(1).$
\end{claim}
\begin{proof}
Given $\cE'$ there are $\Omega_n(n)$ cavities in total, while the maximum number belonging to any one variable is $O_n(\sqrt n)$.
Further, given $\cE''$ we merely pick a bounded number $Y+Y'=O_{\eps,n}(1/\eps)$ of these cavities randomly as neighbours of the new checks.
Thus, the probability of hitting the same variable twice is $o_n(1)$.
\end{proof}

\begin{claim}\label{Claim_A'''3}
We have $\Erw\brk{\abs{\nul(\vA''')-\nul(\vA')}(1-\vecone\cE\cap\cE'\cap\cE''\cap\cE''')}=o_{\eps,n}(1)$.
\end{claim}
\begin{proof}
Clearly $\abs{\nul(\vA''')-\nul(\vA')}\leq X+\vd_{n+1} +1$ because going from $\vA'$ to $\vA'''$ we add one column and at most $X+\vd_{n+1}$ new rows.
Consequently, as $\Erw[X^2],\Erw[\vd_{n+1}^2]=O_{\eps,n}(1)$, the Cauchy-Schwarz inequality yields
\begin{align}
\Erw\brk{\abs{\nul(\vA''')-\nul(\vA')}(1-\vecone\cE'')}&\leq
		\Erw\brk{(X+\vd_{n+1}+1)^2}^{1/2}\bc{1-\pr\brk{\cE''}}^{1/2}=
o_{\eps,n}(1).
	\label{eqClaim_A'''3_1}
\end{align}
Furthermore, \Lem~\ref{Lemma_theta} and Claims~\ref{Claim_A'''2a}--\ref{Claim_A'''2b} readily imply that
\begin{align}	\label{eqClaim_A'''3_2}
\Erw&\brk{\abs{\nul(\vA''')-\nul(\vA')}\vecone\cE''\setminus\cE}	
	\leq O_{\eps,n}(\eps^{-1})\exp(-1/\eps^4)=o_{\eps,n}(1),\\
\Erw&\brk{\abs{\nul(\vA''')-\nul(\vA')}\vecone\cE''\setminus\cE'},\Erw\brk{\abs{\nul(\vA''')-\nul(\vA')}\vecone\cE''\cap\cE'\setminus\cE'''}=o_n(1).\label{eqClaim_A'''3_3}
\end{align}
The assertion follows from \eqref{eqClaim_A'''3_1}--\eqref{eqClaim_A'''3_3}.
\end{proof}

We obtain $\G'''$ by adding checks $a_{i,j}'''$ adjacent to $x_{n+1}$ and $b_{i,j}'''$ not adjacent to $x_{n+1}$.
Recall that $\vec\alpha$ signifies the fraction of frozen cavities.
Further, let $\Sigma''\supset\Sigma'$ be the $\sigma$-algebra generated by  $\G'$, $\A'$, $\vM_-$, $(\vd_i)_{i\in[n+1]}$, $\GAMMA,\vM$ and $\DELTA$.
The random variable $\ALPHA$ and the events $\cE,\cE',\cE''$ are $\Sigma''$-measurable, but $\cE'''$ is not.
Indeed, given $\Sigma''$ the specific cavities of $\G'$ that the new checks $a_{i,j}''',b_{i,j}'''$ join are still random.

\begin{claim}\label{Claim_A'''6}
On the event 
$\cE\cap\cE'\cap\cE''$ we have
\begin{align*}
\Erw\brk{\bc{\nul(\vA''')\negmedspace-\negmedspace\nul(\vA')}\vecone\cE'''\negmedspace\mid\Sigma''}\negmedspace=\negmedspace
	o_{\eps,n}(1)\negmedspace+\negmedspace\prod_{i\geq1}(1\negmedspace-\negmedspace\ALPHA^{i-1})^{\GAMMA_i}
		\negmedspace-\negmedspace\sum_{i\geq1}\negmedspace(1\negmedspace-\negmedspace\vec\alpha^{i-1})\GAMMA_i
	\negmedspace-\negmedspace\sum_{i\geq1}\negmedspace(1-\vec\alpha^{i})(\vM_i^+\negmedspace-\negmedspace\vM_i^-\negmedspace-\negmedspace\GAMMA_i).
\end{align*}
\end{claim}
\begin{proof}
Let
\begin{align*}
\cA=\cbc{a_{i,j}''':i\geq1,\ j\in[\GAMMA_i]}
\end{align*}
 be the set of all the new checks connected to $x_{n+1}$ and let
\begin{align*}
\cB=\cbc{b_{i,j}''':i\geq1,\ j\in[\vM_i^+-\vM_i^--\GAMMA_i]}
\end{align*}
 be the set of all the new checks not connected to $x_{n+1}$.
Let $\tilde \vB$ be the $\cbc{0,1}$-matrix whose rows are indexed by $\cA\cup\cB$ and whose columns are indexed by $V_n=\{x_1,\ldots,x_n\}$ such that for each $a\in\cA\cup\cB'$ and each $x\in V_n$ the corresponding entry equals one iff $x\in\partial_{\G'''} a$.
Further, obtain $\vB$ from $\tilde\vB$ by replacing each one-entry by the entry supplied by $\chi$ that represents the respective new edge of the Tanner graph.
If the event $\cE'''$ occurs, then each column of $B$ contains at most one non-zero entry and each row contains at least one non-zero entry.
In effect,  $\vB$ has full rank, i.e.,
\begin{align*}
 \rk(\vB)=|\cA\cup\cB|=\sum_{i\geq1}\vM_i^+-\vM_i^-.
\end{align*}
Further, let $\vB_*$ be the matrix obtained from $\vB$ by replacing all entries in the $x$-column by zero for every $x\in\fF(\vA')$.
Finally, let $\vC\in\FF^{\cA\cup\cB}$ be a column vector whose entries $\vC_a$, $a\in\cA$, are the entries from $\chi$ representing the edges of the Tanner graph $\G'''$ incident with $x_{n+1}$ and whose remaining entries $\vC_b$, $b\in\cB$, are equal to zero.

By construction, on the event $\cE\cap\cE'\cap\cE''\cap\cE'''$ we have
\begin{align*}
\nul\vA'''&=\nul\begin{pmatrix}\vA'&0\\\vB&\vC\end{pmatrix}.
\end{align*}
Moreover, on $\cE'$ the set $\cX'''$ of non-zero columns of $\vB$ has size at most $|\cX'''|\leq Y+Y'\leq2/\eps$, while there are at least $\eps dn/2$ cavities.
As a consequence, even though the sequence of cavities that the new checks join are drawn without replacement, this sequence is at total variation distance $o_n(1)$ from a sequence of independent samples from the distribution $P$.
Therefore, on $\cE\cap\cE'\cap\cE''$ the conditional probability given $\cE'''$ that $\cX'''$ forms a proper relation is bounded by $O_{\eps,n}(\exp(-1/\eps^4))$.
Hence, \Lem~\ref{Cor_free} implies that on $\cE\cap\cE'\cap\cE''$,
\begin{align}\label{eqClaim_A'''6_2}
\Erw\brk{\bc{\nul(\vA''')-\nul(\vA')}\vecone\cE'''\mid\Sigma''}&=1-\Erw\brk{\rk\bc{\vB_*\ \vC}\mid\Sigma''}+o_{\eps,n}(1).
\end{align}

On $\cE'''$ the matrix $\vQ=\bc{\vB_*\ \vC}$ is a block matrix that decomposes into the $\cA$-rows $\vQ_{\cA}$ and the $\cB$-rows $\vQ_{\cB}$.
Hence, $\rk(\vQ)=\rk(\vQ_{\cA})+\rk(\vQ_{\cB})$.
To complete the proof, we claim that
\begin{align}\label{eqQ'}
\Erw\brk{\rk\bc{\vQ_{\cB}}\mid\Sigma''}&=o_n(1)+\sum_{i\geq1}\bc{1-\ALPHA^i}(\vM_i^+-\vM_i^--\GAMMA_i),\\
\Erw\brk{\rk(\vQ_{\cA})\mid\Sigma''}&=o_n(1)+
	\sum_{i\geq1}\bc{1-\ALPHA^{i-1}}\GAMMA_i+1-\prod_{i\geq1}\bc{1-\ALPHA^{i-1}}^{\GAMMA_i},\label{eqQ''}
\end{align}
where, as we recall, $\ALPHA$ is the probability that a cavity chosen from $p(\nix)$ is frozen.
Indeed, the probability that a $\cB$-row of $\vB$ that contains precisely $i$ non-zero entries gets zeroed out completely in $\vB_*$ equals $\vec\alpha^i+o_n(1)$ and there are $\vM_i^+-\vM_i^--\GAMMA_i$ such rows; hence \eqref{eqQ'}.

Similarly, the probability that an $\cA$-row of $\vB$ with $i-1$ non-zero entries gets zeroed out completely in $\vB_*$ equals $\vec\alpha^{i-1}+o_n(1)$ and there are $\GAMMA_i$ such rows.
Hence, the expected rank of the $\cA$-rows of $\vB_*$ equals $\sum_{i\geq1}\bc{1-\ALPHA^{i-1}}\GAMMA_i+o_n(1)$, which is the first summand in \eqref{eqQ''}.
Moreover, the presence of the $\vC$-column adds one to the rank of $\vQ_{\cA}$ unless not a single one of the $\cA$-rows of $\vB$ gets zero out, 
which occurs with probability $\prod_{i\geq1}\bc{1-\ALPHA^{i-1}}^{\GAMMA_i}+o_n(1)$.
Hence,  we obtain \eqref{eqQ''}.
Finally, the assertion follows from \eqref{eqClaim_A'''6_2}--\eqref{eqQ''}.
\end{proof}

\begin{proof}[Proof of \Lem~\ref{Lemma_A'''}]
Let $\fE=\cE\cap\cE'\cap\cE''\cap\cE'''$.
Combining Claims~\ref{Claim_A'''1}--\ref{Claim_A'''6}, we see that 
$$
\Erw\abs{\Erw\brk{\nul(\vA''')-\nul(\vA')\mid\Sigma''}-
	\bc{\prod_{i\geq1}(1-\ALPHA^{i-1})^{\GAMMA_i}
		-\sum_{i\geq1}(1-\vec\alpha^{i-1})\GAMMA_i
	-\sum_{i\geq1}(1-\vec\alpha^{i})(\vM_i^+-\vM_i^--\GAMMA_i)}\vecone\fE}$$
\begin{align}\label{eqLemma_A'''_0}
=o_{\eps,n}(1).
\end{align}
Since on $\fE$ all degrees $i$ with $\vM_i^+-\vM_i^--\GAMMA_i>0$ are bounded and Chebyshev's inequality shows that  $\vM_i\sim\Erw[\vM_i]=\Omega_n(n)$ for any fixed $i$ \whp, \eqref{eqminus} yields $\vM_i^-=\vM_i-\GAMMA_i$ \whp\
Hence, \eqref{eqLemma_A'''_0} turns into
\begin{align}\label{eqLemma_A'''_1}
\Erw\abs{\Erw\brk{\nul(\vA''')-\nul(\vA')\mid\Sigma''}-
	\bc{\prod_{i\geq1}(1-\ALPHA^{i-1})^{\GAMMA_i}
		-\sum_{i\geq1}(1-\vec\alpha^{i-1})\GAMMA_i
	-\sum_{i\geq1}(1-\vec\alpha^{i})\DELTA_i}\vecone\cE''}&=o_{\eps,n}(1).
\end{align}
Further, since $\sum_{i\geq1}\GAMMA_i\leq\vd_{n+1}$ and $\Erw[\vd_{n+1}]=O_{\eps,n}(1)$, we obtain
\begin{align}
\Erw&\brk{
	\bc{\prod_{i\geq1}(1-\ALPHA^{i-1})^{\GAMMA_i}
		-\sum_{i\geq1}(1-\vec\alpha^{i-1})\GAMMA_i}\vecone\fE}\nonumber\\
&=\Erw\brk{
	\bc{\prod_{i\geq1}(1-\ALPHA^{i-1})^{\GAMMA_i}
		-\sum_{i\geq1}(1-\vec\alpha^{i-1})\GAMMA_i}\vecone\fE\cap\cbc{\sum_{i\geq1}\GAMMA_i\leq\eps^{-1/4}}}+o_{\eps,n}(1)\nonumber\\
&=\Erw\brk{
	\bc{\prod_{i\geq1}(1-\ALPHA^{i-1})^{\GAMMA_i}
		-\sum_{i\geq1}(1-\vec\alpha^{i-1})\GAMMA_i}\vecone\cbc{\sum_{i\geq1}\GAMMA_i\leq\eps^{-1/4}}}+o_{\eps,n}(1)
\hspace{4pt}\mbox{
	\begin{tiny} [by \Lem s~\ref{Lemma_theta}--\ref{Claim_A'''2a}/Claims~\ref{Claim_A'''1}--\ref{Claim_A'''2b}]\end{tiny}}\nonumber\\
&=\Erw\brk{
	\bc{\prod_{i\geq1}(1-\ALPHA^{i-1})^{\hat\GAMMA_i}
		-\sum_{i\geq1}(1-\vec\alpha^{i-1})\hat\GAMMA_i}\vecone\cbc{\sum_{i\geq1}\hat\GAMMA_i\leq\eps^{-1/4}}}+o_{\eps,n}(1)
\hspace{8pt}\mbox{[by \Lem~\ref{Cor_gamma}]}\nonumber\\
&=\Erw\brk{(1-\ALPHA^{\hat\vk-1})^{\vd}-\vd-\vd\ALPHA^{\hat\vk-1}}+o_{\eps,n}(1)
	\hspace{8pt}\mbox{[by the def.\ of $\hat\GAMMA$]}
		\nonumber\\
&	=\Erw\brk{D(1-K'(\ALPHA)/k)-d-\frac dk K'(\ALPHA)}+o_{\eps,n}(1)
	\hspace{8pt}\mbox{[by \eqref{eqSizeBiasd}].}\label{eqAddingVar100}
\end{align}
Similarly, Claim~\ref{Claim_A'''1}, \Lem~\ref{Cor_gamma} and the construction \eqref{eqhatDelta} of $\hat\DELTA$ yield
\begin{align}
&\Erw\brk{\bc{\sum_{i\geq1}(1\negmedspace-\negmedspace\vec\alpha^{i})\DELTA_i}\vecone\cE''}\negmedspace=\negmedspace
\Erw\brk{\bc{\sum_{i\geq1}(1\negmedspace-\negmedspace\vec\alpha^{i})\DELTA_i}\vecone\cbc{\sum_{i\geq1}\DELTA_i\leq\eps^{-1/3}}}\negmedspace+\negmedspace o_{\eps,n}(1)
\negmedspace=\negmedspace\Erw\brk{\sum_{i\geq1}(1\negmedspace-\negmedspace\vec\alpha^{i})\hat\DELTA_i}\negmedspace+\negmedspace o_{\eps,n}(1)\nonumber\\
&=o_{\eps,n}(1)+(1-\eps)\frac dk\sum_{i\geq1}\pr\brk{\vk=i}\Erw[1-\vec\alpha^{i}]
=o_{\eps,n}(1)+\frac dk-\frac dk\Erw[K(\ALPHA)].\label{eqAddingVar101}
\end{align}
Finally, the assertion follows from \eqref{eqLemma_A'''_1}, \eqref{eqAddingVar100} and \eqref{eqAddingVar101}.
\end{proof}

\subsection{Proof of \Lem~\ref{Lemma_A''}}\label{Sec_A''}
The argument resembles the one from the proof of \Lem~\ref{Lemma_A'''} but the details are considerably more straightforward as we merely add checks to obtain $\vA''$ from $\vA'$.
As before we consider the events $\cE,\cE'$ from \eqref{eqE} and \eqref{eqE'}
Moreover, recalling that the total number of new non-zero entries when going from $\vA'$ to $\vA''$ is bounded by $\vd_{n+1}$,
we introduce $\cE''=\cbc{\vd_{n+1}\leq1/\eps}.$

\begin{claim}\label{Claim_A''_0}
We have $\pr\brk{\cE''}=1-O_{\eps,n}(\eps^2).$
\end{claim}
\begin{proof}
This follows from the assumption $\Erw[\vd_{n+1}^2]=O_{\eps,n}(1)$ and Chebyshev's inequality.
\end{proof}

Further, similarly as in the proof of \Lem~\ref{Lemma_A'''} we consider the set
$$\cX=\bigcup_{i\geq1}\bigcup_{j\in[\vM_i-\vM_i^-]}\partial_{\G''}a_{i,j}''$$ 
of variable nodes adjacent to the new checks. 
Let $\cE'''$ be the event that none of the variable nodes in $\cX$ is connected with the set of new checks by more than one edge.

\begin{claim}\label{Claim_A''_E'''}
We have 
$\pr\brk{\cE'''\mid\cE'\cap\cE''}=1-o_n(1).$
\end{claim}
\begin{proof}
Given $\cE'$ there are $\Omega_n(n)$ cavities in total, with each variable node contributing no more than $O_n(\sqrt n)$ cavities.
Moreover, given $\cE''$ we choose $O_{\eps,n}(1/\eps)$ of cavities randomly to attach the new checks.
Consequently, the probability of twice choosing a cavity with the same underlying variable is $o_n(1)$.
\end{proof}

\begin{claim}\label{Claim_A''_1}
We have $\Erw\brk{\abs{\nul(\vA'')-\nul(\vA')}(1-\vecone\cE\cap\cE'\cap\cE''\cap\cE''')}=o_{\eps,n}(1)$.
\end{claim}
\begin{proof}
We have $\abs{\nul(\vA'')-\nul(\vA')}\leq \vd_{n+1}$ as we add at most $\vd_{n+1}$ rows.
Because $\Erw[\vd_{n+1}]=O_{\eps,n}(1)$, Claim~\eqref{Claim_A''_0} and the Cauchy-Schwarz inequality yield
\begin{align}\label{eqClaim_A''_1_1}
\Erw\brk{\abs{\nul(\vA'')-\nul(\vA')}(1-\vecone\cE'')}&\leq
		\Erw\brk{\vd_{n+1}^2}^{1/2}(1-\pr\brk{\cE})^{1/2}
			=o_{\eps,n}(1).
\end{align}
Moreover, 
\Lem~\ref{Lemma_theta}, \Lem~\ref{Claim_A'''2a} and Claim~\ref{Claim_A''_E'''}
show that

$$\Erw\brk{\abs{\nul(\vA'')-\nul(\vA')}\vecone\cE''\setminus\cE},
\Erw\brk{\abs{\nul(\vA'')-\nul(\vA')}\vecone\cE''\setminus\cE'},
\Erw\brk{\abs{\nul(\vA'')-\nul(\vA')}\vecone\cE''\setminus\cE'''}$$
\begin{align}\label{eqClaim_A''_1_2}
=o_{\eps,n}(1).
\end{align}
The assertion follows from \eqref{eqClaim_A''_1_1} and\eqref{eqClaim_A''_1_2}.
\end{proof}

The matrix $\vA''$ results from $\vA'$ by adding checks $a_{i,j}''$, $i\geq1$, $j\in[\vM_i-\vM_i^-]$ that are connected to random cavities of $\A'$.
Moreover,  as before let $\Sigma''\supset\Sigma'$ be the $\sigma$-algebra generated by  $\G'$, $\A'$, $\vM_-$, $(\vd_i)_{i\in[n+1]}$, $\GAMMA,\vM$ and $\DELTA$.
Then $\cE,\cE',\cE''$ are $\Sigma''$-measurable, but $\cE'''$ is not.

\begin{claim}\label{Claim_A''_6}
On the event $\cE\cap\cE'\cap\cE''$ we have
$\Erw\brk{(\nul(\vA'')-\nul(\vA'))\vecone\cE'''\mid{\Sigma''}}=o_{\eps,n}(1)-\sum_{i\geq1}(1-\vec\alpha^{i})(\vM_i-\vM_i^-).$
\end{claim}
\begin{proof}
Let $\cA$ be the set of all the new checks $a_{i,j}''$, $i\geq1$, $j\in[\vM_i-\vM_i^-]$.
Let $\tilde\vB$ be the $\cbc{0,1}$-matrix whose rows are indexed by $\cA$ and whose columns are indexed by $V_n=\{x_1,\ldots,x_n\}$ such that for each $a\in\cA$ and each $x\in V_n$ the corresponding entry equals one iff $x\in\partial_{\G''} a$.
Further, obtain $\vB$ by substituting each one-entry of $\tilde\vB$ by the appropriate non-zero field element from  $\chi$.
If $\cE'''$ occurs, then $\vB$ has rank $\rk(\vB)=|\cA|=\sum_{i\geq1}\vM_i^+-\vM_i$, because no column contains two non-zero entries and each row contains at least one non-zero entry.
Further, let $\vB_*$ be the matrix obtained from $\vB$ by replacing all entries in the $x$-column by zero if $x\in\fF(\vA')$ is frozen to zero in $\vA'$.

On the event $\cE\cap\cE'\cap\cE''\cap\cE'''$ we have
\begin{align}\label{eqClaim_A''6_1}
\nul\vA''&=\nul\begin{pmatrix}\vA'\\\vB\end{pmatrix}.
\end{align}
Moreover, on $\cE'\cap\cE''$ the set $\cX''$ of non-zero columns of $\vB$ has size at most $|\cX''|\leq \vd_{n+1}\leq1/\eps$, while there are at least $\eps dn/2$ cavities.
Hence, on $\cE\cap\cE'\cap\cE''\cap\cE'''$ the probability that $\cX''$ forms a proper relation is bounded by $\exp(-1/\eps^4)$.
Therefore, \Lem~\ref{Cor_free} implies that
\begin{align}\label{eqClaim_A''6_2}
\Erw\brk{\bc{\nul(\vA'')-\nul(\vA')}\vecone\cE'''\mid\Sigma''}&=o_{\eps,n}(1)-\Erw\brk{\rk\bc{\vB_*}\mid{\Sigma''}}.
\end{align}
Further, since an $\ALPHA$-fraction of cavities are frozen, a row of $\vB$ with $i$ non-zero entries gets zeroed out completely in $\vB_*$ with probability $\ALPHA^i+o_n(1)$.
Consequently,
\begin{align}\label{eqClaim_A''6_3}
\Erw\brk{\rk\bc{\vB_*}\mid\Sigma''}&=o_{\eps,n}(1)+\sum_{i\geq1}\bc{1-\ALPHA^i}(\vM_i-\vM_i^-).
\end{align}
Finally, the assertion follows from \eqref{eqClaim_A''6_2} and \eqref{eqClaim_A''6_3}.
\end{proof}

\begin{proof}[Proof of \Lem~\ref{Lemma_A''}]
Let $\fE=\cE\cap\cE'\cap\cE''\cap\cE'''$.
Combining Claims~\ref{Claim_A''_1}--\ref{Claim_A''_6}, we obtain
\begin{align}\label{eqAddingVar01110101}
\Erw\abs{\Erw[\nul(\vA'')-\nul(\vA')\mid{\Sigma''}]+\bc{\sum_{i\geq1}(1-\vec\alpha^{i})(\vM_i-\vM_i^-)}\vecone\fE}	&=o_{\eps,n}(1).
\end{align}
Since on $\fE$ all degrees $i$ with $\vM_i^+-\vM_i^->0$ are bounded \whp\ and $\vM_i^-=\Omega_n(n)$ \whp, we conclude that $\vM_i-\vM_i^-=\GAMMA_i$ for all $i\geq1$ \whp\
Hence,  \eqref{eqAddingVar01110101} turns into
\begin{align}\label{eqAddingVar011}
\Erw\abs{\Erw[\nul(\vA'')-\nul(\vA')\mid{\Sigma''}]+\bc{\sum_{i\geq1}(1-\vec\alpha^{i})\GAMMA_i}\vecone\fE}	&=o_{\eps,n}(1).
\end{align}
Further, because $\sum_{i\geq1}\GAMMA_i\leq\vd_{n+1}$ and $\Erw[\vd_{n+1}]=O_{\eps,n}(1)$,
\begin{align}
\Erw\brk{\bc{\sum_{i\geq1}(1-\vec\alpha^{i})\GAMMA_i}\vecone\fE}&=
\Erw\brk{\bc{\sum_{i\geq1}(1-\vec\alpha^{i})\GAMMA_i}\vecone\cbc{\sum_{i\geq1}\GAMMA_i\leq\eps^{-1/4}}}+o_{\eps,n}(1)\nonumber&&\mbox{[by Claim~\ref{Claim_A''_0}]}\\
&=
\Erw\brk{\bc{\sum_{i\geq1}(1-\vec\alpha^{i})\hat\GAMMA_i}\vecone\cbc{\sum_{i\geq1}\hat\GAMMA_i\leq\eps^{-1/4}}}+o_{\eps,n}(1)
&&\mbox{[by \Lem~\ref{Cor_gamma}]}\nonumber\\
&=d\Erw[1-\ALPHA^{\hat\vk}]+o_\eps(1)=-  d\Erw[\vec\alpha K'(\vec\alpha)]/k+d+o_{\eps,n}(1)
	&&\mbox{[by \eqref{eqSizeBiasd}].}
\label{eqAddingVar012}
\end{align}
The assertion follows from \eqref{eqAddingVar011} and~\eqref{eqAddingVar012}.
\end{proof}

\subsection{Proof of \Lem~\ref{Lemma_valid}}\label{Sec_Lemma_valid}
Once more we break the proof down into a few relatively simple steps.

\begin{claim}\label{Claim_Lemma_valid_1}
We have  $\Erw[\nul(\vA'')]=\Erw[\nul(\vA_{n,\vM})]+o_n(1)$.
\end{claim}
\begin{proof}
The choice of the random variables in (\ref{eqPoissons}) and \Lem~\ref{Lemma_sums} ensure that the event $\cE=\{\sum_{i\geq1}i\vM_i\leq dn/k\}$ has probability $1-o_n(1/n)$.
Further, given $\cE$ the random variables $\nul(\vA'')$ and $\nul(\vA_{n,\vM})$ are identically distributed by the principle of deferred decisions.
Because the nullity of either matrix is bounded by $n$ deterministically, the claim follows.
\end{proof}

To compare $\nul(\vA''')$ and $\nul(\vA_{n+1,\vM^+})$ we consider the event
		$$\cE^+=\cbc{\frac{dn}{2k}\leq\sum_{i\geq1}i\vM_i^+\leq \sum_{i=1}^n\vd_i,\forall i\geq  n/\ln^9n:\vM_i^+=0}.$$

\begin{claim}\label{Claim_Lemma_valid_2}
We have  $\pr\brk{\cE^+}=1-o_n(1/n)$.
\end{claim}
\begin{proof}
This follows from the definition \eqref{eqm} of the random variables $\vM_i^+$ and \Lem~\ref{Lemma_sums}.
\end{proof}

Further, consider the event
	$$\cW=\cbc{\vd_{n+1}\leq\ln n,\ \sum_{i\geq1}i(\DELTA_i+\GAMMA_i)<\ln^4 n}.$$

\begin{claim}\label{Claim_Lemma_valid_3}
We have  $\pr\brk{\cW}=1-o_n(1)$.
\end{claim}
\begin{proof}
This follows from the assumption that $\Erw[\vd^2],\Erw[\vk^2]$ are bounded.
\end{proof}

Moreover, let $\cU$ be the event that $x_{n+1}$ does not partake in any multi-edges of $\G_{n,\vM^+}$.

\begin{claim}\label{Claim_Lemma_valid_4} 
We have $\pr\brk{\cU\mid\cW\cap\cE^+}=1-o_n(\ln^{-6}n)$.
\end{claim}
\begin{proof}
Given $\cW\cap\cE^+$ variable node $x_{n+1}$ has target degree at most $\ln n$ and
all check degrees are bounded by $n/\ln^9n$.
Hence, the probability that $x_{n+1}$ joins the same check twice is $O_n(\ln^{-7}n)$.
\end{proof}

The next claim shows that $\nul(\vA''')$, $\nul(\vA_{n+1,\vM^+})$ can be coupled identically on the `bulk' event $\cE^+\cap\cU\cap\cW$.

\begin{claim}\label{Claim_Lemma_valid_5} 
Given $\cE^+\cap\cU\cap\cW$  the random variables $\nul(\vA''')$ and $\nul(\vA_{n+1,\vM^+})$ are identically distributed and thus
\begin{align}\label{eqBigBound}
\Erw\brk{\bc{\nul(\vA''')-\nul(\vA_{n+1,\vM^+})}\vecone\cU\cap\cW\cap\cE^+}&=0.
\end{align}
\end{claim}
\begin{proof}
By construction, on $\cE^+\cap\cU\cap\cW$ the random matrices $\vA'''$ and $\vA_{n+1,\vM^+}$ are identically distributed, and hence so are their nullities.
\end{proof}

In light of Claims~\ref{Claim_Lemma_valid_2} and~\ref{Claim_Lemma_valid_5} we are left to bound the difference of the nullities on $\cE^+\setminus(\cU\cap\cW)$.

\begin{claim}\label{Claim_Lemma_valid_6} 
There is a coupling of $\vA_{n+1,\vM^+}$ and $\vA'''$ on $\cE^+$ such that
$\abs{\nul(\vA''')-\nul(\vA_{n+1,\vM^+})}\leq 2\sum_{i\geq1}i(\DELTA_i+\GAMMA_i).$
\end{claim}
\begin{proof}
We estimate the number of edges of the Tanner graph $\G_{n+1,\vM^+}$ incident with the checks $a_{i,j}$, $\vM_i^-<j\leq\vM_i^+$ or the new variable $x_{n+1}$ of $\G_{n+1,\vM^+}$.
By construction, there are at most $\sum_{i\geq1}i(\DELTA_i+\GAMMA_i)$ such edges.
Similarly, there are no more than $\sum_{i\geq1}i(\DELTA_i+\GAMMA_i)$ edges incident with the new checks $a_{i,j}'''$, $b_{i,j}'''$ added to $\vA'$ to obtain $\vA'''$.
By the principle of deferred decisions on $\cE''$ we can couple the Tanner graphs of $\vA'''$ and $\vA_{n+1,\vM^+}$ such that they coincide on all the edges that join variables $x_1,\ldots,x_n$ and checks $a_{i,j}$, $j\leq\vM_i^-$, and hence the matrices themselves so that they coincide on all the corresponding matrix entries.
Consequently, $\vA'''$ and $\vA_{n+1,\vM^+}$ differ in no more than $2\sum_{i\geq1}i(\DELTA_i+\GAMMA_i)$ entries.
\end{proof}

We proceed to bound the difference of the nullities on $\cE^+\setminus\cW$.

\begin{claim}\label{Claim_Q123}
We have $\Erw\brk{\sum_{i\geq1}i(\DELTA_i+\GAMMA_i)\vecone{\cE^+\setminus\cW}}=o_n(1)$.
\end{claim}
\begin{proof}
 The event $\cE^+\setminus\cW$ is contained in the union of the three events
\begin{align*}
\cQ_1&=\cE^+\cap\cbc{\exists i>\log n:\GAMMA_i>0},&
\cQ_2&=\cE^+\cap\cbc{\vd_{n+1}>\log n}\setminus\cQ_1,&
\end{align*}
$$
\cQ_3=\cE^+\cap\cbc{\sum_{i\geq1}i\DELTA_i>\ln^3n}\setminus(\cQ_1\cup\cQ_2).
$$
To bound the contribution of $\cQ_1$, consider $\vm_{\eps,n}^+=\sum_{i\geq1}\vM_i^+\disteq\Po((1-\eps)d(n+1)/k)$.
We claim that, with the copies $(\vk_i)_{i\geq1}$ of $\vk$ independent of everything else,
$$
\Erw\brk{\sum_{i\geq1}i\GAMMA_i\vecone\cQ_1}\leq O_n(1/n)\cdot\bc{1+\Erw\brk{\sum_{i=1}^{\vm_{\eps,n}^+}\vecone\{\vk_i\geq\log n\}\vk_i^2\vec d_{n+1}}}=$$
\begin{align}\label{eqQuickBound10}
O_n(1)\cdot \pr\brk{\vk\geq\log n}+O_n(1/n)=O_n(\log^{-2}n).
\end{align}
Indeed, the last equality sign follows from the first because $\Erw[\vk^2]=O_n(1)$ and the first equality sign follows because $\vm_{\eps,n}^+$ is independent of $\vd_{n+1}$ and the $\vk_i$.
Further, to obtain the first inequality we consider the $\vm_{\eps,n}^+$ checks one by one.
The degree of the $i$th check is distributed as $\vk_i$.
We discard it unless $\vk_i\geq\log n$.
But if $\vk_i\geq\log n$, then the probability that $\vk_i$ is adjacent to $x_{n+1}$ is bounded by $O_n(\vk_i\vd_{n+1}/\sum_{h=1}^{n+1}\vd_h)$ and $\sum_{h=1}^{n+1}\vd_h\geq n$.
Thus, we obtain (\ref{eqQuickBound10}).
Further, we observe that (\ref{eqQuickBound10}) yields
$\pr\brk{\cQ_1}\leq\Erw\sum_{i\geq1}i\GAMMA_i\vecone\cQ_1=O_n(\log^{-2}n).$
Hence,  as $\Erw\sum_{i\geq1}i\DELTA_i=O_n(1)$ we obtain
\begin{align}\label{eqQuickBound11}
\Erw\brk{\sum_{i\geq1}i\DELTA_i\vecone\cQ_1}&\leq\pr\brk{\cQ_1}\log n
		+\Erw\brk{\sum_{i\geq1}i\DELTA_i\vecone\cbc{\sum_{i\geq1}i\DELTA_i\geq\log n}}=o_n(1).
\end{align}
Combining \eqref{eqQuickBound10} and \eqref{eqQuickBound11}, we conclude that
\begin{align}\label{eqQuickBound12}
\Erw\brk{\sum_{i\geq1}i(\DELTA_i+\GAMMA_i)\vecone\cQ_1}=o_n(1).
\end{align}

Regarding $\cQ_2$, we deduce from the bound $\Erw[\vd_{n+1}^r]=O_n(1)$ for an $r>2$ that
\begin{align}\label{eqQuickBound98}
\Erw\brk{\sum_{i\geq1}i\GAMMA_i\vecone\cQ_2}&\leq O_n(\log n)\Erw\brk{\vd_{n+1}\vecone\{\vd_{n+1}>\log n\}}=o_n(1).
\end{align}
Moreover, since the $\DELTA_i$ are independent of $\vd_{n+1}$ and $\Erw\sum_{i\geq1}i\DELTA_i=O_n(1)$, we obtain $\Erw\brk{\sum_{i\geq1}i\DELTA_i\vecone\cQ_2}$
\newline
$=o_n(1)$.
Hence, \eqref{eqQuickBound98} yields
\begin{align}\label{eqQuickBound21}
\Erw\brk{\sum_{i\geq1}i(\DELTA_i+\GAMMA_i)\vecone\cQ_2}=o_n(1).
\end{align}

Moving on to $\cQ_3$ and recalling the definition \eqref{eqm} of $\vec\Delta$, we find
\begin{align}\label{eqQuickBound33}
\pr\brk{\cQ_3}&\leq\Erw\brk{\sum_{i\geq1}i\DELTA_i}\ln^{-3}n=O_n(\Erw[\vk^2]\ln^{-3}n)=o_n(\log^{-2}n).
\end{align}
Moreover, on $\cQ_3$ we have $\sum_{i\geq1}i\GAMMA_i\leq\log^2n$ because $\vd_{n+1}\leq\log n$ and $\GAMMA_i=0$ for all $i\geq\log n$.
Consequently, since the $\DELTA_i$ are mutually independent and $\sum_{i\geq1}\Erw[i\DELTA_i]=O_n(1)$, \eqref{eqQuickBound33} yields
\begin{align}\label{eqQuickBound30}
\Erw\brk{\sum_{i\geq1}i(\DELTA_i+\GAMMA_i)\vecone\cQ_3}
	&\leq o_n(1)+4\Erw\brk{\sum_{i\geq1}i\DELTA_i\vecone\cbc{\sum_{i\geq1}i\DELTA_i\geq\ln^3n}}
	=o_n(1).
\end{align}
Finally, the assertion follows from \eqref{eqQuickBound12}, \eqref{eqQuickBound21} and \eqref{eqQuickBound30}.
\end{proof}

\begin{proof}[Proof of \Lem~\ref{Lemma_valid}]
The first assertion concerning $\vA''$ and $\vA_{n,\vM}$ follows from Claim~\ref{Claim_Lemma_valid_1}.
\newline
Concerning $\vA'''$ and $\vA_{n+1,\vM^+}$, Claim~\ref{Claim_Lemma_valid_2} shows that it suffices to couple $\nul(\vA''')\vecone\cE^+$ and $\nul(\vA_{n+1,\vM^+})\vecone\cE^+$, because both random variables are bounded by $n+1$.
Indeed, thanks to Claim~\ref{Claim_Lemma_valid_5} we merely need to couple 
$\nul(\vA''')\vecone\cE^+\setminus(\cU\cap\cW)$ and $\nul(\vA_{n+1,\vM^+})\vecone\cE^+\setminus(\cU\cap\cW)$, and Claim~\eqref{Claim_Lemma_valid_6} supplies a coupling such that
\begin{align}\label{eqLemma_valid_1}
\abs{\nul(\vA''')\vecone\cE^+-\nul(\vA_{n+1,\vM^+})}\vecone\cE^+\setminus(\cU\cap\cW)
	&\leq2\sum_{i\geq1}i(\DELTA_i+\GAMMA_i)\vecone\cE^+\setminus(\cU\cap\cW).
\end{align}
Hence, it suffices to show that
\begin{align}\label{eqLemma_valid_2}
\Erw\brk{\sum_{i\geq1}i(\DELTA_i+\GAMMA_i)\vecone\cE^+\setminus(\cU\cap\cW)}=o_n(1).
\end{align}
Indeed, in light of Claim~\ref{Claim_Q123} we merely need to estimate
$\sum_{i\geq1}i(\DELTA_i+\GAMMA_i)\vecone\cE^+\cap\cW\setminus \cU$.
But since on $\cE^+\cap\cW$ we have $\sum_{i\geq1}i(\DELTA_i+\GAMMA_i)\leq\ln^4n$,  Claim~\ref{Claim_Lemma_valid_4} yields
\begin{align}\label{eqLemma_valid_3}
\Erw\brk{\sum_{i\geq1}i(\DELTA_i+\GAMMA_i)\vecone\cE^+\cap\cW\setminus \cU}
	\leq\bc{1-\pr\brk{\cU\mid\cE^+\cap\cW}}\ln^4n=o_n(1).
\end{align}
Finally, the assertion follows from Claim~\ref{Claim_Q123} and \eqref{eqLemma_valid_1}--\eqref{eqLemma_valid_3}.
\end{proof}

\section{Proof of \Thm~\ref{thm:2core}}\label{Sec_thm:2core}\label{sec:2core}

 \noindent
We describe how to extend the proof of~\cite{molloy2005cores} to $\bfG$. \fixxx{First, we will work on $G=\bfG_{0,n}$ (defined in Section~\ref{Sec_outline1}), the configuration model for $\bfG$. By Lemma~\ref{Lemma_simple}, properties that holds with probability $1-o_n(1)$ for $G$ also hold with probability $1-o_n(1)$ for $\bfG$. Second,} using the terminology in~\cite{molloy2005cores}, variable nodes in $G$ are called vertices, and each check node corresponds to a hyperedge in the following sense: if $f_a$ is a check node adjacent with variable nodes $\{v_{a_1},\ldots, v_{a_h}\}$ for some $h\ge 1$, then the set of vertices $\{v_{a_1},\ldots, v_{a_h}\}$ is called a hyperedge. A check node with size 0 corresponds to an isolated hyperedge with size 0, i.e.\ this hyperedge does not contain any vertex. 

\fixx{Suppose $\pr(\bfd\ge 2)>0$ as otherwise the theorem holds trivially by Remark~\ref{remark:specialCases}(b).}
\fixxx{We first prove Theorem~\ref{thm:2core} in the case $\bfk\ge 1$.} Consider the parallel stripping process where all vertices of degree less than 2 are deleted in each step, together with the hyperedges (if any) incident with them. Take a random vertex $v\in[n]$.  Let $\lambda_t$ be the probability that $v$ survives after $t$ iterations of the stripping process. It is easy to see that $\lambda_t$ is monotonically non-increasing and thus $\lambda=\lim_{t\to\infty} \lambda_t$ exists. For any vertex $u\in [n]$, let $\partial^j(u)$ denote the set of vertices of distance $j$ from $u$. Recall that there exists a constant $\sigma>0$ such that $\ex\bfd^{2+\sigma}<\infty$ and $\ex\bfk^{2+\sigma}<\infty$ by our assumptions on $\bfd$ and $\bfk$. We claim that
\begin{claim} \label{c:neighbourhood}
	With high probability, the maximum degree and the maximum size of hyperedges in $G$ \fixxx{are} at most \newline
	$(n\log n)^{1/(2+\sigma)}$, and for every $u\in[n]$ and for all fixed $R$, $|\cup_{j\le R}\partial^j(u)|=O_n(n^{1/(2+\sigma)}\log^2 n)$. 
\end{claim}
Let $H_t$ be the subgraph of $G$ obtained after $t$ iterations of the parallel stripping process. Consider Doob's martingale $(\ex(H_t\mid e_1,\ldots, e_j))_{0\le j\le m}$ where random hyperedges are added in the order $e_1,\ldots, e_{m}$ using the configuration model, \fixxx{and $m$ denotes the number of hyperedges in $G$}. By Claim~\ref{c:neighbourhood}, swapping two clones in the configuration model would affect $H_t$ by $O_n(n^{1/(2+\sigma)}\log^2 n)$, as each altered hyperedge can only affect the vertices (if surviving the first $t$-th iteration or not) within its $t$-neighbourhood.  Standard concentration arguments \fixxx{(see, for instance, the proof of~\cite[Theorem 2.19]{Regular})} based on Azuma's inequality (with Lipschitz constant $Cn^{1/(2+\sigma)}\log^2 n$ for some fixed $C>0$) produce that $||H_t|-\lambda_t n|=O_n(n^{(4+\sigma)/(4+2\sigma)}\log^3 n)=o_n(n)$. Next we deduce an expression for $\lambda_t$.  Consider a random hypertree $T$ iteratively built as follows. The root of $T$ is $v$, which is incident to $d_v$ hyperedges of size $\bfk_1,\ldots, \bfk_{d_v}$ where the $\bfk_i$s are i.i.d.\ copies of $\hat\bfk$ where
\begin{equation}\label{hatk}
	\pr(\hat\bfk=j) = \frac{j \pr(\bfk=j)}{k}.
\end{equation}
Then the $i$-th hyperedge is incident to other $\bfk_i-1$ vertices (other than $v$) whose degrees are i.i.d.\ copies of $\hat\bfd$, where 
\begin{equation}\label{hatd}
	\pr(\hat\bfd=j) = \frac{j \pr(\bfd=j)}{d}.
\end{equation}
This builds the first neighbourhood of $v$ in $T$. Iteratively we can build the $r$-neighbourhood of $v$ in $T$ for any fixed $r$. It follows from the following claim that the $r$-neighbourhood of $v$ in $G$ converges in distribution to the $r$-neighbourhood of $T$, as $n\to\infty$, for any fixed $r\ge 1$. This is because when uniformly picking a random variable clone (or check clone), the degree of the corresponding variable node (or check node) has the distribution in~\eqn{hatd} (or~\eqn{hatk}). Let $S$ be a set of vertices in $G$. We say $S$ induces a cycle if there is a closed walk $x_0x_1\ldots x_{\ell}=x_0$ such that all $x_i\in S$, and every pair of consecutive vertices in the walk are contained in a hyperedge in $G$.\smallskip

\begin{claim}~\label{c:cycle}
	With high probability, for all fixed $R\ge 1$, $\cup_{j\le R}\partial^j(v)$ induces no cycles.
\end{claim}

(The proofs of Claim $6.1$ and $6.2$ can be found at the end of this section.)
If $v$ survives $t$ iterations of the stripping process then at least two hyperedges incident with $v$ survives after $t$ iterations of the stripping process. 
\fixx{On the other hand, let $x$ be a hyperedge of size at least 1 and let $u$ be a vertex incident with $x$.} Let $\rho_t$ denote the probability that $u$ is incident with at least one hyperedge other than $x$ which survives after $t$ iterations of the stripping process. \fixxx{We will deduce a recursion for $\rho_t$ and then deduce $\lambda_t$ from $\rho_t$.} Note that the degree of $u$ follows the distribution from~\eqn{hatd}. Then, ignoring an $o_n(1)$ error accounting for the probability of the complement of the events in Claims~\ref{c:neighbourhood} and~\ref{c:cycle}:
\[
\rho_0=1
\]
and
\begin{align*}
	\rho_{t+1}&=\sum_{j\ge 2} \frac{j\pr(\bfd=j)}{d} \sum_{S\subseteq[j-1], |S|\ge 1} \sum_{k_1,\ldots,k_{j-1}\fixxx{\ge 1}} \prod_{i=1}^{j-1} \pr(\fixxx{\hat\bfk_i=k_i}) \prod_{i\in S} \rho_t^{k_i-1} \prod_{i\in [j-1] \setminus S} (1-\rho_t)^{k_i-1}\\  
	&=\sum_{j\ge 2} \frac{j\pr(\bfd=j)}{d} \sum_{h\ge 1} \binom{j-1}{h} \left(\sum_{\fixxx{k'\ge 1}}\pr(\fixxx{\hat\bfk=k'}) \rho_t^{k'-1} \right)^h \left(\sum_{\fixxx{k'\ge 1}}\pr(\fixxx{\hat\bfk=k'}) (1-\rho_t)^{k'-1}\right)^{j-1-h}\\
	&=\sum_{j\ge 2} \frac{j\pr(\bfd=j)}{d} \sum_{h\ge 1}\binom{j-1}{h} \left(\frac{K'(\rho_t)}{k}\right)^h\left(1-\frac{K'(\rho_t)}{k}\right)^{j-1-h}\\
	&=\sum_{j\ge 2} \frac{j\pr(\bfd=j)}{d}\left(1-\left(1-\frac{K'(\rho_t)}{k}\right)^{j-1}\right)=1-\frac{D'(1-\frac{K'(\rho_t)}{k})}{d},
\end{align*}
noting that 
\[
\ex \rho^{\hat\bfk-1}\fixxx{=\sum_{k'\ge 1} \pr(\hat\bfk=k')\rho^{k'-1}}=\sum_{j\ge 1} \frac{j\pr(\bfk=j)}{k}\rho^{j-1}=\frac{K'(\rho)}{k}.
\]
Consequently,
\begin{align*}
	\lambda_t &= \sum_{j\ge 2} \pr(\bfd=j) \sum_{h\ge 2}\binom{j}{h} \left(\sum_{\fixxx{k'\ge 1}}\pr(\fixxx{\hat\bfk=k'}) \rho_t^{k'-1} \right)^h \left(\fixxx{1-\sum_{k'\ge 1}\pr(\hat\bfk=k') \rho_t^{k'-1}}\right)^{j-h}\\
	&= \sum_{j\ge 2} \pr(\bfd=j) \sum_{h\ge 2}\binom{j}{h} \left(\ex \rho_t^{\hat\bfk-1}\right)^h \left(1-\ex \rho_t^{\hat\bfk-1}\right)^{j-h}\\
	&= \sum_{j\ge 2} \pr(\fixxx{\bfd}=j) \left(1-\left(\fixxx{1-}\frac{K'(\rho_t)}{k}\right)^j-j\frac{K'(\rho_t)}{k}\left(1-\frac{K'(\rho_t)}{k}\right)^{j-1}\right)\\
	&=\fixxx{\pr(\bfd\ge 2)- \left(D\left(1-\frac{K'(\rho_t)}{k}\right) - \pr(\bfd=0)-\pr(\bfd=1)\left(\fixxx{1-}\frac{K'(\rho_t)}{k}\right)\right) -\left(\frac{K'(\rho_t)}{k}D'\left(\fixxx{1-}\frac{K'(\rho_t)}{k}\right)- \pr(\bfd=1) \frac{K'(\rho_t)}{k} \right)}\\
	&=1-D\left(1-\frac{K'(\rho_t)}{k}\right)-\frac{K'(\rho_t)}{k}D'\left(\fixxx{1-}\frac{K'(\rho_t)}{k}\right).
\end{align*}
Let $g(x)=1-\frac{1}{d}D'(1-\frac{K'(x)}{k})$. Then $g'(x)=\frac{1}{dk}D''(1-\frac{K'(x)}{k})K''(x)$ which is non-negative over $[0,1]$. We also have $\phi(x)=g(x)-x$, where $\phi$ is given in~\eqn{def:ff}.
Since $\phi(1)=-D'(0)/d\le 0$,  $\phi'(\rho)<0$ by the hypothesis, and $g(x)$ is nondecreasing in $[0,1]$, it follows that \fixxx{$|g'(\rho)|<1$ and thus} $\rho$ is an attractive fix point of $x=g(x)$. 
As $\rho_0=1$. It follows that $\rho_t\to\rho$ as $t\to \infty$. Consequently, for every $\hat\eps>0$ there is sufficiently large $I$ such that $|\rho_t-\rho|<\hat\eps$. Hence, after $I$ iterations of the parallel stripping process, the number of vertices remaining is $(\lambda+o(1)) n+O_n(\hat\eps n)$ where 
\begin{equation}
	\lambda=1-D\left(1-\frac{K'(\rho)}{k}\right)-\frac{K'(\rho)}{k}D'\left(\fixxx{1-}\frac{K'(\rho)}{k}\right).\label{lambda2}
\end{equation}
If \fixx{$\rho=0$ then $\lambda=0$ by Remark~\ref{remark:specialCases}(c).} Our theorem for $\bfn^*$ follows by letting $I\to\infty$. \fixxx{Since $\bfk\ge 1$, $K(0)=0$ and thus $\bfm^*/n=\frac{d}{k}K(0)+o_{\hat\eps,n}(1) = o_{\hat\eps,n}(1) $. This establishes~\eqn{eqthm:2core} when $\rho=0$.}


Suppose $\rho>0$. It is sufficient to show that the 2-core is obtained after further removing  $O_n(\hat\eps n)$ vertices, following the same approach as~\cite[Lemma 4]{molloy2005cores}. We briefly sketch it. Following the same argument as before, the probability that a random vertex has degree $j\ge 2$ after $I$ iterations of the stripping process is
\[
\sum_{i\ge \fixxx{j}} \pr(\bfd=i) \binom{i}{j} (\ex \rho_I^{\hat\bfk-1})^j (1-\ex \rho_I^{\hat\bfk-1})^{i-j}= \sum_{i\ge \fixxx{j}} \pr(\bfd=i) \binom{i}{j}\left(\frac{K'(\rho_I)}{k}\right)^j\left(1-\frac{K'(\rho_I)}{k}\right)^{i-j}.
\] 
Similarly, the probability of a uniformly random hyperedge in $G$ having size $j\ge 1$ and surviving the first $I$ iterations of the stripping process is
\[
\pr(\bfk=j) \rho_I^j.
\]
The number of vertices with degree less than 2 after $I$ iterations is bounded by $(\lambda_I-\lambda_{I+1}) n +o_n(n)$. Hence, by choosing $I$ sufficiently large, we can make these quantities arbitrarily close to those with $\rho_I$ replaced by $\rho$. Now standard concentration arguments apply to show that the number of degree $j\ge 2$ vertices is 
$\gamma_j n+O_n(\hat\eps n)$, where
\[
\gamma_j=\sum_{i=0}^{\infty} \pr(\bfd=i) \binom{i}{j}\left(\frac{K'( \rho)}{k}\right)^j\left(1-\frac{K'( \rho)}{k}\right)^{i-j},
\] 
the number of vertices of degree less than 2 is $O_n(\hat\eps n)$. \fixxx{Similarly, the number of remaining hyperedges of size $j$ is 
	$
	\pr(\bfk=j) \rho^j m +O_n(\hat\eps n)
	$,
	and the total degree of the remaining vertices is
	\begin{equation}
		m\sum_{j\ge 1} j \pr(\bfk=j)  \rho^j +O_n(\hat\eps n)=m\rho K'(\rho)+O_n(\hat\eps n)=(dn/k)\rho K'(\rho)+O_n(\hat\eps n). \label{Dt}
	\end{equation}
}
Note that $\hat\eps$ can be made arbitrarily small by choosing sufficiently large $I$.

Now we remove one hyperedge incident with a vertex with degree 1 at a time. Call this process SLOWSTRIP. Let $G_t$ denote the hypergraph obtained after $t$ steps of SLOWSTRIP and let $X_t$ denote the total degree of the vertices of degree 1 in $G_t$,  Then, \fixxx{for all $t=O_n(\hat\eps n)$ such that $X_t>0$:}
\begin{align*}
	&\ex(X_{t+1}-X_t\mid G_t)\\
	&= -1 + \sum_{j\ge 1} \frac{j\pr(\bfk=j) \rho^j m}{\rho K'(\rho) m} \cdot (j-1) \cdot \frac{2\gamma_2 \fixxx{n}}{\fixxx{(dn/k)\rho K'(\rho)}} +O_n(\hat\eps)\\
	&=-1+\frac{1}{\rho K'(\rho)}\left( \sum_{j\ge 1} j(j-1) \pr(\bfk=j) \rho^j\right) \frac{2\cdot \frac12 (K'(\rho)/k)^2 D''(1-K'(\rho)/k)}{K'(\rho)d\rho/k}+O_n(\hat\eps)\\
	&=-1+\frac{D''(1-K'(\rho)/k)K''(\rho)}{kd}+O_n(\hat\eps).
\end{align*}
Note that in the first equation above, $-1$ accounts for the removal of one variable clone $x$ from the set of vertices of degree less than 2. The term $j\pr(\bfk=j) \rho^j m/\rho K'(\rho) m$ approximates the probability that $x$ is contained in a hyperedge of size $j$, up to an $O_n(\hat\eps)$ error. In that case, $j-1$ variable clones that lie in the same hyperedge as $x$ will be removed. For each of these $j-1$ deleted variable clones, if it lies in a variable of degree 2, then it results in one new variable node of degree 1. The probability for that to happen is approximated by \fixxx{$2\gamma_2 n/D_t$, up to an $O_n(\hat\eps)$ error, where $D_t$ denotes the total degree of $G_t$ and by~\eqn{Dt}, $D_t=(dn/k)\rho K'(\rho)+O_n(\hat\eps n)$.
	For the second equation above, note that
	\begin{align*}
		\gamma_2=\sum_{i\ge 2} \pr(\bfd=i) \frac{i(i-1)}{2} \left(\frac{K'(\rho_I)}{k}\right)^2\left(1-\frac{K'(\rho_I)}{k}\right)^{i-2}+O_n(\hat\eps)=\frac12 \left(\frac{K'(\rho)}{k}\right)^2 D''(1-K'(\rho)/k)+O_n(\hat\eps).
	\end{align*}
}
By the assumption that $\fixxx{\phi}'(\rho)<0$ we have 
\[
-1+\frac{D''(1-K'(\rho)/k)K''(\rho)}{kd}<0.
\]
Hence, $ \ex(X_{t+1}-X_t\mid G_t)<-\delta$ for some $\delta>0$, by making $\hat\eps$ sufficiently small (i.e.\ by choosing sufficiently large $I$). Then the standard Azuma inequality~\cite[Lemma 29]{GaoMolloy} (with Lipschitz constant $(n\log n)^{1/(2+\sigma)}$ by Claim~(\ref{c:neighbourhood}) will be sufficient to show that $X_t$ decreases to 0 after $O_n(\hat\eps n)=o_{\hat\eps,n}(n)$ steps (See details in~\cite[Lemma 4]{molloy2005cores}). The case $\rho>0$ of the theorem follows by 
\[
\lim_{n\to\infty} \frac{\bfm^*}{n} =  \lim_{n\to\infty} \frac{m}{n} \cdot \sum_{j\ge \fixxx{1}} \pr(\bfk=j) \rho^j= \frac{d}{k}K(\rho),
\]  
as desired. \fixxx{This proves~\eqn{eqthm:2core} when $\bfk\ge 1$. 
	
	Suppose now that $p_0=\pr(\bfk=0)>0$.  Let $\bar G$ be the hypergraph obtained from $G$ by deleting all hyperedges with size 0. Let $\bar m$ denote the number of hyperedges in $\bar G$. Then, with probability $1-o_n(1)$, $\bar m\sim (1-p_0)m$. The size of a uniformly random hyperedge in $\bar G$ has the same distribution as $\bar\bfk$, defined by $\bfk$ conditioned on $\bfk\ge 1$. Let $\bar k=\ex \bar\bfk$. Then, $\bar k=k/(1-p_0)$. Let $\bar K(\alpha)$ be the probability generating function of $\bar\bfk$. Then, immediately
	\[
	\bar K(\alpha)=\frac{K(\alpha)-p_0}{1-p_0},\quad \bar K'(\alpha)=\frac{K'(\alpha)}{1-p_0},\quad \bar K''(\alpha)=\frac{K''(\alpha)}{1-p_0}.
	\] 
	Let $\bar\Phi$ be the function obtained from $\Phi$ by replacing $K(\alpha)$, $K'(\alpha)$ and $k$ by $\bar K(\alpha)$, $\bar K'(\alpha)$ and $\bar k$ respectively. It is straightforward to see that the set of stable points of $\Phi$ corresponds to the set of stable points of $\bar\Phi$. Thus, by letting $\bar\rho=\max\{x\in[0,1]: \bar\Phi'(x)=1\}$ it follows then that $\bar\rho=\rho$.
	Applying~\eqn{eqthm:2core} with $\bar \bfk\ge 1$ to $\bar G$,
	\begin{eqnarray*}
		\lim_{n\to\infty} \frac{\bar\bfn^*}{n} &=& 1-D\left(1-\frac{\bar K'(\bar\rho)}{\bar k}\right)-\frac{\bar K'(\bar\rho)}{\bar k}D'\left(1-\frac{\bar K'(\bar\rho)}{\bar k}\right) = 1-D\left(1-\frac{ K'(\rho)}{ k}\right)-\frac{ K'(\rho)}{ k}D'\left(1-\frac{ K'(\rho)}{ k}\right)\\
		\lim_{n\to\infty} \frac{\bar\bfm^*}{n} &=& \frac{d}{\bar k} \bar K(\bar\rho)=\frac{d}{k}(K(\rho)-p_0),
	\end{eqnarray*}
	where $\bar\bfn^*$ and $\bar\bfm^*$ denote the numbers of vertices and hyperedges in $\bar G$. Since $\bfn^*=\bar\bfn^*$ and $\bfm^*=\bar\bfm^*+(1+o_n(1))p_0 m=(1+o_n(1))(\bar\bfm^*+p_0 dn/k)$, as the set of hyperedges of size 0 in $G$ remain in the 2-core of $G$, the equations~\eqn{eqthm:2core} holds as well for the case that $p_0>0$.\qed \medskip
}

\noindent {\em Proof of Claim~\ref{c:neighbourhood}. } Since both $\ex\bfd^{2+\sigma}=O_n(1)$ and $\ex\bfk^{2+\sigma}=O_n(1)$, the probability that $\bfd>(n\log n)^{1/(2+\sigma)}$ or $\bfk>(n\log n)^{1/(2+\sigma)}$ is $O_n(1/n\log n)$. The bound on the maximum degree and maximum size of the hyperedges in $\bfG$ follows by taking the union bound. 

For any $u\in[n]$, let $N_i(u)=|\partial^i(u)|$. We will prove that with high probability for every $u$ and for every fixed $i$, $N_i(u)=O_n(n^{1/(2+\sigma)}\log^2 n)$, which then completes the proof for Claim~\ref{c:neighbourhood}. We prove by induction. Let $d_1,\ldots, d_{N_i(u)}$ denote the degrees of the vertices in $\partial^i(u)$.  Then the number of hyperedges incident with these vertices is bounded by $M:=\sum_{j=1}^{N_i(u)} d_j$. By the construction of $\bfG$, each $M$ is stochastically dominated by $\sum_{j=1}^{N_i(u)} (1+o_n(1)) \hat\bfd_j$ where $\hat\bfd_j$ are i.i.d.\ copies of $\hat\bfd$ whose distribution is given in~\eqn{hatd}. The $o_n(1)$ error is caused by the exposure of $\cup_{j\le i}\partial^j(u)$ which contains $o_n(n)$ vertices by induction. Since $\ex \bfd^{2+\sigma}=O_n(1)$, we have $\hat d :=\ex \hat\bfd =O_n(1)$. Note that $\ex M=\hat d N_i(u)$. Applying the Chernoff bound to the sum of independent $[0,1]$-valued random variables we have
\[
\pr\left(M\ge 2\hat d N_i(u)+ n^{1/(2+\sigma)}\log^2 n \right) = \pr\left(\sum_{j=1}^{N_i(u)} \frac{\hat \bfd_i}{(n\log n)^{1/(2+\sigma)} } \ge \frac{2\hat d N_i(u)}{(n\log n)^{1/(2+\sigma)}} + (\log n)^{(3+\sigma)/(2+\sigma)}\right) <n^{-2}.
\]
Similarly, $N_{i+1}(u)$ is bounded by $\sum_{j=1}^M k_i$, where $k_i$ are the sizes of the hyperedges incident with the vertices in $\partial^i(u)$. Similarly, $\sum_{j=1}^M k_i$ is stochastically dominated by $(1+o_n(1))\sum_{j=1}^M \hat\bfk_j $ where $\hat\bfk_j$ are i.i.d.\ copies of $\hat\bfk$ whose distribution is defined in~\eqn{hatk}. Let $\hat k=\ex\hat\bfk$. Applying the Chernoff bound again we obtain that with probability at least $1-n^{-2}$, $N_{i+1}(u)< 2\hat k M +n^{1/(2+\sigma)}\log^2 n < 4\hat d\hat k N_i(u) +(1+2\hat k)n^{1/(2+\sigma)}\log^2 n$. Apply this recursion inductively and the union bound on the failure probability, we obtain $N_i(u)=O_n(n^{1/(2+\sigma)}\log^2 n)$, as desired. \qed\smallskip

\noindent {\em Proof of Claim~\ref{c:cycle}. } Fix $\eps>0$. Choose $L=L(\eps,r)$  sufficiently large so that the probability that $d_v>L$ is smaller than $\eps$ (note that $v$ is a uniformly random vertex). Given $d_v\le L$. Let $k_1,\ldots, k_{d_v}$ be the sizes of the hyperedges incident to $v$. Similarly to the proof of Claim~\ref{c:neighbourhood},  $k_j$s are approximated by i.i.d.\ copies of $\hat\bfk$ defined in~\eqn{hatk}, up to an $1+o(1)$ multiplicative error. We can assume $L$ is sufficiently large so that with probability at least $1-\eps$, $\sum_{i=1}^{d_v} k_i\le L$. Inductively, we can make $L$ sufficiently large so that $|\partial^i(v)|\le L$ for all $i\le R$.    Let $\E_i$ denote the set of hyperedges incident with vertices in $\partial^i(v)$, but not incident with any in $\partial^{i-1}(v)$. Cycles in $\partial^i(v)$ can appear in two ways: (a) two vertices in $\partial^i(v)$ are incident with the same hyperedge in $\E_i$; (b) two hyperedges in $\E_{i-1}$ are incident with the same vertex in $\partial^i(v)$. We will prove that with high probability, none of the two cases occurs for any fixed $i$. For (a), let $(d_j)_{j\in \partial^i(v)}$ denote the degrees of the vertices in $\partial^i(v)$. The expected number of occurrences of pairs of vertices in (a) is 
\begin{equation}\label{pairs}
	\ex\left(\sum_{j, k\in \partial^i(v)} \binom{d_j}{2} \binom{d_k}{2} \sum_{h\in [\bfm]} \binom{k_h}{2} O_n(n^{-2}) \right)= O_n(n^{-1}) \ex\left(\sum_{j,k\in \partial^i(v)} d_j^2 d_k^2\right).
\end{equation}
Note that $|\partial^j(v)|\le L$ for each $j\le R$. This immediately implies that $d_j\le L$ for all $j\in \partial^i(v)$. Hence, the above probability is $O_n(n^{-1})$.  The probability that $|\partial^i(v)|\le L$ fails is at most $R\eps$ by our choice of $L$. Hence, the probability that (a) fails is at most $R\eps+o_n(1)$.
The treatment of (b) is analogous. Our claim now follows by letting $\eps\to 0$. \qed

\section{Proof of \Thm~\ref{Thm_tight}}\label{Sec_tight} 

\noindent
Recall that
\begin{equation}\label{def:f}
	\phi(\alpha)=1-\alpha-\frac{1}{d}D'\left(1-\frac{K'(\alpha)}{k}\right).
\end{equation}

\fixx{For \Thm~\ref{Thm_tight} and Remark~\ref{remark:condition},} it is sufficient to prove that if condition (i) \fixx{or} (ii) is satisfied then 
(a)
$\max_{\alpha\in[0,1]}\Phi(\alpha)=\max\{\Phi(0),\Phi(\rho)\}$;
and (b) $\phi'(\rho)<0$ unless
\begin{equation}
	\pr(\bfd=1)=0\quad \mbox{and}\quad 2(\ex\bfk-1)\pr(\bfd=2)>\ex \bfd. \label{exception}
\end{equation} 

Since $\Phi(\alpha)$ is continuous on $[0,1]$, the maximum occurs at either 0 or 1 or at a stable point.

\smallskip

\noindent {\em Case A: $\Var(\bfk)=0$.} In this case, $\bfk=k$ always and thus $K(\alpha)=\alpha^k$. 
We must have $k\ge 1$ since otherwise $k=d=0$.  If $k=1$ then $\phi'(\alpha)=-1$ which implies (b) immediately. \fixx{Moreover, $K''(\alpha)=0$ for all $\alpha\in [0,1]$ and thus $\Phi'(\alpha)=(d/k)K''(\alpha)\phi(\alpha)=0$ for all $\alpha\in [0,1]$. This implies (a).}

Next consider the case that $k=2$. Then, $\phi''(\alpha)=-\frac{1}{d}D'''(1-\alpha)< 0$ on $(0,1)$ unless $\bfd\le 2$. Consider the case that $\supp{\bfd}\cap \bbN_{\ge 3}\neq \emptyset$. Then $\phi$ is concave and can have at most 2 roots. Obviously $\alpha=0$ is a root. Let $\rho$ denote the other root if exists. We must have $\phi'(\rho)<0$ by the concavity of $\phi$. Hence, the maximum of $\Phi$ cannot be achieved at 1. Thus, the maximas of $\Phi$ can only be from $\{0,\rho\}$. This verifies (a) and (b). Now assume $\bfd\le 2$. Then $\phi''(\alpha)=0$ on $[0,1]$. Hence $\phi'(\alpha)=\phi'(1)=-1+\pr(\bfd=2)/d<0$ for all $\alpha\in [0,1]$. Thus, $\phi(\alpha)$ is a line with a negative slope and  has exactly one root at $0$ on $[0,1]$. Hence $\rho=0$ and $\phi'(\rho)<0$. This verifies (a) and (b). 

Next we consider the case that $k\ge 3$. We have
\begin{align*}
	\phi(\alpha)&=1-\alpha-\frac{1}{d}D'(1-\alpha^{k-1})\\
	\phi'(\alpha)&=-1+\frac{(k-1)\alpha^{k-2}}{d}D''(1-\alpha^{k-1})\\
	\phi''(\alpha)&=\frac{k-1}{d}\alpha^{k-3}\Big((k-2)D''(1-\alpha^{k-1})-(k-1)D'''(1-\alpha^{k-1})\alpha^{k-1}\Big)\\
	&=\frac{k-1}{d}\alpha^{k-3}\Big((k-2)D''(t)-(k-1)D'''(t)(1-t)\Big) \quad \mbox{where $t=1-\alpha^{k-1}$}.
\end{align*}
Hence,
\begin{align}
	&& \phi(0)=0 && \phi(1)=-\frac{1}{d}D'(0)\le 0 && \label{f}\\
	&& \phi'(0)=-1 && \phi'(1)=-1+\frac{k-1}{d}D''(0).&& \label{f'}
\end{align}
Recall that $\Phi'(\alpha)=\frac{d}{k}K''(\alpha) \phi(\alpha)$. We have $K''(\alpha)> 0$ for all $\alpha\in(0,1]$.  By~\eqn{f} we have $\Phi'(1)\le 0$ and thus the supremum of $\Phi(\alpha)$ can only occur at 0 or a stable point. In all of the following sub-cases, we will prove that $\phi''(\alpha)$ has at most 1 root in $[0,1]$ (except for some trivial cases that we discuss separately). It follows immediately that $\phi$ can have at most three roots on $[0,1]$ including the trivial one at $\alpha=0$. Now we prove that this implies claims (a) and (b).

If $\phi$ has only a trivial root, then so is $\Phi'(\alpha)$. Thus, $\alpha=0$ is the unique maxima of $\Phi(\alpha)$ and $\rho=0$. This verifies (a). As $\phi'(0)=-1$ we immediately have $\phi'(\rho)<0$.

If $\phi$ has two roots, then the larger root is $\rho$. Since $\phi'(0)<0$, in this case, $\phi$ is negative in $(0,\rho)$ and positive in $(\rho,1)$. This is only possible when $\phi(1)=0$ \fixx{by~\eqn{f}}, which requires $\pr(\bfd=1)=0$. In this case, $\rho=1$. 
Next we consider two further cases: (i) $2(k-1)\pr(\bfd=2)>d$ corresponding to $\phi'(1)>0$; (ii) $2(k-1)\pr(\bfd=2)<d$ corresponding to $\phi'(1)<0$. As $\phi$ has only two roots, case (ii) obviously cannot happen. Thus, it means that the only situation that $\phi$ has two roots would be $\pr(\bfd=1)=0$ and $2(k-1)\pr(\bfd=2)>d$, as in~\eqn{exception}. In this situation we are only required to verify (a). Note that $\phi$ is negative in $(0,1)$ as $\rho=1$. It follows then that $\Phi(\alpha)$ is a decreasing function in $(0,1)$. Hence, $\alpha=0$ is the unique maxima, as desired.

If $\phi$ has three roots, then there is a root $\rho^*$ between $0$ and $\rho$. Then $\phi$ is negative in $(0,\rho^*)$ and positive in $(\rho^*,\rho)$. As $K''(\alpha)>0$ for all $\alpha\in (0,1]$, the sign of $\phi$  implies that $\rho^*$ is a local minima and $\rho$ is a local maxima. This verifies (a).  Moreover, as $\phi$ is positive in $(\rho^*,\rho)$ and $\phi(\rho)=0$, $\phi'(\rho)<0$ follows immediately. 

\medskip

\noindent {\em Case A1: $\Var(\bfk)=0$ and $\Var(\bfd)=0$.} In this case $\bfd=d$. Then $D(\alpha)=\alpha^d$. If $d\ge 3$ then
\begin{align*}
	\phi''(\alpha)&=\frac{k-1}{d}\alpha^{k-3}\Big((k-2)d(d-1)t^{d-2}-(k-1)d(d-1)(d-2)t^{d-3}(1-t)\Big)\\
	&=(k-1)(d-1)t^{d-3} \alpha^{k-3}\Big((k-2)t-(k-1)(d-2)(1-t)\Big) \quad \mbox{where $t=1-\alpha^{k-1}$}.
\end{align*}
Obviously, $\phi''(\alpha)$ has a unique root in $[0,1]$. 

If $d=1$ then $\phi'(\alpha)=-1$ and so $\phi$ has only a trivial root at $\alpha=0$; 
If $d=2$ then $\phi''(\alpha)>0$ in $(0,1)$ and so $\phi$ is convex and thus has only a trivial root at $\alpha=0$ by~\eqn{f}. Hence for $d\le 2$, $\rho=0$ and is the unique maxima. Claims (a) and (b) hold trivially.
\smallskip

\noindent {\em Case A2: $\Var(\bfk)=0$ and $\bfd\sim \po_{\ge r}(\lambda)$.} In this case $D(\alpha)=h_r(\lambda\alpha)/h_r(\lambda)$, where
\begin{align}
	h_r(x)&=\sum_{j\ge r}\frac{x^j}{j!} \quad \mbox{for all nonnegative integers $r$}; \label{h1}\\
	h_r(x)&=e^x \quad \mbox{for all negative integers $r$}.\label{h2}
\end{align}
Then, for all integers $t$,
\[
D'(\alpha)=\frac{\lambda h_{r-1}(\lambda \alpha)}{h_r(\lambda )},\  D''(\alpha)=\frac{\lambda ^2h_{r-2}(\lambda \alpha)}{h_r(\lambda )}, \ D'''(\alpha)=\frac{\lambda ^3h_{r-3}(\lambda \alpha)}{h_r(\lambda )}.
\]
Since $\ex \bfd=d$, it requires that $\lambda$ satisfies
\begin{equation}
	D'(1)=\frac{\lambda h_{r-1}(\lambda)}{h_r(\lambda)}=d.\label{lambda}
\end{equation}
Thus,
\[
\phi''(\alpha)=\frac{(k-1)d\alpha^{k-3}}{h_r(\lambda )}\Big((k-2)h_{r-2}(\lambda  t)-(k-1)(1-t)h_{r-3}(\lambda  t)\Big).
\]
Solving $\phi''(\alpha)=0$ yields
\begin{equation}
	\frac{k-1}{k-2}(1-t)= \frac{h_{r-2}(\lambda  t)}{h_{r-3}(\lambda  t)}=1-\frac{h_{r-3}(\lambda  t)-h_{r-2}(\lambda  t)}{h_{r-3}(\lambda t)}. \label{eq:truncate}
\end{equation}
The right hand side above is obviously a constant function if $r\le 2$. If $r\ge 3$, then $h_{r-3}(\lambda t)-h_{r-2}(\lambda t)=(\lambda  t)^{r-3}/(r-3)!$, and  $h_{r-3}(\lambda t)$ is a power series of $\lambda t$ with minimum degree $r-3$. Hence, by dividing $(\lambda t)^{r-3}/(r-3)!$ from both the numerator and the denominator, we immediately get that the right hand side of~\eqn{eq:truncate} is an increasing function. However the left hand side of~\eqn{eq:truncate} is a decreasing function. Hence~\eqn{eq:truncate} has at most one solution, implying that $\phi''(\alpha)$ has at most one root. \smallskip

\noindent {\em Case B: $\bfk\sim \po_{\ge s}(\gamma)$.} We must have $\gamma$ satisfy
\[
\frac{\gamma h_{s-1}(\gamma)}{h_s(\gamma)}=k,
\]
so that $\ex\bfk=k$. Here $k>s$ is required (to guarantee the existence of $\gamma$ if $s\ge 1$, and to avoid triviality if $s=0$). Now we have $K(\alpha)=h_s(\gamma\alpha)/h_s(\gamma)$, where $h_s$ is defined as in~\eqn{h1} and~\eqn{h2}.
Thus, 
\begin{align*}
	\phi(\alpha)&=1-\alpha-\frac{1}{d}D'\left(1-\frac{h_{s-1}(\gamma \alpha)}{h_{s-1}(\gamma)}\right)\\
	\phi'(\alpha)&=-1+\frac{\gamma h_{s-2}(\gamma\alpha)}{d h_{s-1}(\gamma)}D''\left(1-\frac{h_{s-1}(\gamma \alpha)}{h_{s-1}(\gamma)}\right)\\
	\phi''(\alpha)&=\frac{\gamma^2}{d h_{s-1}(\gamma)}\left(h_{s-3}(\gamma \alpha)D''\left(1-\frac{h_{s-1}(\gamma \alpha)}{h_{s-1}(\gamma)}\right) - \frac{h_{s-2}(\gamma \alpha)^2}{h_{s-1}(\gamma)} D'''\left(1-\frac{h_{s-1}(\gamma \alpha)}{h_{s-1}(\gamma)}\right) \right).
\end{align*}
Hence,
\begin{align}
	&& \phi(0)=0 && \phi(1)=-\frac{1}{d}D'(0)\le 0 && \label{f2}\\
	&& \phi'(0)=-1 && \phi'(1)=-1+\frac{\gamma h_{s-2}(\gamma)}{d h_{s-1}(\gamma)}D''(0).&& \label{f'2}
\end{align}

As before, we will prove that $\phi''(\alpha)$ has at most 1 root in $[0,1]$ (except for some trivial cases that will be discussed separately), which is sufficient to ensure (a) and (b).

\medskip

\noindent {\em Case B1: $\bfk\sim \po_{\ge s}(\gamma)$ and $\Var(\bfd)=0$.} In this case $\bfd=d$. Then $D(\alpha)=\alpha^d$. If $d\ge 3$ then
solving $\phi''(\alpha)=0$ yields
\begin{equation}
	\frac{d-2}{h_{s-1}(\gamma)} \cdot h_{s-2}(\gamma\alpha) = \left(1-\frac{h_{s-1}(\gamma \alpha)}{h_{s-1}(\gamma)}\right) \frac{h_{s-3}(\gamma\alpha)}{h_{s-2}(\gamma\alpha)}. \label{B1}
\end{equation}
On the right hand side above,
$1-h_{s-1}(\gamma \alpha)/h_{s-1}(\gamma)\ge 0$ and is a decreasing function of $\alpha$. 
We also have
\[
\frac{h_{s-3}(\gamma\alpha)}{h_{s-2}(\gamma\alpha)}=\frac{h_{s-3}(\gamma\alpha)}{h_{s-3}(\gamma\alpha)-(\gamma\alpha)^{s-3}/(s-3)!}=\left(1-\frac{(\gamma\alpha)^{s-3}/(s-3)!}{h_{s-3}(\gamma\alpha)}\right)^{-1},
\]
which is positive and a decreasing function of $\alpha$ if $s\ge 3$, and is equal to 1 if $s\le 2$. Hence, the left hand side of~\eqn{B1} is an increasing function whereas the right hand side is a decreasing function. Hence $\phi''(\alpha)$ has at most one root.

If $d\le 2$ the same argument as in Case A1 shows that claims (a) and (b) hold.
\smallskip

\noindent {\em Case B2: $\bfk\sim \po_{\ge s}(\gamma)$ and $\bfd\sim \po_{\ge r}(\lambda)$.} In this case $D(\alpha)=h_r(\lambda\alpha)/h_r(\lambda)$, and $\lambda$ necessarily satisfies~\eqn{lambda}.
Then solving $\phi''(\alpha)=0$ yields
\[
\frac{\lambda}{h_{s-1}(\gamma)} h_{s-2}(\gamma\alpha) = \frac{h_{s-3}(\gamma\alpha)}{h_{s-2}(\gamma\alpha)} \cdot  \frac{h_{r-2}(\lambda(1-h_{s-1}(\gamma \alpha)/h_{s-1}(\gamma)))}{h_{r-3}(\lambda(1-h_{s-1}(\gamma \alpha)/h_{s-1}(\gamma)))}
\]
The left hand side is an increasing function whereas the right hand side is the product of two functions, both of which are either equal to 1 or a positive decreasing function. Thus, $\phi''(\alpha)$ has at most one root.

\subsection*{Acknowledgment}
We thank David Saad for a helpful conversation and Guilhem Semerjian for bringing~\cite{Lelarge} to our attention. We thank Haodong Zhu for pointing out a small omission in Lemma \ref{Lemma_conc}.

\begin{appendix}

\section{Proof of \Lem~\ref{Lemma_sums}}\label{Sec_sums}

\noindent
Since $\Erw[\vec\lambda^r]<\infty$, the event $\cM=\cbc{\max_{i\in[s]}\vec\lambda_i\leq n/\ln^9n}$ has probability
\begin{align}\label{eqLemma_sums0}
\pr\brk{\cM}=1-o_n(1/n).
\end{align}
Moreover, fixing a small enough $\eta=\eta(\delta)>0$ and a large enough $L=L(\eta)>0$ and setting $Q_j=\sum_{i\in[s]}\vecone\{\lambda_i=j\}$, we obtain from the Chernoff bound that 
$\pr\brk{\forall j\leq L:|Q_j-s\pr\brk{\vec\lambda=j}|>\sqrt n\ln n}=o_n(1/n)$.
Hence, by Bayes' rule,
\begin{align}\label{eqLemma_sums1}
\pr\brk{\exists j\leq L:|Q_j-s\pr\brk{\vec\lambda=j}|>\sqrt n\ln n\mid\cM}&=o_n(1/n).
\end{align}
In addition, let $\cH=\cbc{h\in\NN:(1+\eta)^{h-1}L\leq n/\ln^9n}$ and for $h\in\cH$ let
\begin{align*}
R_h&=\sum_{j\geq1}Q_j\vecone\{L(1+\eta)^{h-1}<j\leq L(1+\eta)^{h}\wedge n/\ln^9n\},
\end{align*}
\begin{align*}
\bar R_h&=s\sum_{j\geq1}\pr\brk{\vec\lambda=j}\vecone\{L(1+\eta)^{h-1}<j\leq L(1+\eta)^{h}\wedge n/\ln^9n\}.
\end{align*}
Then the Chernoff bound and Bayes' rule yield
\begin{align}\label{eqLemma_sums2}
\pr\brk{\exists h\in\cH:\abs{R_h-\bar R_h}>\eta\bar R_h+\ln^2n\mid\cM}&=o_n(1/n).
\end{align}
Finally, given $\cM$ and $|Q_j-s\pr\brk{\vec\lambda=j}|\le\sqrt n\ln n$ for all $j\leq L$ and 
$\abs{R_h-\bar R_h}\leq\eta\bar R_h+\ln^2n$ for all $h\in\cH$, we obtain
\begin{align*}
\frac1s\sum_{i=1}^s\vec\lambda_i&\leq \sum_{j=1}jQ_j/s+\sum_{h\in\cH}(1+\eta)^hLR_h/s\\
	&=o_n(1)+\Erw\brk{\vec\lambda\vecone\{\vec\lambda\leq L\}}
	+\sum_{h\in\cH}(1+\eta)^{h+1}(\bar R_h+(\ln^2n))/s\leq \Erw\brk{\vec\lambda\vecone}+\delta/2+o_n(1).
\end{align*}
Similarly,  $\frac1s\sum_{i=1}^s\vec\lambda_i\geq \Erw\brk{\vec\lambda\vecone}-\delta/2+o_n(1)$.
Thus, the assertion follows from \eqref{eqLemma_sums0}--\eqref{eqLemma_sums2}


\section{Stochastic vs.\ linear independence}\label{Apx_0pinning}

\noindent
A precursor of \Prop~\ref{Prop_Alp} for finite field was obtained in~\cite[\Lem~3.1]{Ayre}.
Instead of dealing with linear independence, that statement dealt with stochastic dependencies.
Formally, given an $m\times n$-matrix $A$ over a finite field $\FF$, let $\mu_A$ be the probability distribution on $\FF^n$ defined by
\begin{align*}
\mu_A(\sigma)&=\vecone\cbc{\sigma\in\ker A}/|\ker A|.
\end{align*}
(This definition is nonsensical over infinite fields for the obvious reason that $|\ker A|\in\{1,\infty\}$.)
Let $\SIGMA=\SIGMA_A\in\FF^n$ denote a sample from $\mu_A$.
The stochastic independence statement reads as follows.

\begin{lemma}[{\cite[\Lem~3.1]{Ayre}}]\label{Lemma_stoch_pinning}
For any $\delta>0$, $\ell>0$ and for any finite field $\FF$ there exists $\cT=\cT(\delta,\ell,\FF)>0$ such that for any matrix $A$ over $\FF$ the following is true.
Choose $\THETA\in[\cT]$ uniformly at random.
Then with probability at least $1-\delta$ the matrix $A[\THETA]$ satisfies
\begin{align}\label{eqLemma_stoch_pinning}
\sum_{\substack{I\subset[n]:|I|=\ell}}\max_{\tau\in\FF^I}\abs{
		\mu_{A[\THETA]}\bc{\cbc{\forall i\in I:\SIGMA_{i}=\tau_i}}-\prod_{i\in I}\mu_{A[\THETA]}\bc{\cbc{\SIGMA_{i}=\tau_i}}}&<\delta n^\ell.
\end{align}
\end{lemma}

\noindent
In words, for most sets $I$ of $\ell$ coordinates the joint distribution of the coordinates $(\SIGMA_i)_{i\in I}$ is close to a product distribution in total variation distance.
Furthermore, the number $\THETA$ of rows that we add to $A$ is bounded in terms of $\eps,\ell$ only; i.e., $\THETA$ does not depend on the size $m\times n$ of $A$ or on the matrix $A$ itself.
\Lem~\ref{Lemma_stoch_pinning} and its proof are inspired by the `pinning lemma' from~\cite{CKPZ}.

The following lemma shows that how \Prop~\ref{Prop_Alp} implies \Lem~\ref{Lemma_stoch_pinning}; in a nutshell, the lemma states that linear independence is stronger than stochastic independence.

\begin{lemma}\label{Lemma_linstoch}
Let $A$ be an $m\times n$-matrix over a finite field $\FF$.
Unless $I\subset[n]$ is a proper relation of $A$ we have
\begin{align}\label{eqLemma_linstoch1}
\mu_A\bc{\cbc{\forall i\in I:\SIGMA_i=\tau_i}}&=\prod_{i\in I}\mu_A\bc{\cbc{\SIGMA_i=\tau_i}}&&\mbox{for all }\tau\in\FF^I.
\end{align}
\end{lemma}
\begin{proof}
Since for every $\tau\in\FF^I$ we have
\begin{align*}
\mu_A\bc{\cbc{\forall i\in I:\SIGMA_i=\tau_i}}&=\vecone\cbc{\forall i\in I\cap\fF(A):\tau_i=0}
	\mu_A\bc{\cbc{\forall i\in I\setminus\fF(A):\SIGMA_i=\tau_i}},\\
\prod_{i\in I}\mu_A\bc{\cbc{\SIGMA_i=\tau_i}}&=
	\vecone\cbc{\forall i\in I\cap\fF(A):\tau_i=0}\prod_{i\in I\setminus\fF(A)}\mu_A\bc{\cbc{\SIGMA_i=\tau_i}},
\end{align*}
we may assume that $I\cap\fF(A)=\emptyset$ by simply passing on to $I\setminus\fF(A)$ if necessary.
Hence, the task reduces to proving \eqref{eqLemma_linstoch1} under the assumption  that $I\subset[n]\setminus\fF(A)$ is no relation of $A$.

To prove this statement let $N=\nul(A)$ and suppose that $\xi_1,\ldots,\xi_N\in\FF^n$ form a basis of $\ker A$.
Let $\Xi\in\FF^{n\times N}$ be the matrix with columns $\xi_1,\ldots,\xi_N$ and let $\Xi_1,\ldots,\Xi_N$ signify the rows of $\Xi$.
The homomorphism $z\in\FF^N\to\ker A$, $z\mapsto\Xi z$ maps the uniform distribution on $\FF^N$ to the uniform distribution $\mu_A$ on $\ker A$.
Therefore, to prove \eqref{eqLemma_linstoch1} it suffices to prove that the projection of this homomorphism to the $I$-rows, i.e., the map $z\in\FF^N\mapsto(\Xi_i z)_{i\in I}$ is surjective.
Equivalently, we need to show that
\begin{align}\label{eqLemma_linstoch2}
\rk(\Xi_i)_{i\in I}&=|I|.
\end{align}

Assume for contradiction that \eqref{eqLemma_linstoch2} is violated.
Then there exists a vector $z\in\FF^I\setminus\cbc 0$ such that $\sum_{i\in I}z_i\Xi_i=0$.
This implies that for all $x\in \FF^n$,
\begin{align*}
Ax=0&\quad\Rightarrow\quad\sum_{i\in I}z_ix_i=0.
\end{align*}
As a consequence, there exists a row vector $y$ of length $m$ such that $(yA)_j=\vecone\cbc{i\in I}z_i$ for all $j\in[n]$.
Hence, $\emptyset\neq\supp(yA)\subset I$.
Thus $I$ is a relation of $A$, in contradiction to our assumption that it is not.
\end{proof}

\noindent
Thus, \Lem~\ref{Lemma_stoch_pinning} is an immediate consequence of \Prop~\ref{Prop_Alp} and \Lem~\ref{Lemma_linstoch}.
Indeed, the proof of \Prop~\ref{Prop_Alp} renders the explicit bound $\cT=\lceil 4\ell^3/\delta^4\rceil+1$ on the number of coordinates that need to get pegged.
By comparison, the stochastic approach via the arguments from~\cite{Ayre,Victor} leads to a value of $\cT$ that is exponential in $\ell$ (although it may be possible to improve this estimate via probabilistic arguments).

\section{A self-contained proof of the upper bound on the rank}\label{Sec_interp}

\noindent
The `$\leq$'-inequality in \eqref{eqthm:rank} was previously proved by Lelarge~\cite{Lelarge}, who derived the bound from the Leibniz determinant formula and the formula for the matching number of random bipartite graphs from~\cite{BLS}.
The proof of that formula, however, is far from straightforward.
Therefore, as a point of interest in this section we show that another idea from mathematical physics, the interpolation method from spin glass theory~\cite{FranzLeone,Guerra}, can be harnessed to obtain a self-contained proof of the upper bound on the rank.
The proof uses similar ideas as the proof of the lower bound outlined in \Sec~\ref{Sec_overview}.
Thus, phrased in terms of the nullity, the aim in this section is to show that \whp
	\begin{align}\label{eqthm:rank_upper}
	\nul(\A)/n&\geq \max_{\alpha\in[0,1]}\Phi(\alpha)+o_n(1).
	\end{align}

\subsection{The interpolation method}
The basic idea behind the interpolation method is to construct a family of random matrices $\vA_{\eps}(t)$ parametrised by `time' $t$.
At $t=\vm_{\eps,n}$ we obtain precisely the matrix $\vA_{\eps,n}$.
At the other extreme, $\vA_\eps(0)$ is a block diagonal matrix whose nullity can be read off easily.
To establish the lower bound we will control the change of the nullity with respect to $t$.
By comparison to applications of the interpolation method to other combinatorial problems (e.g., \cite{bayati,CKPZ,FranzLeone,Panchenko}), the construction here is relatively elegant.
In particular, throughout the interpolation we will be dealing with an actual random matrix, rather than some other, more contrived object.

\begin{figure}
\includegraphics[height=4.5cm]{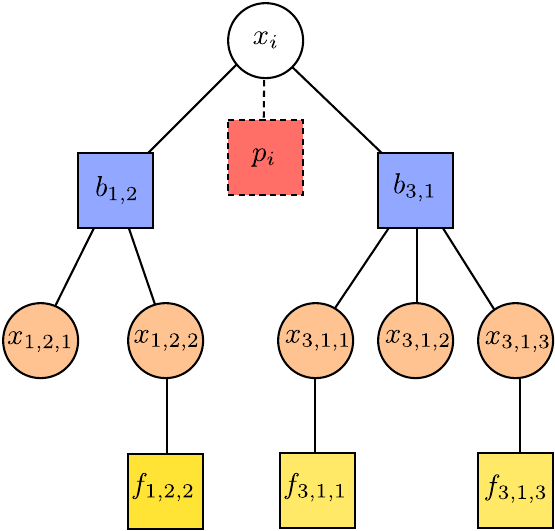}
\hfill\includegraphics[height=4.5cm]{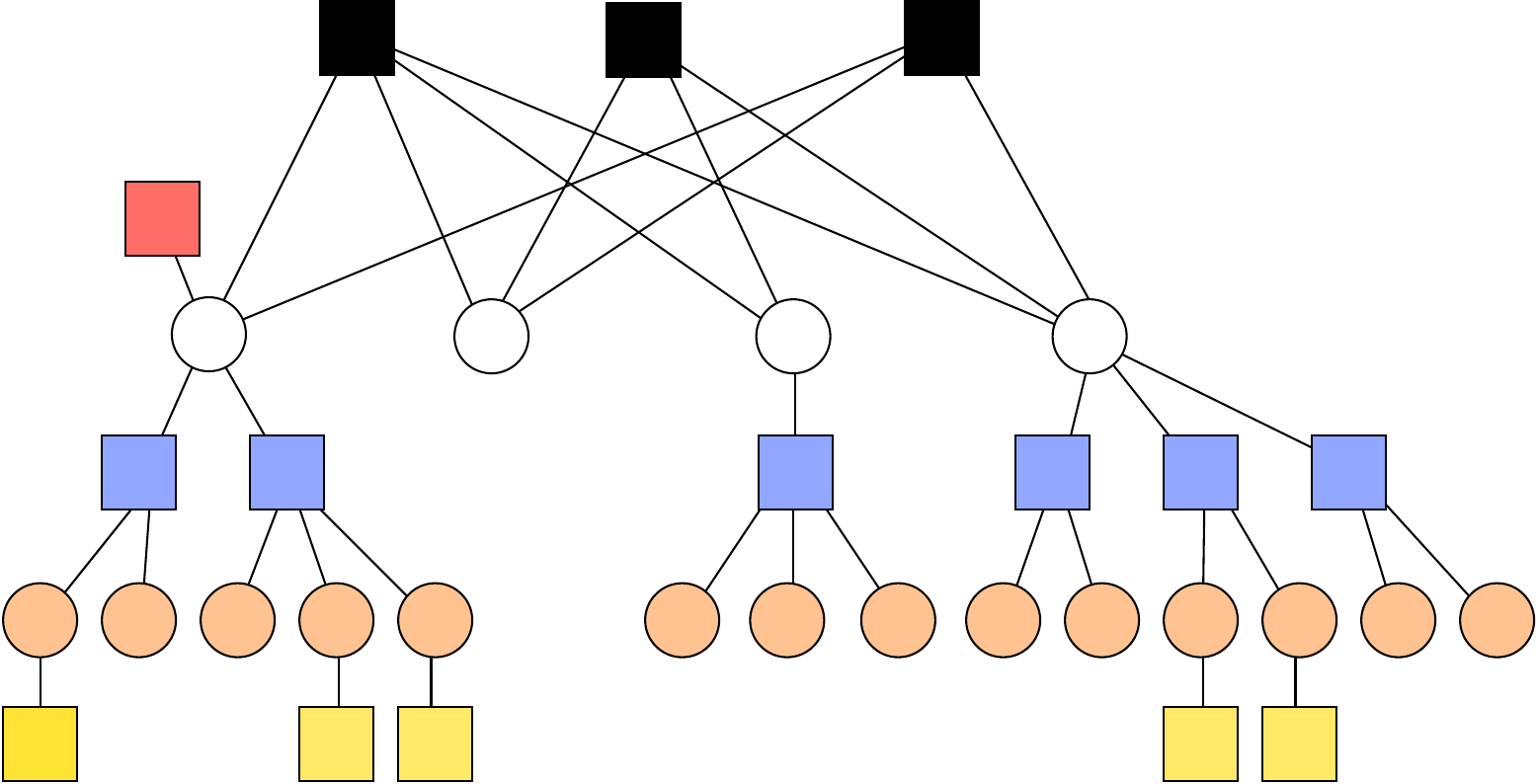}
\caption{Left: sketch of the component of $x_i$ at $t=0$; the check $p_i$ is present iff $i\leq\THETA$.
Right: sketch of the factor graph $\G_\eps(t)$ for $0<t<\vm_{\eps,n}$, with the $a_{i,j}$ coloured black and the other colours as in the left figure.}\label{Fig_interpolation}
\end{figure}

Getting down to the details, apart from $t$ and $\eps$ we need two further parameters: an integer $\cT=\cT(\eps)\geq0$ and a real $\beta\in[0,1]$, which, in order to obtain the optimal bound, we choose such that
\begin{equation}\label{eqbeta}
\Phi(\beta)=\max_{\alpha\in[0,1]}\Phi(\alpha).
\end{equation}	
Further, let $\vm_{\eps,n}\disteq \Po((1-\eps)dn/k)$.
Also let $(\vk_i,\vk_i',\vk_i'')_{i\geq1}$ and $(\vd_i)_{i\geq1}$ be copies of $\vk$ and $\vd$, respectively, mutually independent and independent of $\vm_{\eps,n}$.
Additionally, choose $\THETA\in[\cT]$ uniformly and independently of everything else.
Finally, recall that $(\row_i,\col_i)_{i\geq1}$ are uniformly distributed on the unit interval and independent of all other randomness.

The Tanner graph $\G_\eps(t)$ has variable nodes 
\begin{align*}
x_1,\ldots,x_n&&\mbox{and}&&(x_{i,j,h})_{i\in[\vm_{\eps,n}-t],\,j\in[\vk_i'],\,h\in[\vk_i'-1]}.
\end{align*}
Moreover, let $\cF_t$ be a random set that contains each of the variable nodes $x_{i,j,h}$ with probability $\beta$ independently.
Then the check nodes are 
\begin{align*}
a_1,\ldots,a_{t},&&(b_{i,j})_{i\in [\vm_{\eps,n}-t],\, j\in[\vk_i']},&&p_1,\ldots,p_{\THETA},&&
f_{i,j,h}\quad\mbox{ for each }x_{i,j,h}\in\cF_t.
\end{align*}
To define the edges of the Tanner graph let $\vec\Gamma_\eps(t)$ be a random maximal matching of the complete bipartite graph with vertex sets
\begin{align*}
\bigcup_{i=1}^n\cbc{x_i}\times[\vd_i],\qquad\bc{\bigcup_{i=1}^{t}\cbc{a_i}\times[\vk_i]}
		\cup\cbc{b_{i,j}:i\in[\vm_{\eps,n}-t],\ j\in[\vk_i']}.
\end{align*}
For each matching edge $\{(x_i,s),(a_j,t)\}\in\vec\Gamma_\eps(t)$ insert an edge between $x_i$ and $a_j$ into the Tanner graph and for each $\{(x_i,s),b_{j,h}\}\in\vec\Gamma_\eps(t)$ insert an edge between $x_i$ and $b_{j,h}$.
Thus, $\G_\eps(t)$ may contain multi-edges.
Further, add an edge between $x_i$ and $p_i$ for $i=1,\ldots,\THETA$ and
add an edge between $x_{i,j,h}$ and $b_{i,j}$ for each $h\in[\vk_i'-1]$ as well as an edge between every $x_{i,j,h}\in\cF_t$ and the check $f_{i,j,h}$.
Finally, let $\vA_{\eps}(t)$ be the random matrix induced by $\G_\eps(t)$.
Formally, with the rows indexed by the check nodes and the columns indexed by the variable nodes, we let
\begin{align*}
(\vA_{\eps}(t))_{p_i,x_j}&=\vecone\cbc{i=j}&&(i\in[\THETA],j\in[n]),\\
(\vA_{\eps}(t))_{a_i,x_j}&=\chi_{\row_i,\col_j}\sum_{u=1}^{\vk_i}\sum_{v=1}^{\vd_j}
		\vecone\cbc{\{(x_j,v),(a_i,u)\in\vec\Gamma_\eps(t)\}}
		&&(i\in[t],j\in[n]),\\
(\vA_{\eps}(t))_{b_{h,i},x_j}&=\vecone\cbc{x_j\in\partial_{\G_\eps(t)}b_{h,i}}
	&&(h\in[\vm_{\eps,n}-t],j\in[n]),\\
(\vA_{\eps}(t))_{b_{h,i},x_{u,v,w}}&=\vecone\cbc{h=u,\,i=v}&&
	(h,u\in[\vm_{\eps,n}-t],i\in[\vk_h'],v\in[\vk_u'],\\
	&&&w\in[\vk_u'-1]), \\
(\vA_{\eps}(t))_{f_{h,i,j},x_{u,v,w}}&=\vecone\cbc{(h,i,j)=(u,v,w)}&&
(h,u\in[\vm_{\eps,n}-t],i\in[\vk_h'],j\in[vk_h'-1],\\&&&\qquad v\in[\vk_u'],w\in[\vk_u'-1]).
\end{align*}
All other entries of $\vA_\eps(t)$ are equal to zero.

The semantics is as follows.
The checks $a_i$ will play exactly the same role as before, i.e., each is adjacent to $\vk_i$ of the variable nodes $x_1,\ldots,x_n$ \whp\
By contrast, each $b_{i,j}$ is adjacent to precisely one of the variables $x_1,\ldots,x_n$.
In addition, $b_{i,j}$ is adjacent to the $\vk_i'-1$ variable nodes $x_{i,j,h}$, $h\in[\vk_i'-1]$.
These variable nodes, in turn, are adjacent only to $b_{i,j}$ and to $f_{i,j,h}$ if $x_{i,j,h}\in\cF$.
The checks $f_{i,j,h}$ are unary, i.e., $f_{i,j,h}$ simply forces $x_{i,j,h}$ to take the value zero.
Finally, each of the checks $p_i$ is adjacent to $x_i$ only, i.e., $p_1,\ldots,p_{\THETA}$ just freeze $x_1,\ldots,x_{\THETA}$.

For $t=1$ the Tanner graph contains $\vm_{\eps,n}\disteq\Po((1-\eps)dn/k)$ `real' checks $a_i$ and none of the checks $b_{i,j}$ or $f_{i,j,h}$.
In effect, $\vA_\eps(1)$ is distributed precisely as $\vA_\eps$ from \Sec~\ref{Sec_outline1}.
By contrast, at $t=0$ there are no checks $a_i$ involving several of the variables $x_1,\ldots,x_n$.
As a consequence, the Tanner graph decomposes into $n$ connected components, one for each of the $x_i$.
In fact, each component is a tree comprising $x_i$, some of the checks $b_{j,h}$ and their proprietary variables $x_{j,h,s}$ along with possibly a check $f_{j,h,s}$ that freezes $x_{j,h,s}$ to zero.
{For $i\in[\THETA]$ there is a check $p_i$ freezing $x_i$ to zero as well.}
Thus, $\vA_\eps(0)$ is a block diagonal matrix consisting of $n$ blocks, one for each component.
In effect, the rank of $\vA_\eps(0)$ will be easy to compute.
Finally, for $0<t<1$ we have a blend of the two extremal cases.
There will be some checks $a_i$ and some $b_{i,j}$ with their retainer variables and checks;  see Figure~\ref{Fig_interpolation}.

We are going to trace the nullity of $\vA_\eps(t)$ as $t$ increases.
But since the newly introduced variables $x_{i,j,h}$ inflate the nullity, we subtract the `obvious' correction term to retain the same scale throughout the process.
In addition, we need a correction term to make up for the greater total number of check nodes in $\vA_\eps(0)$ by comparison to $\vA_\eps(\vm_{\eps,n})$.
Thus, let
\begin{align*}
\cN_t&=\nul\vA_{\eps}(t)+|\cF_t|-\sum_{i=1}^{\vm_{\eps,n}-t}\vk_i'(\vk_i'-1),&
\cY_t&=\sum_{i=1}^{\vm_{\eps,n}}(\vk_i-1)(\beta^{\vk_i}-1).
\end{align*}
The following two statements summarise the interpolation argument.
First, we compute $\Erw[\cN_0]$.

\begin{proposition}\label{Lemma_interpol2}
For any fixed $\theta\geq0$ we have
	$n^{-1}\Erw[\cN_0]=D(1-K'(\beta)/k)+dK'(\beta)/k-d+o_{\eps,n}(1).$
\end{proposition}

\noindent
The next proposition provides monotonicity.

\begin{proposition}\label{Lemma_interpolation}
For any $\eps>0$ there exists $\cT=\cT(\eps)>0$ such that with probability $1-o_n(1/n)$ uniformly for all $0\leq t<\vm_{\eps,n}$ we have
$\Erw[\cN_{t+1}+\cY_{t+1}\mid\vm_{\eps,n}]\geq\Erw[\cN_{t}+\cY_{t}\mid\vm_{\eps,n}]+o_{\eps,n}(1).$
\end{proposition}

\noindent
As an immediate consequence of \Prop s~\ref{Lemma_interpol2} and~\ref{Lemma_interpolation} we obtain the desired lower bound on the nullity.

\begin{corollary}\label{Prop_upper}
We have
	$\frac1n\Erw[\nul(\vA_\eps)]\geq\max_{\alpha\in[0,1]}\Phi(\alpha)+o_{\eps,n}(1).$
\end{corollary}
\begin{proof}
\Prop~\ref{Lemma_interpolation} implies that
\begin{align}\nonumber
\Erw[\nul \vA_{\eps,n}]&=\Erw[\nul \vA_\eps(\vm_{\eps,n})]=\Erw[\cN_{\vm_{\eps,n}}]=\Erw[\cN_{\vm_{\eps,n}}+\cY_{\vm_{\eps,n}}]-\Erw[\cY_{\vm_{\eps,n}}]\\
		&\geq \Erw[\cN_0+\cY_0]-\Erw[\cY_{\vm_{\eps,n}}]+o_\eps(n)=\Erw[\cN_0]-\Erw[\cY_{\vm_{\eps,n}}]+o_{\eps,n}(n).\label{eqProp_upper_1}
\end{align}
Further, by \Prop~\ref{Lemma_interpol2},
\begin{align*}
\frac1n\Erw[\cN_0]&=-d+dK'(\beta)/k+D(1-K'(\beta)/k)+o_{\eps,n}(1),\\
\frac1n\Erw[\cY_{\vm_{\eps,n}}]&=\frac dk\bc{\beta K'(\beta)-k+1-K(\beta)}+o_{\eps,n}(1).
\end{align*}
Hence, \eqref{eqbeta} yields
$$n^{-1}(\Erw[\cN_0]-\Erw[\cY_{\vm_{\eps,n}}])=\Phi(\beta)+o_\eps(1)=\max_{\alpha\in[0,1]}\Phi(\alpha)+o_{\eps,n}(1)$$
and the assertion follows from \eqref{eqProp_upper_1}.
\end{proof}

\noindent
Combining \Prop~\ref{Cor_lower}, \Prop~\ref{Prop_lower} and \Cor~\ref{Prop_upper} and the standard concentration for $\nul{\vA_{\eps}}$ from Lemma~\ref{Lemma_conc} completes the proof of \eqref{eqthm:rank_upper}.
We proceed to prove \Prop s~\ref{Lemma_interpol2} and \ref{Lemma_interpolation}.

\subsection{Proof of \Prop~\ref{Lemma_interpol2}}
Each component of $\G_\eps(0)$ contains precisely one of the variable nodes $x_1,\ldots,x_n$.
In effect, $\vA_{\eps}(0)$ has a block diagonal structure, and the overall nullity is nothing but the sum of the nullities of the blocks.
It therefore suffices to calculate the nullity of the block $\vec B_s$ representing the connected component of $x_s$.
%
Indeed, because $\sum_{s=1}^n\abs{\partial^2x_s}=\sum_{i\leq\vm'_\eps(0)}\vk_i'(\vk_i'-1)$ and 
$\sum_{s=1}^n\abs{\partial^2x_s\cap\cF_0}=|\cF_0|$ we have
\begin{align*}
\cN_0&=\sum_{s=1}^n\vN_s,&\mbox{where }\quad\vN_s=\nul(\vec B_s)-\abs{\partial^2x_s}+\abs{\partial^2x_s\cap\cF_0}.
\end{align*}
Consequently, since $\THETA=O_{n}(1)$ it suffices to prove that
\begin{align}\label{eqLemma_interpol2_1}
\Erw\brk{\vN_s}=\begin{cases}dK'(\beta)/k+D(1-K'(\beta)/k)-d+o_\eps(1)&\mbox{ if }s>\THETA,\\	
O_n(1)&\mbox{ otherwise}.		
\end{cases}
\end{align}
In fact, the second case in \eqref{eqLemma_interpol2_1} simply follows from  $\vN_s\leq\vd_s$ and $\Erw[\vd_s]=O_n(1)$ for all $s$.

Hence, suppose that $s>\THETA$.
As $|\vN_s|\leq\vd_s$ and $\Erw[\vd_s^r]=O_{\eps,n}(1)$ for an $r>2$ we find $\xi>0$ such that
\begin{align}\label{eqLemma_interpol2_2}
\Erw[|\vN_s|\vecone\{\vd_s>\eps^{\xi-1/2}\}]&=o_{\eps,n}(1).
\end{align}
Moreover, let
$\Xi=\sum_{i=1}^{\vm_\eps'(0)}\vk_i'\vecone\cbc{\vk_i'>\eps^{-8}}$, $\vec M_j'=\sum_{i=1}^{\vm_\eps'(0)}\vecone\{\vk_i'=j\}.$
Because $\Erw[\vk^2]=O_{\eps,n}(1)$ we have
\begin{align}\label{eqLemma_interpol2_3}
\Erw\brk{\Xi}&\leq\frac{dn}k\Erw\brk{\vk\vecone\{\vk\geq\eps^{-8}\}}=nO_{\eps,n}(\eps^8),
\end{align}	
while $\vM_j'\sim(1-\eps)dn\pr\brk{\vk=j}/k$ for all $j\leq\eps^{-8}$ \whp\ by Chebyshev's inequality.
Hence, introducing the event
\begin{align*}
\cE_s&=\cbc{\vd_s\leq\eps^{\xi-1/2},\ \Xi\leq n\eps^6,\ 
\forall j\leq\eps^{-8}:\vM_j'\sim(1-\eps)dn\pr\brk{\vk=j}/k,\,
\sum_{i=1}^n\vd_i\sim dn,\,
\sum_{i\geq3}i\vM_i'\sim(1-\eps)dn},
\end{align*}
we obtain from \eqref{eqLemma_interpol2_2} and \eqref{eqLemma_interpol2_3} that
\begin{align}\label{eqLemma_interpol2_4}
\Erw[\vN_s]&=\Erw\brk{\vN_s\vecone\cE_s}+o_{\eps,n}(1).
\end{align}	

With $\vec\gamma\leq\vec d_s$ the actual degree of $x_s$ in $\G_\eps(s)$,
let $\vec\kappa_1,\ldots,\vec\kappa_{\vec\gamma}$ be the degrees of the checks adjacent to $x_s$.
We claim that given $\cE_s$ and $\vd_s$, 
\begin{align}\label{eqLemma_interpol2_5}
\dTV((\vec\kappa_1,\ldots,\vec\kappa_{\vec\gamma}),(\hat\vk_1,\ldots,\hat\vk_{\vd_s}))&=o_{\eps,n}(\eps^{1/2}).
\end{align}
Indeed, on $\cE_s$ the probability that $x_s$ is adjacent to a check of degree greater than $\eps^{-8}$ is 
\newline
$O_{\eps,n}(\vd_s\Xi/\sum_{j\geq3}j\vM_j')=o_{\eps,n}(\eps)$.
Further, given $\cE_s$ we have $$\sum_{j\geq3}j\vM_j'\geq(1-2\eps)dn,$$ and thus $\pr[\vec\gamma<\vd_s\mid\cE_s]=o_{\eps,n}(\eps^{1/2})$.
Moreover, given $\vec\gamma=\vd_s$, for each $i\in[\vd_s]$ the probability that the $i$-th clone of $x_s$ gets matched to a check of degree $j\leq\eps^{-8}$ is $$j\vM_j'/\sum_{h\geq3}h\vM_h'=j\pr\brk{\vk=j}/k+o_n(1)=\pr\brk{\hat\vk=j}+o_n(1).$$
These events are asymptotically independent for the different clones.
Thus, we obtain \eqref{eqLemma_interpol2_5}.

Finally, we can easily compute $\vN_s$ given the vector $(\vec\kappa_1,\ldots,\vec\kappa_\GAMMA)$.
The matrix $\vB_s$ has fairly simple structure.
The first $\GAMMA$ rows have a non-zero entry in the first column representing $x_s$.
Additionally, for $i=1,\ldots,\GAMMA$ the $i$th row contains $\vec\kappa_i-1$ further non-zero entries, and the columns where theses non-zero entries occur are disjoint for all $i$.
Finally, at the bottom of the matrix there is a block freezing the variables in $\cF_0\cap\partial^2x_s$ to zero.
We therefore claim that the rank of the matrix works out to be
\begin{align}\label{eqNeu1}
\Erw[\rk(\vB_s)\mid \vec\kappa_1,\ldots,\vec\kappa_\GAMMA]
	&=\sum_{i=1}^{\GAMMA}(1-\beta^{\vec\kappa_i-1})+|\cF_0\cap\partial^2x_s|+1-\prod_{i=1}^{\GAMMA}(1-\beta^{\vec\kappa_i-1}).
\end{align}
To see this, let us first compute the rank of the matrix $\vB_s'$ without the first column.
Then row $i\in[\GAMMA]$ contributes to the rank unless all the variables in the corresponding equation other than $x_s$ belong to $\cF_0$, an event that occurs with probability $\beta^{\vec\kappa_i-1}$; hence the first summand.
In addition, the $|\cF_0\cap\partial^2x_s|$ rows pegging variables to zero contribute to the rank (second summand).
Furthermore, going back to $\vB_s$, the first column adds to the rank unless non of the first $\GAMMA$ rows of $\vB_s'$ gets zeroed out completely, an event that has probability $\prod_{i=1}^{\GAMMA}(1-\beta^{\vec\kappa_i-1})$.
Since
\begin{align*}
\Erw[\vN_s\mid \vec\kappa_1,\ldots,\vec\kappa_\GAMMA]&=
	1+\sum_{i=1}^{\GAMMA}(\vec\kappa_i-1)-\Erw[\rk(\vB_s)\mid \vec\kappa_1,\ldots,\vec\kappa_\GAMMA]
		-\Erw\brk{\abs{\partial^2x_s}-\abs{\partial^2x_s\cap\cF_0}\mid \vec\kappa_1,\ldots,\vec\kappa_\GAMMA}\\
	&=1-\Erw[\rk(\vB_s)\mid \vec\kappa_1,\ldots,\vec\kappa_\GAMMA]+
	\Erw\brk{\abs{\partial^2x_s\cap\cF_0}\mid \vec\kappa_1,\ldots,\vec\kappa_\GAMMA},
\end{align*}
substituting \eqref{eqNeu1} in yields
\begin{align}\label{eqLemma_interpol2_6}
\Erw[\vN_s\mid \vec\kappa_1,\ldots,\vec\kappa_\GAMMA]
	&=\prod_{i=1}^{\GAMMA}(1-\beta^{\vec\kappa_i-1})-\sum_{i=1}^{\GAMMA}(1-\beta^{\vec\kappa_i-1}).
\end{align}
Combining \eqref{eqLemma_interpol2_4}, \eqref{eqLemma_interpol2_5} and \eqref{eqLemma_interpol2_6} completes the proof.

\subsection{Proof of \Prop~\ref{Lemma_interpolation}}
To couple the random variables $\cN_{t+1}$ and $\cN_t$ we need to investigate short linear relations among the cavities, i.e., the clones from $\bigcup_{i=1}^n\cbc{x_i}\times[\vd_i]$ that are not incident to an edge of $\vec\Gamma_\eps(t)$.
Denote this set by $\cC(t)$.
Further, let $P_t$ be the distribution on the set of variables induced by drawing a random cavity, i.e.,
\begin{align*}
P_t(x_i)=|\cC(t)\cap(\{x_i\}\times[\vd_i])|/|\cC(t)|,
\end{align*}
 and let $\vy_{1},\vy_{2}\ldots$ be independent samples from $P_t$.

\begin{lemma}\label{Lemma_pinint}
For any $\delta>0$ and $\ell>0$ there is $\cT=\cT(\delta,\ell)>0$ such that 
\begin{align*}
\pr\brk{\vy_{1},\ldots,\vy_{\ell}\mbox{ form a proper relation}}&<\delta.
\end{align*}
\end{lemma}
\begin{proof}
The choice of $\vm_{\eps,n}$ guarantees that $|\cC(t)|\geq\eps n/2$ \whp\
Moreover, since $\Erw[\vd]=O_{\eps,n}(1)$ we find $L=L(\eps,\delta)>0$ such that
the event  $\cL=\cbc{\sum_{i=1}^n\vd_i\vecone\{\vd_i>L\}<\eps\delta^2 n/16}$ has probability $\pr[\cL]\geq1-\delta/8$.
Therefore, we may condition on $\cE=\cL\cap\{|\cC(t)|\geq\eps n/2\}$.

Let $\vx_1,\ldots,\vx_\ell$ be variables drawn uniformly with replacement from $V_n=\{x_1,\ldots,x_n\}$.
Then on the event $\cE$ we have, for any $\ell$-tuple $y_1,\ldots,y_\ell$ of variables,
\begin{align*}
\pr\brk{\vy_1=y_1,\ldots,\vy_\ell=y_\ell\mid\vA_\eps(t)}&\leq\pr\brk{\vx_1=y_1,\ldots,\vx_\ell=y_\ell\mid\vA_\eps(t)}(2L/\eps)^\ell
	+\delta^2.
\end{align*}
Consequently, because the distribution of $\G_\eps(t)-\{p_1,\ldots,p_{\vec\theta}\}$ is invariant under permutations of $x_1,\ldots,x_n$, Remark~\ref{Prop_Alp++} shows that  $\pr\brk{\vx_1=y_1,\ldots,\vx_\ell=y_\ell\mid\vA_\eps(t)}<\delta(\eps/(2L))^\ell/2$, provided that $\cT=\cT(\delta,\ell)$ is large enough.
\end{proof}

We proceed to derive \Prop~\ref{Lemma_interpolation} from \Lem~\ref{Lemma_pinint} and a coupling argument.
Let $\G_{\eps}'(t)$ be the Tanner graph obtained from $\G_{\eps}(t+1)$ by removing the check $a_{t+1}$, let $\A_{\eps}'(t)$ be the corresponding matrix and let
\begin{align*}
\cN_t'&=\nul\vA_{\eps}'(t)+|\cF_{t+1}|-\sum_{i=1}^{\vm_{\eps,n}-t-1}\vk_i'(\vk_i'-1).
\end{align*}
Then clearly
\begin{align*}
\Erw\brk{\cN_{t+1}-\cN_t\mid\vm_{\eps,n}}=
	\Erw\brk{\cN_{t+1}-\cN_t'\mid\vm_{\eps,n}}-
	\Erw\brk{\cN_{t}-\cN_t'\mid\vm_{\eps,n}}.
\end{align*}

Let $\ALPHA\in[0,1]$ be the fraction of frozen cavities in $\G_\eps'(t)$, with the convention that $\ALPHA=0$ if the set $\cC'(t)$ of these cavities is empty.

\begin{lemma}\label{Claim_Lemma_interpolation_2}
We have
 $\Erw\abs{\Erw[\cN_{t+1}-\cN_t'\mid\vA_\eps(t)',\vm_{\eps,n}]-(K(\ALPHA)-1)}=o_{\eps,n}(1)$.
\end{lemma}
\begin{proof}
The random matrix $\vA_\eps(t+1)$ is obtained from $\vA_\eps'(t)$ by inserting a new random check $a_{t+1}$.
Pick $\zeta=\zeta(\eps)>0$ small enough and $\delta=\delta(\zeta)>0$ smaller still.
Since $|\nul(\vA_\eps'(t))-\nul(\vA_\eps(t+1))|\leq1$ and $\Erw[\vk^2]=O_{\eps,n}(1)$ we may condition on the event that $\vk_{t+1}\leq\eps^{-1}$.
Similarly, \Lem~\ref{Lemma_pinint} shows that we may assume that the set $\cX$ of variables of $\G_\eps'(t)$ where the new check node $a_{t+1}$ attaches does not form a proper relation, provided that $\cT=\cT(\eps)$ is chosen sufficiently large.
Therefore, \Lem~\ref{Cor_free} yields
\begin{align*}
\Erw[\cN_{t+1}-\cN_t\mid\vA_\eps(t)',\vm_{\eps,n}]&=
\Erw[\nul(\vA_\eps(t+1))-\nul(\vA_\eps'(t))\mid\vA_\eps'(t),\vm_{\eps,n}]\\
&=
\Erw\brk{\ALPHA^{\vk_{t+1}}-1\mid\vA_\eps'(t),\vm_{\eps,n}}+o_\eps(1)=K(\ALPHA)-1+o_{\eps,n}(1),
\end{align*}
as claimed.
\end{proof}

\begin{lemma}\label{Claim_Lemma_interpolation_4}
Let $Q(\alpha,\beta)=\Erw\brk{\vk(\alpha \beta^{\vk-1}-1)}$ for $\alpha\in[0,1]$.
Then $\Erw\abs{\Erw\brk{\cN_t-\cN_t'\mid\vA'_\eps(t),\vm_{\eps,n}}-Q(\ALPHA,\beta)}=o_{\eps,n}(1).$
\end{lemma}
\begin{proof}
The factor graph $\G_t(\eps)$ is obtained from $\G_t'(\eps)$ by adding the
checks $b_{\vm_{\eps,n}-t-1,h}$ for $h\in[\vk'_{\vm_\eps-t-1}]$, the corresponding variables
$x_{\vm_{\eps,n}-t-1,h,j}$ and possibly their respective checks $f_{\vm_{\eps,n}-t-1,h,j}$.
Since by construction
 $$|\cN_t-\cN_t'|\leq\vk'_{\vm_\eps-t-1}$$
and $\Erw[\vk^2]=O_{\eps,n}(1)$ we may condition on the event that $\vk'_{\vm_\eps-t-1}\leq\eps^{-1}$.
In effect, \Lem~\ref{Lemma_pinint} shows that we may assume the set $\cX$ of cavities adjacent to the new checks $b_{\vm_{\eps,n}-t-1,h}$ does not form a proper relation, provided that $\cT=\cT(\eps)$ is chosen large enough.
Moreover, the number of frozen cavities in $\cX$ is within $o_n(1)$ of a binomial distribution $\Bin(\vk_{\vm_{\eps,n}-t-1}',\ALPHA)$ in total variation.
Therefore, \Lem~\ref{Cor_free} shows that 
\begin{align*}
\Erw\brk{\cN_{t}-\cN_t'\mid\vA_{\eps,n}'(t),\vm_{\eps,n}}&=
\Erw[\nul(\vA_{\eps}(t))\negmedspace-\negmedspace\nul(\vA_\eps'(t))\negmedspace-\negmedspace\vk'_{\vm_{\eps,n}-t-1}(\vk'_{\vm_{\eps,n}-t-1}\negmedspace-\negmedspace1)\negmedspace+\negmedspace|\cF'|\mid\vA_\eps'(t),\vm_{\eps,n}]\\&=
	Q(\ALPHA,\beta)+o_{\eps,n}(1),
\end{align*}
as claimed.
\end{proof}

\begin{lemma}\label{Claim_Lemma_interpolation_5}
We have
$\Erw[\cY_{t+1}-\cY_t]=\Erw[(\vk-1)(\beta^{\vk}-1)]$.
\end{lemma}
\begin{proof}
This is the result of a straightforward calculation.
\end{proof}

\begin{proof}[Proof of \Prop~\ref{Lemma_interpolation}]
Combining \Lem s~\ref{Claim_Lemma_interpolation_2}--\ref{Claim_Lemma_interpolation_5}, we obtain
\begin{align}\label{eqLemma_interpolation_1}
\Erw[\cN_{t+1}+\cY_{t+1}]-\Erw[\cN_{t}+\cY_{t}]&=
	\Erw\brk{\ALPHA^{\vk}-1-\vk(\ALPHA\beta^{\vk-1}-1)+(\vk-1)(\beta^{\vk}-1)}+o_{\eps,n}(1).
\end{align}
Since $x^k-kxy^{k-1}+(k-1)y^k\geq0$ for all $k\geq1$, $x,y\in[0,1]$, the assertion follows from~\eqref{eqLemma_interpolation_1}.
\end{proof}

\section{Verification of $(\vm-\vm')/n$ and $\rank(\A')/n$}\label{apx_maurice}

Let $\vm_j$ denote the number of rows with exactly $j$ nonzero entries. With standard concentration arguments, we know that \whp\ $\vm_0\sim \vm \pr(\vk=0) \sim (dn/k) K(0)$, and $\vm_1\sim \vm \pr(\vk=1) \sim (dn/k) K'(0) $. Consequently, \whp\ $(\vm-\vm')/n \sim 
	d(1-K(0)-K'(0))/k$.

For $\rank(\A')$, let $X_i$ be the indicator variable that there exists a row with exactly one nonzero entry, and that nonzero entry occurs at the $i$-th column. Then $\rank(\A')=\sum_{i=1}^n X_i$.  Conditioning on $\vm_1$ and $\vD=\sum_{j} j\vm_j$, we know that $\ex X_i = \sum_j \pr(\vd=j)(1- (\vm_1/\vD)^j)$ for every $i$. Since \whp\ $\vm_1/\vD\sim K'(0)/k$, the standard concentration results immediately yield that \whp\ $\rank(\A')/n\sim 1- \sum_j \pr(\vd=j) (K'(0)/k)^j = 1-D(1-K'(0)/k)$.

\end{appendix}


\end{document}